\documentclass[12pt]{paper}

\usepackage{color}
\usepackage{mathptmx}       
\usepackage{helvet}         
\usepackage{courier}        
\usepackage{type1cm}        
\usepackage{makeidx}         
\usepackage{graphicx}        
\usepackage{multicol}        
\usepackage[bottom]{footmisc}
\usepackage{amsmath}
\usepackage{amsfonts}
\usepackage{amssymb}
\usepackage{amsthm}
\usepackage{amsopn}
\usepackage{subcaption}
\usepackage{comment}
\newtheorem{theorem}{Theorem}
\newtheorem{corollary}[theorem]{Corollary}
\newtheorem{lemma}[theorem]{Lemma}

\newtheorem{proposition}[theorem]{Proposition}

\newtheorem{remark}[theorem]{Remark}

\usepackage[top=2cm, bottom=2cm, left=2cm, right=2cm]{geometry}
\numberwithin{theorem}{subsection}
\numberwithin{figure}{subsection}
\numberwithin{equation}{subsection}
\DeclareMathOperator{\CR}{CR}

\DeclareMathOperator{\dist}{dist}
\DeclareMathOperator{\SLE}{SLE}
\DeclareMathOperator{\CLE}{CLE}
\DeclareMathOperator{\GFF}{GFF}
\DeclareMathOperator{\inte}{int}
\DeclareMathOperator{\ext}{ext}
\DeclareMathOperator{\inrad}{inrad}
\DeclareMathOperator{\outrad}{outrad}
\DeclareMathOperator{\dimH}{dim}

\begin{document}

\title{Level Lines of Gaussian Free Field II: \\Whole-Plane $\GFF$}
\author{Menglu WANG and Hao WU}


%
%
\maketitle

\abstract{We study the level lines of $\GFF$ starting from interior points. We show that the level line of $\GFF$ starting from an interior point turns out to be a sequence of level loops. The sequence of level loops satisfies ``target-independent" property. All sequences of level loops starting from interior points give a tree-structure of the plane. We also introduce the continuum exploration process of $\GFF$ starting from interior. The continuum exploration process of whole-plane $\GFF$ satisfies ``reversibility".}

\tableofcontents
\newcommand{\eps}{\epsilon}
\newcommand{\ov}{\overline}
\newcommand{\U}{\mathbb{U}}
\newcommand{\T}{\mathbb{T}}
\newcommand{\HH}{\mathbb{H}}
\newcommand{\LA}{\mathcal{A}}
\newcommand{\LC}{\mathcal{C}}
\newcommand{\LD}{\mathcal{D}}
\newcommand{\LF}{\mathcal{F}}
\newcommand{\LK}{\mathcal{K}}
\newcommand{\LE}{\mathcal{E}}
\newcommand{\LL}{\mathcal{L}}
\newcommand{\LU}{\mathcal{U}}
\newcommand{\LV}{\mathcal{V}}
\newcommand{\LZ}{\mathcal{Z}}
\newcommand{\R}{\mathbb{R}}
\newcommand{\C}{\mathbb{C}}
\newcommand{\N}{\mathbb{N}}
\newcommand{\Z}{\mathbb{Z}}
\newcommand{\E}{\mathbb{E}}
\newcommand{\PP}{\mathbb{P}}
\newcommand{\QQ}{\mathbb{Q}}
\newcommand{\A}{\mathbb{A}}
\newcommand{\bn}{\mathbf{n}}
\newcommand{\MR}{MR}
\newcommand{\cond}{\,|\,}
\newcommand{\la}{\langle}
\newcommand{\ra}{\rangle}
\newcommand{\tree}{\Upsilon}

\section{Introduction}
This is the second in a series of two papers the first of which is $\cite{WangWuLevellinesGFFI}$. Suppose that $h$ is a smooth real-valued function defined on the complex plane $\C$, then level lines of $h$ are curves along which $h$ has constant value. In \cite{SchrammSheffieldDiscreteGFF, SchrammSheffieldContinuumGFF}, the authors derive that one can still make sense of level lines of Gaussian Free Field ($\GFF$), which is nolonger a pointwise-defined function and can only be viewed as a distribution. The level lines of $\GFF$ are still continuous curves, which are variants of Schramm Loewner Evolution ($\SLE_4$). 

In \cite{WangWuLevellinesGFFI}, we study level lines of $\GFF$ that start from boundary points. In the current paper, we will study level lines of $\GFF$ that start from  interior points. We show that the level line of $\GFF$ starting from an interior point turns out to be a sequence of level loops (Theorems \ref{thm::interior_levelloops_coupling} and \ref{thm::interior_levelloops_deterministic}). We explain the interaction behavior between two sequences of level loops (Theorems \ref{thm::interior_levelloops_interacting_commontarget} and \ref{thm::interior_levelloops_interacting_commonstart}): two sequences of level loops with distinct starting points but common target point will merge; two sequences of level loops with common start point and distinct target points coincide up to the first disconnecting time after which the two processes continue toward their target points in a conditionally independent way. The latter fact is so-called 
 ``target-independent" property of the sequence of level loops. Consider all sequences of level loops starting from interior points, they give a tree-structure of the complex plane (Theorem \ref{thm::interior_levelloops_determins_field}).

We also introduce continuum exploration processes of $\GFF$ starting from interior points. They are sequences of quasisimple loops (Theorems \ref{thm::interior_continuum_coupling} and \ref{thm::interior_continuum_deterministic}). We describe the interaction behavior between two continuum exploration processes (Theorems \ref{thm::interior_continuum_interacting_commontarget} and \ref{thm::interior_continuum_interacting_commonstart}). All continuum exploration processes starting from interior points also give a tree-structure of the plane (Theorem \ref{thm::interior_continuum_determins_field}). We wish to highlight an interesting fact about the continuum exploration processes.The continuum exploration processes of whole-plane $\GFF$ satisfy reversibility: the continuum exploration process starting from the origin targeted at $\infty$ ``coincides" (in a certain sense which will be made precise in Theorem \ref{thm::wholeplane_continuum_reversibility}) with the continuum exploration process starting from $\infty$ targeted at the origin. 

In a series of papers \cite{MillerSheffieldIG1, MillerSheffieldIG2, MillerSheffieldIG3, MillerSheffieldIG4}, the authors study the flow lines of $\GFF$. In particular, \cite{MillerSheffieldIG1} studies the flow lines starting from boundary points and \cite{MillerSheffieldIG4} studies the flow line starting from interior points. The current paper is motivated by \cite{MillerSheffieldIG4}. However, the situation for the level line starting from interior is quite different from the flow line case: the flow line starting from an interior point is a continuous curve; whereas, the level line starting from interior will merge with itself, and the only natural way to describe it is by a sequence of loops. 

This paper is also an important part in a program on conformal invariant metric on $\CLE_4$ which includes \cite{WernerWuCLEExploration, SheffieldWatsonWuSimpleCLEDoubly, WangWuLevellinesGFFI}. In \cite{WernerWuCLEExploration}, the authors constructed a conformal invariant growing process in $\CLE_4$ (recalled in Section \ref{subsec::growing_cle4}) where each loop has a time parameter. The authors conjectured that the time parameter is a deterministic function of the loop configuration. Later, Scott Sheffield pointed out that the conformal invariant growing process in $\CLE_4$ is closely related to the exploration process of $\GFF$ (which is made precise in \cite[Section 3]{WangWuLevellinesGFFI} and Sections \ref{subsec::gff_discrete_continuum}, \ref{subsec::gff_alternate_continuum}). The tools that we develop in the current paper will be used in the program by Scott Sheffield, Samuel Watson, Wendelin Werner and Hao Wu which tries to prove the conjecture that the time parameter in $\CLE_4$ constructed in \cite{WernerWuCLEExploration} is a deterministic function of the loop configuration, and hence gives a conformal invariant metric on $\CLE_4$ loops. 
 
\subsection{Sequence of level loops starting from interior}
\label{subsec::sequence_levelloops_interior}
To consider level line of $\GFF$ starting from an interior point, it is natural to consider a random sequence of level loops instead of a random path, and the height difference between two neighbor loops is  $r\lambda$ for some fixed number $r\in (0,1)$.
A simple loop in the complex plane is the image of the unit circle in the plane under a continuous injective map. 
In other words, a simple loop is a subset of $\C$ which is homeomorphic to the unit circle $\partial\U$. Note that a simple loop is oriented either clockwise or counterclockwise.
If $L$ is a simple loop, then $L$ separates the plane into two connected components that we call its interior $\inte(L)$ (the bounded one) and its exterior $\ext(L)$ (the unbounded one). Assume that the origin is contained in $\inte(L)$, we define the inradius and outradius of the loop $L$ to be 
\[\inrad(L)=\sup\{r>0: r\U\subset\inte(L)\};\quad \outrad(L)=\inf\{R>0: \C\setminus R\U\subset \ext(L)\}.\]
We call a sequence of simple loops $(L_n, n\in\Z)$ in the plane a \textbf{transient sequence of adjacent simple loops} disconnecting the origin from infinity if the sequence satisfies the following properties. 
\begin{enumerate}
\item[(1)] For each $n$, the loop $L_n$ disconnects the origin from infinity;
\item[(2)] For each $n$, the loop $L_{n-1}$ is contained in the closure of $\inte(L_n)$ and $L_{n-1}\cap L_n\neq\emptyset$;
\item[(3)] The sequence is transient:
\[\outrad(L_n)\to 0\ \text{ as }\ n\to-\infty; \quad \inrad(L_n)\to\infty\ \text{ as }\ n\to\infty.\]
\end{enumerate}

Given a transient sequence of adjacent simple loops $(L_n,n\in\Z)$, we use the term \textbf{level loop boundary conditions with height difference $r\lambda$} to describe the boundary values given by the following rule. Assume that the loop $L_n$ has boundary value $c$ on the left-side and $2\lambda+c$ on the right-side, the boundary value of $L_{n-1}$ is 
\[\begin{cases}
c+r\lambda \text{ on the left-side and } 2\lambda+c+r\lambda \text{ on the right-side},& \text{if } L_{n-1} \text{ is clockwise};\\
c-r\lambda \text{ on the left-side and } 2\lambda+c-r\lambda \text{ on the right-side},& \text{if } L_{n-1} \text{ is counterclockwise}.
\end{cases}\]
See Figure \ref{fig::levelloops_boundary_outside}.

\begin{figure}[ht!]
\begin{subfigure}[b]{0.48\textwidth}
\begin{center}
\includegraphics[width=0.625\textwidth]{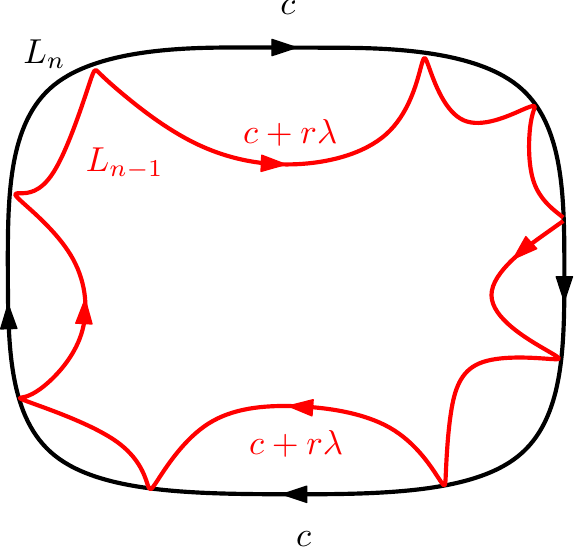}
\end{center}
\caption{If $L_{n-1}$ is clockwise, it has boundary value $c+r\lambda$ to the left-side.}
\end{subfigure}
$\quad$
\begin{subfigure}[b]{0.48\textwidth}
\begin{center}\includegraphics[width=0.625\textwidth]{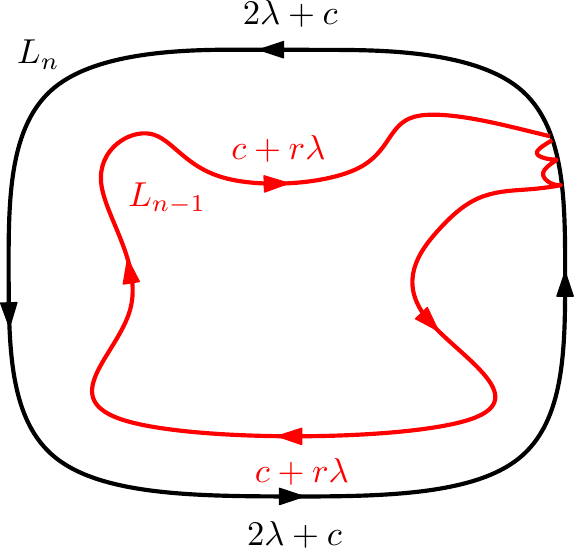}
\end{center}
\caption{If $L_{n-1}$ is clockwise, it has boundary value $c+r\lambda$ to the left-side.}
\end{subfigure}
\begin{subfigure}[b]{0.48\textwidth}
\begin{center}
\includegraphics[width=0.625\textwidth]{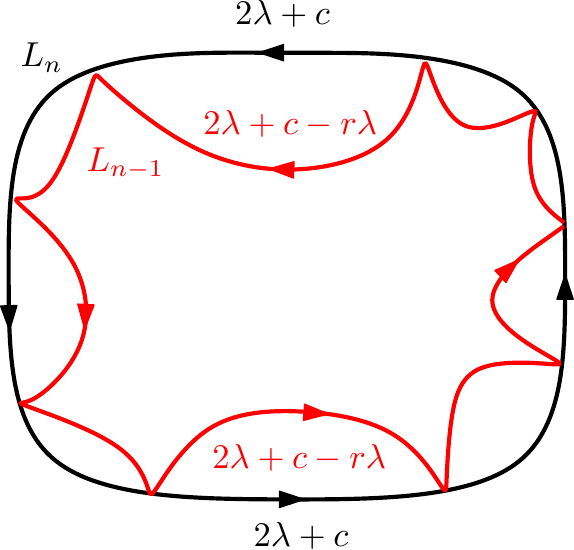}
\end{center}
\caption{If $L_{n-1}$ is counterclockwise, it has boundary value $2\lambda+c-r\lambda$ to the right-side.}
\end{subfigure}
$\quad$
\begin{subfigure}[b]{0.48\textwidth}
\begin{center}\includegraphics[width=0.625\textwidth]{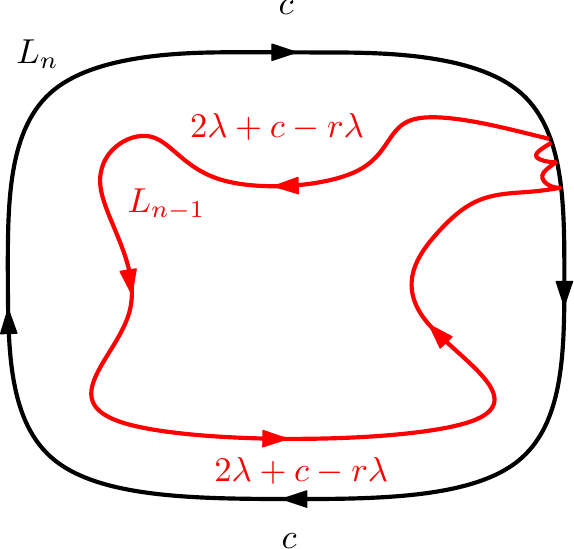}
\end{center}
\caption{If $L_{n-1}$ is counterclockwise, it has boundary value $2\lambda+c-r\lambda$ to the right-side.}
\end{subfigure}
\caption{\label{fig::levelloops_boundary_outside} Explanation of the level loop boundary conditions with height difference $r\lambda$. Assume that the boundary value of the loop $L_n$ is $c$ to the left-side and is $2\lambda+c$ to the right-side.}
\end{figure}

Note that, once we know the boundary values of a loop $L_{n_0}$ in the sequence, we can tell the boundary values of all loops $L_n$, for $n\le n_0$, by the level loop boundary conditions with height difference $r\lambda$. 
\medbreak
We say a domain $D\subset\C$ has harmonically non-trivial boundary if a Brownian motion started at a point in $D$ hits $\partial D$ almost surely. We call a sequence of simple loops $(L_n, n\in \Z)$ in $D$ a transient sequence of adjacent simple loops disconnecting $z\in D$ from $\partial D$ if the sequence satisfies the following properties.
\begin{enumerate}
\item[(1)] For each $n$, the loop $L_n$ disconnects $z$ from $\partial D$;
\item[(2)] For each $n$, the loop $L_{n-1}$ is contained in the closure of the connected component of $D\setminus L_n$ that contains $z$;
\item[(3)] The sequence is transient: the loop $L_n$ converges to $\{z\}$ in Hausdorff metric as $n\to -\infty$. 
\end{enumerate}

The level loop boundary conditions with height difference $r\lambda$ for a transient sequence of adjacent loops in $D$ is defined in the same way as before. 

The following two theorems explain that the level line of $\GFF$ starting from an interior point turns out to be a transient sequence of adjacent loops.

\begin{theorem}\label{thm::interior_levelloops_coupling}
Fix a connected domain $D\subset \C$ with harmonically non-trivial boundary, a starting point $z\in D$, and a number $r\in (0,1)$. Let $h$ be a $\GFF$ on $D$ with some boundary values. 
There exists a coupling between $h$ and a random transient sequence $(L_n, n\in \Z)$ of adjacent simple loops disconnecting $z$ from $\partial D$ such that the following domain Markov property is true. For any stopping time $N$, the conditional law of $h$ given $(L_n, n\le N)$ is a GFF on $D\setminus \cup _{n\le N}L_n$ whose boundary value agrees with the boundary value of $h$ on $\partial D$ and is given by level loop boundary conditions with height difference $r\lambda$ on $(L_n, n\le N)$.

If $D=\C$ and $h$ is a whole-plane $\GFF$ (modulo a global additive constant in $r\lambda\Z$), there is also a coupling between $h$ and a random transient sequence $(L_n, n\in \Z)$ of adjacent simple loops disconnecting the origin from infinity such that the following domain Markov property is true. For any stopping time $N$, the conditional law of $h$ given $(L_n, n\le N)$ is a $\GFF$ (modulo a global additive constant in $r\lambda\Z$) on $\C\setminus \cup _{n\le N}L_n$ whose boundary value is given by level loop boundary conditions with height difference $r\lambda$ on $(L_n, n\le N)$.
\end{theorem}

\begin{theorem}\label{thm::interior_levelloops_deterministic}
In the coupling of a $\GFF h$ and a random sequence of loops $(L_n, n\in \Z)$ as in Theorem \ref{thm::interior_levelloops_coupling}, the sequence $(L_n, n\in \Z)$ is almost surely determined by the field $h$ viewed as a distribution modulo a global additive constant in  $r\lambda\Z$.   
\end{theorem}

In the coupling between a whole-plane $\GFF$ $h$ and a sequence of loops $(L_n,n\in\Z)$ given by Theorem \ref{thm::interior_levelloops_coupling}, we call the sequence of level loops \textbf{the alternating height-varying sequence} of level loops of $h$ starting from the origin targeted at $\infty$.

The following two theorems explain the interaction behaviors between two sequences of level loops. 

\begin{theorem}\label{thm::interior_levelloops_interacting_commontarget}
Fix $r\in (0,1)$. Suppose that $h$ is a whole-plane $\GFF$ modulo a global additive constant in $r\lambda\Z$. Fix two starting points $z_1,z_2$. For $i=1,2$, let $(L^{z_i}_n, n\in\Z)$ be the alternating height-varying sequence of level loops of $h$ with height difference $r\lambda$ starting from $z_i$ targeted at $\infty$. Define
\[N_1=\min\{n: L^{z_1}_n\text{ disconnects }z_2 \text{ from }\infty\}; \quad N_2=\min\{n: L^{z_2}_n\text{ disconnects }z_1 \text{ from }\infty\}.\]
Then, almost surely, the loops $L^{z_1}_{N_1}$ and $L^{z_2}_{N_2}$ coincide; moreover, the two sequences of loops $(L^{z_1}_n, n\ge N_1)$ and $(L^{z_2}_n, n\ge N_2)$ coincide. 
\end{theorem}

\begin{theorem}\label{thm::interior_levelloops_interacting_commonstart}
Fix $r\in (0,1)$. Suppose that $h$ is a whole-plane $\GFF$ modulo a global additive constant in $r\lambda\Z$. Fix three points $z, w_1 ,w_2$. For $i=1,2$, let $(L^{z\to w_i}_n, n\in \Z)$ be the alternating height-varying sequence of level loops of $h$ with height difference $r\lambda$ starting from $z$ targeted at $w_i$; and denote by $U_n^{w_i}$ the connected component of $\C\setminus L_n^{z\to w_i}$ that contains $w_i$. Then there exists a number $M$ such that 
\[L_n^{z\to w_1}=L_n^{z\to w_2},\quad \text{for }n\le M-1; \quad \text{and}\quad U_n^{w_1}\cap U_n^{w_2}=\emptyset,\quad\text{ for }n=M.\]
Given $(L_n^{z\to w_1}, L_n^{z\to w_2},  n\le M)$, the two sequences continue towards their target points respectively in a conditionally independent way.
\end{theorem}
The following theorem explains that all sequences of level loops give a tree-structure of the plane. 
\begin{theorem}
\label{thm::interior_levelloops_determins_field}
Fix $r\in (0,1)$. Suppose that $h$ is a whole-plane $\GFF$ modulo a global additive constant in $r\lambda \Z$. Suppose that $(z_m, m\ge 1)$ is any countable dense set and, for each $m$, let $(L_n^m, n\in\Z)$ be the alternating height-varying sequence of level loops of $h$ with height difference $r\lambda$ starting from $z_m$ targeted at $\infty$. Then the collection $((L^m_n, n\in\Z), m\ge 1)$ almost surely determines the field $h$ and its law does not depend on the choice of $(z_m,m\ge 1)$. Moreover, the collection $((L^m_n, n\in\Z), m\ge 1)$ forms a tree structure in the sense that, almost surely, any two sequences of level loops merge eventually.
\end{theorem}

\begin{remark}
In the coupling $(h, (L_n, n\in\Z))$ as in Theorem \ref{thm::interior_levelloops_coupling} for $D=\C$, the Hausdorff dimension of the intersection of two neighbor loops is almost surely given by (see \cite[Remark 1.1.5]{WangWuLevellinesGFFI}) 
\[\dimH_H(L_n\cap L_{n+1})=2-\frac{1}{8}(2+r)^2.\]
This implies that, given the sequence of level loops $(L_n, n\in\Z)$, we can tell the height difference $r\lambda$ through the Hausdorff dimension of the intersection of two neighbor loops. 

Generally, for any two loops $L_n$ and $L_m$ in the sequence, on the event $\{L_n\cap L_m\neq\emptyset\}$, their height difference $\LD>0$ can be read from the Hausdorff dimension of their intersection:
\[\dimH_H(L_n\cap L_m)=2-\frac{1}{8}\left(\frac{\LD}{\lambda}+2\right)^2.\] 
\end{remark}

\subsection{Continuum exploration process starting from interior}

In this section, we will describe a continuum exploration process of $\GFF$ starting from an interior point. This process can be viewed as the limit of the sequence of level loops in Theorem \ref{thm::interior_levelloops_coupling} as $r$ goes to zero.

Before defining the continuum exploration process, we need to define a suitable space of loops (see also \cite[Section 4.1]{SheffieldExplorationTree}). As we introduced in Section \ref{subsec::sequence_levelloops_interior}, a simple loop is a subset of $\C$ which is homeomorphic to the unit circle $\partial\U$. Recall that $\inte(L)$ is the bounded connected component of $\C\setminus L$, and $\inte(L)$ is a simply connected domain. There exists a conformal map $\varphi$ from $\inte(L)$ onto $\U$ and $\varphi$ can be extended as a homeomorphism from the closure of $\inte(L)$ onto $\overline{\U}$. 

\begin{figure}[ht!]
\begin{subfigure}[b]{0.48\textwidth}
\begin{center}
\includegraphics[width=0.625\textwidth]{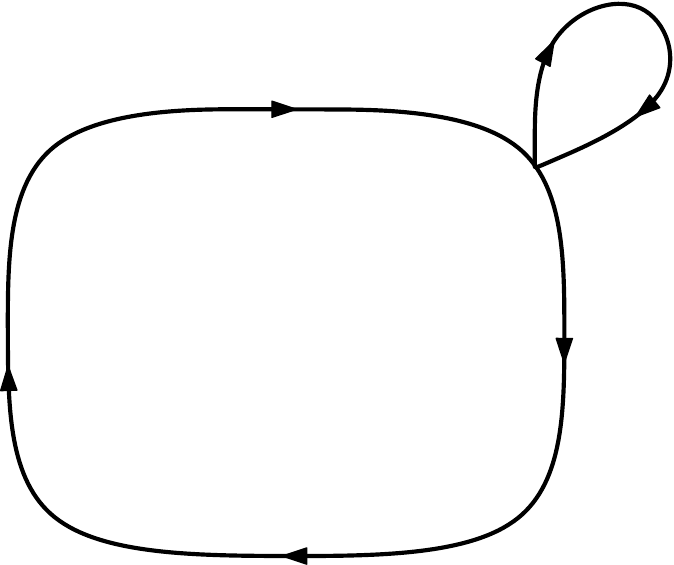}
\end{center}
\caption{A clockwise quasisimple loop.}
\end{subfigure}
$\quad$
\begin{subfigure}[b]{0.48\textwidth}
\begin{center}\includegraphics[width=0.625\textwidth]{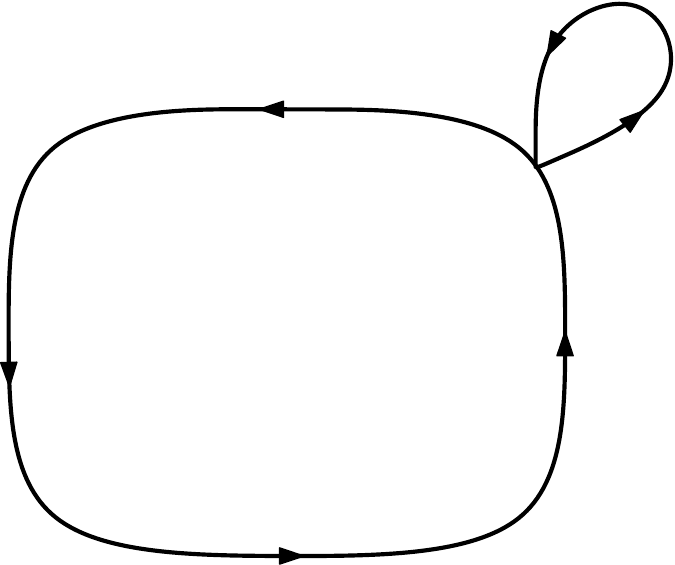}
\end{center}
\caption{A counterclockwise quasisimple loop.}
\end{subfigure}
\caption{\label{fig::quasisimpleloop_orientation} The orientations of quasisimple loops.}
\end{figure}

A quasisimple loop is the (oriented) boundary of a simply connected domain. Note that a quasisimple loop can be obtained as the limit of a sequence of simple loops under $L^{\infty}$ metric. A quasisimple loop may intersect itself, but it can not cross itself traversely. If $L$ is a quasisimple loop, we denote by $\ext(L)$ the unbounded connected component of $\C\setminus L$ (there may be several bounded connected components of $\C\setminus L$, but there is only one that is unbounded).  We will say $L$ is clockwise (resp. counterclockwise) if the boundary of $\ext(L)$ is clockwise (resp. counterclockwise). See Figure \ref{fig::quasisimpleloop_orientation}. Suppose that a quasisimple loop $L$ disconnects the origin from $\infty$, and let $U^0_L$ be the connected component of $\C\setminus L$ that contains the origin. Define
\[\inrad(L)=\sup\{r>0: r\U\subset U^0_L\},\quad \outrad(L)=\inf\{R>0: \C\setminus R\U\subset \ext(L)\}.\]

We will use the convergence of quasisimple loops in Carath\'eodory topology which is defined in the following way. We say that a sequence of quasisimple loops $(L_n, n\ge 1)$ converges to a quasisimple loop $L$ in  Carath\'eodory topology seen from $\infty$ if $\phi_0(\ext(L_n))$ converges to $\phi_0(\ext(L))$ in Carath\'eodory topology seen from the origin where $\phi_0(z)=1/z$ (see \cite[Section 3.6]{LawlerConformallyInvariantProcesses} for Carath\'eodory topology). 

We call a sequence of quasisimple loops $(L_u, u\in\R)$ in the plane a \textbf{transient c\`adl\`ag sequence of adjacent quasisimple loops} disconnecting the origin from infinity if the sequence satisfies the following properties.
\begin{enumerate}
\item [(1)] There exists a sequence of reals $(u_j, j\in\Z)$ such that 
\[u_j<u_{j+1}\ \text{ for all }\ j, \quad u_j\to-\infty\ \text{ as }\ j\to-\infty, \quad u_j\to\infty\ \text{ as }\ j\to\infty,\]
and that, for all $j$, the loops $L_u$  have the same orientation for $u_j\le u<u_{j+1}$ which is different from the orientation of $L_{u_{j+1}}$.
\item [(2)] For each $u$, the loop $L_u$ disconnects the origin from $\infty$.
\item [(3)] For $v<u$, the loop $L_v$ is contained in the closure of the connected component of $\C\setminus L_u$ that contains the origin. Moreover, the sequence is c\`adl\`ag: for each $u$, under Carath\'eodory topology seen from $\infty$, we have
\begin{itemize}
\item  the left limit $\lim_{v\uparrow u}L_v$ exists, denoted by $L_{u-}$;
\item  the right limit $\lim_{v\downarrow u}L_u$ exists and equals $L_u$.
\end{itemize}
And the two loops are adjacent: $L_u\cap L_{u-}\neq\emptyset$.
\item [(4)] The sequence is transient: 
\[\outrad(L_u)\to 0\ \text{ as }\ u\to-\infty; \quad \inrad(L_u)\to\infty\ \text{ as }\ u\to\infty.\]
\end{enumerate}

\begin{figure}[ht!]
\begin{subfigure}[b]{0.48\textwidth}
\begin{center}
\includegraphics[width=0.75\textwidth]{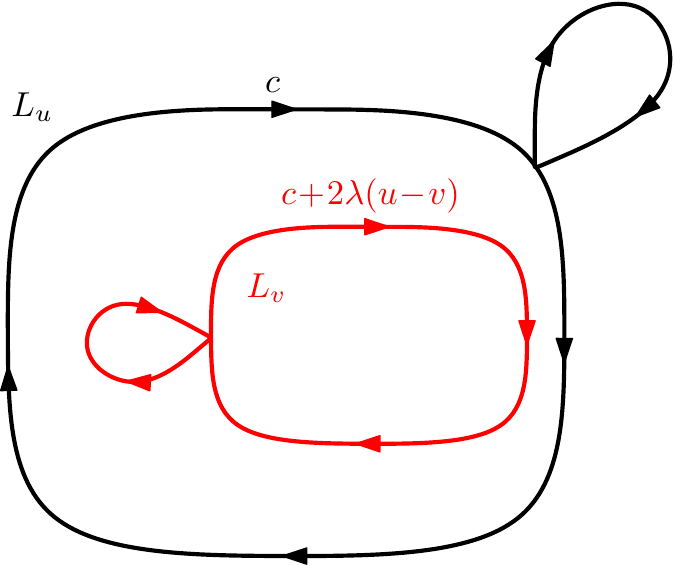}
\end{center}
\caption{Assume $u_j\le v< u<u_{j+1}$. If $L_{v}$ is clockwise, it has boundary value $c+2\lambda(u-v)$ to the left-side.}
\end{subfigure}
$\quad$
\begin{subfigure}[b]{0.48\textwidth}
\begin{center}\includegraphics[width=0.75\textwidth]{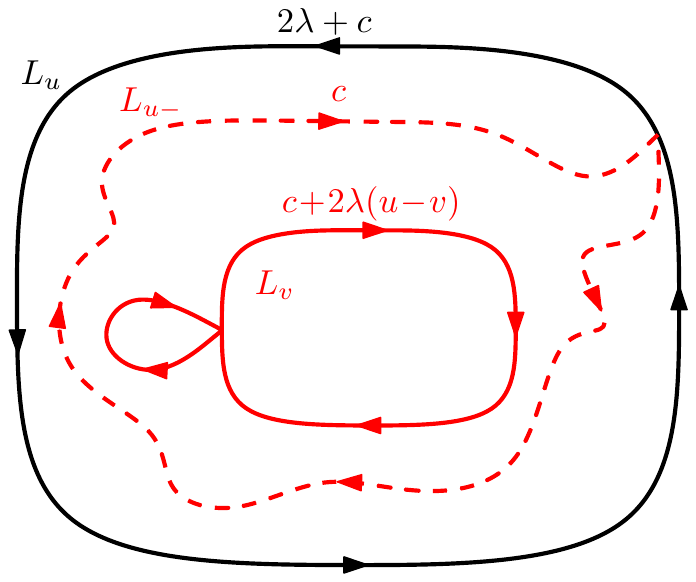}
\end{center}
\caption{Assume $u_j\le v< u=u_{j+1}$. If $L_{v}$ is clockwise, it has boundary value $c+2\lambda(u-v)$ to the left-side.}
\end{subfigure}
\begin{subfigure}[b]{0.48\textwidth}
\begin{center}
\includegraphics[width=0.75\textwidth]{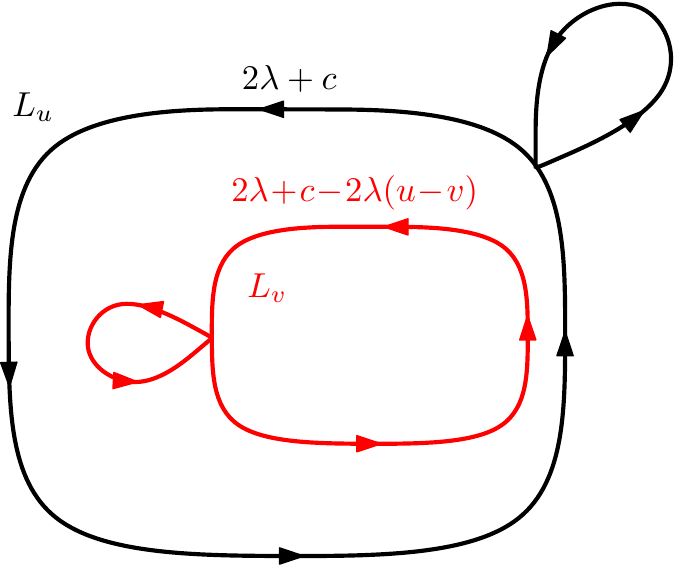}
\end{center}
\caption{Assume $u_j\le v< u<u_{j+1}$. If $L_{v}$ is counterclockwise, it has boundary value $2\lambda+c-2\lambda(u-v)$ to the right-side.}
\end{subfigure}
$\quad$
\begin{subfigure}[b]{0.48\textwidth}
\begin{center}\includegraphics[width=0.75\textwidth]{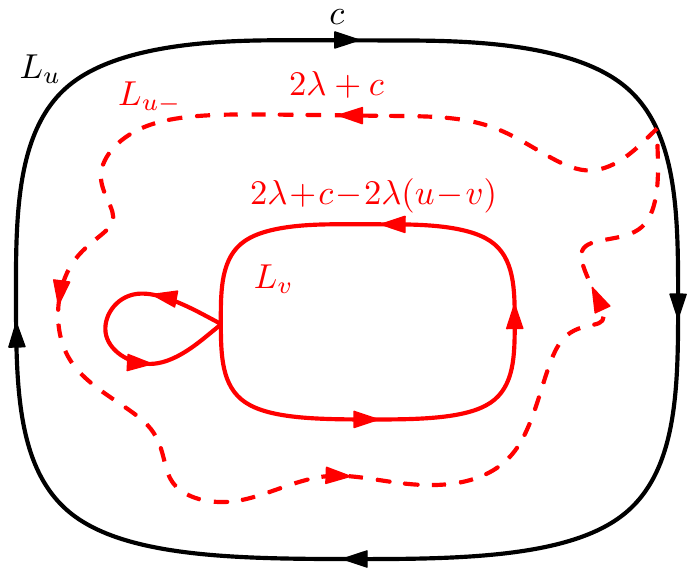}
\end{center}
\caption{Assume $u_j\le v< u=u_{j+1}$. If $L_{v}$ is counterclockwise, it has boundary value $2\lambda+c-2\lambda(u-v)$ to the right-side.}
\end{subfigure}
\caption{\label{fig::quasisimpleloop_boundary_outside} Explanation of the continuum level loop boundary conditions. Assume that the boundary value of the loop $L_u$ is $c$ to the left-side and is $2\lambda+c$ to the right-side.}
\end{figure}

Given a transient c\`adl\`ag sequence of adjacent quasisimple loops $(L_u, u\in\R)$, we use the term \textbf{continuum level loop boundary conditions} to describe the boundary values given by the following rule. Assume that the loop $L_u$ has boundary value $c$ to the left-side and $2\lambda+c$ to the right-side, and assume that $u_j<u\le u_{j+1}$ for some $j$. Then, for $v\in [u_j, u)$, the boundary value of $L_v$ is 
\[\begin{cases}
c+2\lambda(u-v) \text{ on the left-side and }2\lambda+c+2\lambda(u-v) \text{ on the right-side},&\text{if }L_v\text{ is clockwise};\\
c-2\lambda(u-v) \text{ on the left-side and }2\lambda+c-2\lambda(u-v) \text{ on the right-side},&\text{if }L_v\text{ is counterclockwise}.
\end{cases}\]
See Figure \ref{fig::quasisimpleloop_boundary_outside}. Note that, once we know the boundary value of a loop $L_u$ in the sequence, we can tell the boundary values of all loops $L_v$, for $v<u$, by the continuum level loop boundary conditions. 

Suppose that $D\subset \C$ is a domain with harmonically non-trivial boundary. We call a sequence of quasisimple loops $(L_u, u\in\R)$ in $D$ a transient c\`adl\`ag sequence of adjacent quasisimple loops disconnecting $z\in D$ from $\partial D$ if the sequence satisfies the following properties.

\begin{enumerate}
\item [(1)] There exists a sequence of reals $(u_j, j\le j_0)$ such that 
\[u_j<u_{j+1}\ \text{ for all }\ j\le j_0-1, \quad u_j\to-\infty\ \text{ as }\ j\to-\infty,\]
and that, for all $j\le j_0-1$, the loops $L_u$  have the same orientation for $u_j\le u<u_{j+1}$ which is different from the orientation of $L_{u_{j+1}}$.
\item [(2)] For each $u$, the loop $L_u$ disconnects $z$ from $\partial D$.
\item [(3)] For $v<u$, the loop $L_v$ is contained in the closure of the connected component of $D\setminus L_u$ that contains $z$. Moreover, the sequence is c\`adl\`ag under Carath\'eodory topology seen from $\infty$, and for each $u$, the loop $L_u$ has nonempty intersection with its left limit.
\item [(4)] The sequence is transient: 
the loop $L_u$ converges to $\{z\}$ in Hausdorff metric as $u\to-\infty$.
\end{enumerate}

The continuum level loop boundary conditions for a transient c\`adl\`ag sequence of quasisimple loops in $D$ can be defined similarly. 

The following two theorems explains that the continuum exploration process of $\GFF$ starting from an interior point is a c\`adl\`ag sequence of quasisimple loops.  

\begin{theorem}
\label{thm::interior_continuum_coupling}
Fix a connected domain $D\subset\C$ with harmonically non-trivial boundary and a starting point $z\in D$. Suppose that $h$ is a $\GFF$ on $D$ with some boundary values and that $(L_u, u\in\R)$ is a transient c\`adl\`ag sequence of adjacent quasisimple loops in $D$ disconnecting $z$ from $\partial D$. There exists a coupling between $h$ and $(L_u, u\in\R)$ such that the following domain Markov property is true. For any stopping time $T$, the conditional law of $h$ given $(L_u, u\le T)$ is a $\GFF$ on $D\setminus \cup_{u\le T}L_u$ whose boundary value agrees with the boundary value of $h$ on $\partial D$ and is given by the continuum level loop boundary conditions on $(L_u, u\le T)$. 

Suppose that $h$ is a whole-plane $\GFF$ modulo a global additive constant in $\R$ and that $(L_u, u\in\R)$ is a random transient c\`adl\`ag sequence of adjacent quasisimple loops disconnecting the origin from infinity. There is a coupling between $h$ and $(L_u, u\in\R)$ such that the following domain Markov property is true. For any stopping time $T$, the conditional law of $h$ given $(L_u, u\le T)$ is a $\GFF$ (modulo a global additive constant in $\R$) on $\C\setminus \cup_{u\le T}L_u$ whose boundary value is given by the continuum level loop boundary conditions on $(L_u, u\le T)$.
\end{theorem}

\begin{theorem}
\label{thm::interior_continuum_deterministic}
In the coupling of a $\GFF$ $h$ and a random sequence of quasisimple loops $(L_u, u\in\R)$ as in Theorem \ref{thm::interior_continuum_coupling}, the sequence $(L_u, u\in\R)$ is almost surely determined by the field $h$ viewed as a distribution modulo a global additive constant in $\R$.
\end{theorem}

In the coupling between a whole-plane $\GFF$ $h$ and a random sequence of quasisimple loops $(L_u, u\in\R)$ given by Theorem \ref{thm::interior_continuum_coupling}, we call the sequence $(L_u, u\in\R)$ the \textbf{alternating continuum exploration process} of $h$ starting from the origin targeted at $\infty$; and we call $(u_j, j\in\Z)$ the \textbf{sequence of transition times}.

The following two theorems explain the interaction behavior between two alternating continuum exploration processes. 

\begin{theorem}\label{thm::interior_continuum_interacting_commontarget}
Suppose that $h$ is a whole-plane $\GFF$ modulo a global additive constant in $\R$. Fix two starting points $z_1,z_2$. For $i=1,2$, let $(L^{z_i}_u, u\in\R)$ be the alternating continuum exploration process of $h$ starting from $z_i$ targeted at $\infty$. Define
\[T_1=\inf\{u: L^{z_1}_u\text{ disconnects }z_2 \text{ from }\infty\}; \quad T_2=\inf\{u: L^{z_2}_u\text{ disconnects }z_1 \text{ from }\infty\}.\]
Then, almost surely, the loops $L^{z_1}_{T_1}$ and $L^{z_2}_{T_2}$ coincide; moreover, the two sequences of loops $(L^{z_1}_{u+T_1}, u\ge 0)$ and $(L^{z_2}_{u+T_2}, u\ge 0)$ coincide. 
\end{theorem}

\begin{theorem}\label{thm::interior_continuum_interacting_commonstart}
Suppose that $h$ is a whole-plane $\GFF$ modulo a global additive constant in $\R$. Fix three points $z, w_1 ,w_2$. For $i=1,2$, let $(L^{z\to w_i}_u, u\in\R)$ be the alternating continuum exploration process of $h$ starting from $z$ targeted at $w_i$; and denote by $U_u^{w_i}$ the connected component of $\C\setminus L_u^{z\to w_i}$ that contains $w_i$. Then there exists a number $T$ such that 
\[L_u^{z\to w_1}=L_u^{z\to w_2},\quad \text{for }u<T; \quad \text{and}\quad U_T^{w_1}\cap U_T^{w_2}=\emptyset.\]
Given $(L_u^{z\to w_1}, L_u^{z\to w_2},  u\le T)$, the two sequences continue towards their target points respectively in a conditionally independent way.
\end{theorem}

The following theorem explains that all the alternating continuum exploration processes of $\GFF$ give a tree-structure of the plane.

\begin{theorem}
\label{thm::interior_continuum_determins_field}
Suppose that $h$ is a whole-plane $\GFF$ modulo a global additive constant in $\R$. Suppose that $(z_m, m\ge 1)$ is any countable dense set and, for each $m$, let $(L_u^m, u\in\R)$ be the alternating continuum exploration process of $h$ starting from $z_m$ targeted at $\infty$. Then the collection $((L^m_u, u\in\R), m\ge 1)$ almost surely determines the field $h$ and its law does not depend on the choice of $(z_m,m\ge 1)$. Moreover, the collection $((L^m_u, u\in\R), m\ge 1)$ forms a tree structure in the sense that, almost surely, any two sequences of quasisimple loops merge eventually.
\end{theorem}

\begin{figure}[ht!]
\begin{subfigure}[b]{0.48\textwidth}
\begin{center}
\includegraphics[width=\textwidth]{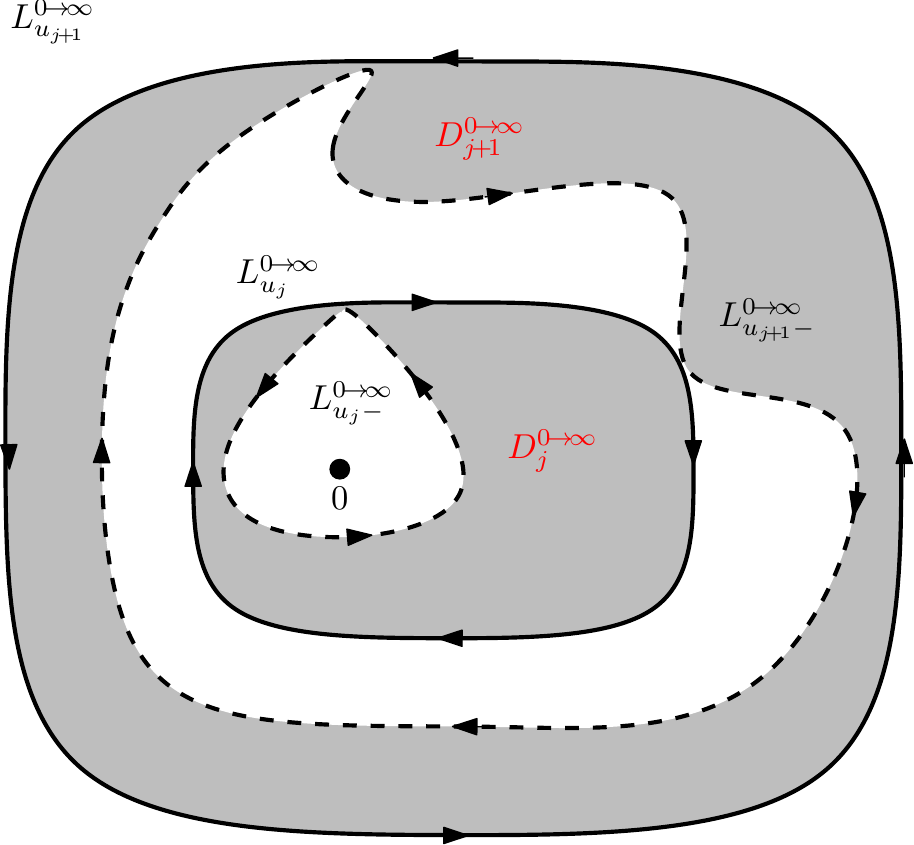}
\end{center}
\caption{The transition region $D_{j}^{0\to\infty}$ is the domain lying between the loop $L_{u_j-}^{0\to\infty}$ and the loop $L_{u_j}^{0\to\infty}$.}
\end{subfigure}
$\quad$
\begin{subfigure}[b]{0.48\textwidth}
\begin{center}\includegraphics[width=\textwidth]{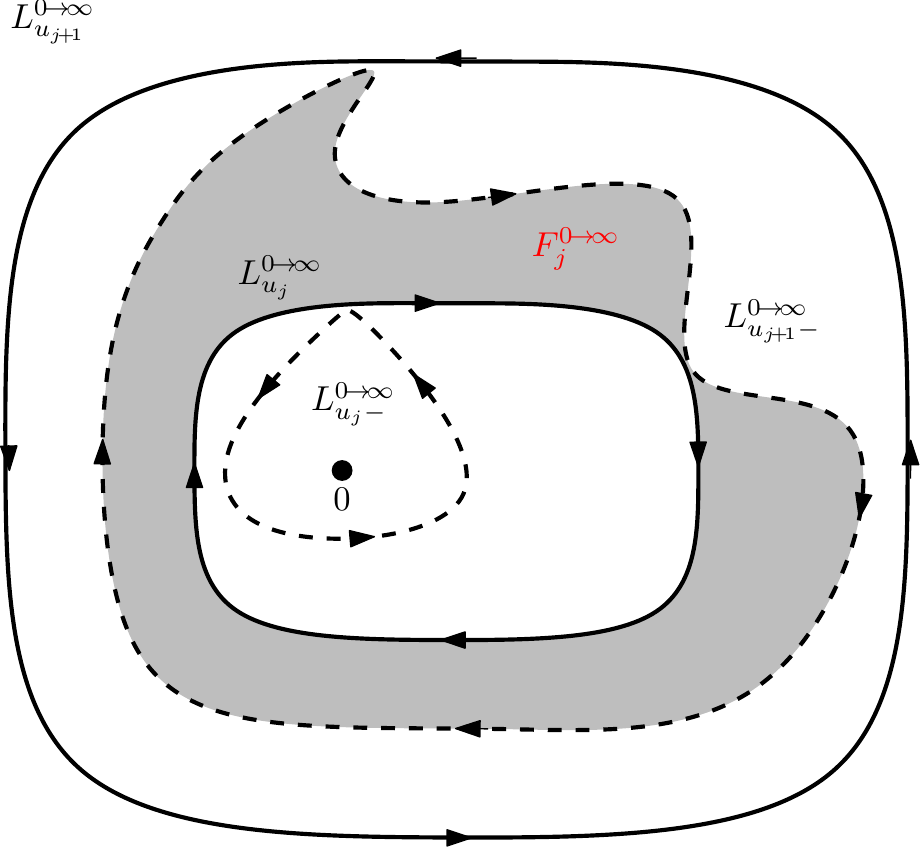}
\end{center}
\caption{The transition region $F_{j}^{0\to\infty}$ is the domain lying between the loop $L_{u_j}^{0\to\infty}$ and the loop $L_{u_{j+1}-}^{0\to\infty}$.}
\end{subfigure}
\caption{\label{fig::quasisimpleloop_transition_regions} Explanation of the sequence of transition regions and the sequence of stationary regions.}
\end{figure}

Next, we will explain the ``reversibility" of the alternating continuum exploration processes. 
Suppose that $h$ is a whole-plane $\GFF$ modulo a global additive constant in $\R$. Let $(L_u^{0\to\infty}, u\in\R)$ be the alternating continuum exploration process of $h$ starting from the origin targeted at $\infty$ where $(u_j, j\in\Z)$ is the sequence of transition times. For each transition time $u_j$, define the corresponding transition region by 
\[D^{0\to\infty}_j:=\overline{\ext(L^{0\to\infty}_{u_j-})}\setminus \ext(L_{u_j}^{0\to\infty}).\]
In other words, the region $D_j^{0\to\infty}$ is the domain lying between the loops $L_{u_j-}^{0\to\infty}$ and $L_{u_j}^{0\to\infty}$. See Figure \ref{fig::quasisimpleloop_transition_regions}(a). We call $(D_j^{0\to\infty}, j\in\Z)$ the \textbf{sequence of transition regions} of $h$ starting from the origin targeted at $\infty$. For each transition time $u_j$, define the corresponding stationary region by 
\[F_{j}^{0\to\infty}=\overline{\ext(L_{u_j}^{0\to\infty})}\setminus \ext(L^{0\to\infty}_{u_{j+1}-}).\]
In other words, the region $F_j^{0\to\infty}$ is the domain lying between $D_j^{0\to\infty}$ and $D_{j+1}^{0\to\infty}$. See Figure \ref{fig::quasisimpleloop_transition_regions}(b). We call $(F_j^{0\to\infty}, j\in\Z)$ the \textbf{sequence of stationary regions} of $h$ starting from the origin targeted at $\infty$.

\begin{theorem}
\label{thm::wholeplane_continuum_reversibility}
Suppose that $h$ is a whole-plane $\GFF$ modulo a global additive constant in $\R$. Let $(D_j^{0\to\infty}, j\in\Z)$ (resp. $(F_j^{0\to\infty}, j\in\Z)$) be the sequence of transition regions (resp. the sequence of stationary regions) of $h$ starting from the origin targeted at $\infty$. Let $(D_j^{\infty\to 0}, j\in\Z)$ (resp. $(F_j^{\infty\to 0}, j\in\Z)$) be the sequence of transition regions (resp. the sequence of stationary regions) of $h$ starting from $\infty$ targeted at the origin. Then, almost surely, the two unions
\[\bigcup_{j\in\Z}D_j^{0\to\infty}\quad \text{and}\quad \bigcup_{j\in\Z}D_j^{\infty\to 0}\]
coincide; moreover, the two unions 
\[\bigcup_{j\in\Z}F_j^{0\to\infty}\quad \text{and}\quad \bigcup_{j\in\Z}F_j^{\infty\to 0}\]
coincide.
\end{theorem}

\subsection{Relation to previous works and outline}

We prove Theorems \ref{thm::interior_levelloops_coupling} to \ref{thm::interior_levelloops_determins_field} in Section \ref{sec::interior_levelloops}. Following is the outline of Section \ref{sec::interior_levelloops}.
\begin{itemize}
\item Section \ref{subsec::wholeplane_gff} is an introduction to whole-plane $\GFF$. We briefly recall the definition and the properties of whole-plane $\GFF$ and refer to \cite[Section 2.2]{MillerSheffieldIG4} for details and proofs. 
\item In Section \ref{subsec::alternate_levelloops_construction}, we first recall the construction and properties of boundary emanating level lines of $\GFF$ from \cite{WangWuLevellinesGFFI}, then introduce the sequence of level loops starting from interior, and complete the proof of Theorem \ref{thm::interior_levelloops_coupling}. \item In Section \ref{subsec::interaction}, we explain the interaction behavior between two sequences of level loops starting from interior, and complete the proofs of Theorems \ref{thm::interior_levelloops_deterministic} to \ref{thm::interior_levelloops_determins_field}. 
\end{itemize} 

\noindent We prove Theorems \ref{thm::interior_continuum_coupling} to \ref{thm::wholeplane_continuum_reversibility} in Section \ref{sec::continuum_exploration}. Following is the outline of Section \ref{sec::continuum_exploration}.
\begin{itemize}
\item In Section \ref{subsec::growing_cle4}, we first briefly introduce a continuum exploration process for $\CLE_4$ and refer to \cite{WernerWuCLEExploration} for details. Then we show that the continuum exploration process is c\`adl\`ag which is an ingredient of the proof of Theorem \ref{thm::interior_continuum_coupling}.
\item In Section \ref{subsec::gff_discrete_continuum}, we explain how to get the continuum exploration process of $\GFF$ from discrete exploration process of $\GFF$. 
\item In Section \ref{subsec::gff_alternate_continuum} we complete the proof of Theorem \ref{thm::interior_continuum_coupling}.
\item In Section \ref{subsec::continuum_interaction}, we explain the interaction behavior between two continuum exploration processes, and complete the proofs of Theorems \ref{thm::interior_continuum_deterministic} to \ref{thm::interior_continuum_determins_field}.
\item In Section \ref{subsec::continuum_reversibility}, we complete the proof of Theorem \ref{thm::wholeplane_continuum_reversibility}.
\end{itemize}

\bigbreak
\noindent\textbf{Acknowledgements.} 
The current project is motivated by \cite{MillerSheffieldIG4}, and we appreciate Jason Miller and Scott Sheffield for helpful discussions on both \cite{MillerSheffieldIG4} and the current paper. 
We also thank Richard Kenyon, Samuel Watson, and Wendelin Werner for helpful discussions. 
H.\ Wu's work is funded by NSF DMS-1406411.

\section{Sequence of level loops starting from interior}
\label{sec::interior_levelloops}

\subsection{The whole-plane $\GFF$}
\label{subsec::wholeplane_gff}
In this section, we will give an overview of whole-plane $\GFF$, we refer the proofs related to the properties of whole-plane $\GFF$ to \cite[Section 2.2]{MillerSheffieldIG4}. For the construction and the properties of ordinary $\GFF$, one can consult \cite[Section 3]{MillerSheffieldIG1} or \cite[Section 2.2]{WangWuLevellinesGFFI}. 

For a domain $D\subset \C$, we denote by $H_s(D)$ the space of $C^{\infty}$ functions with compact support in $D$ and denote by $H_{s,0}(D)$ the space of those $\phi\in H_s(D)$ with $\int \phi(z)dz=0$. We denote by $(\cdot, \cdot)$ the $L^2$ inner product. Let $H_0(\C)$ denote the Hilbert space closure of $H_{s,0}(\C)$ equipped with the Dirichlet inner product. Let $(\alpha_j)$ be a sequence of i.i.d. one-dimensional standard Gaussian (with mean zero and variance one). Let $(f_j)$ be an orthonormal basis of $H_0(\C)$. The \textbf{whole-plane $\GFF$ modulo a global additive constant in $\R$} is an equivalence class of distributions (where two distributions are equivalent when their difference is a constant in $\R$), a representative of which is given by 
$h=\sum \alpha_jf_j$. For each $\phi\in H_{s,0}(\C)$, the limit $\sum \alpha_j(f_j,\phi)$ almost surely exists, thus we write
\[(h,\phi)=\lim_n\sum_{j=1}^n \alpha_j (f_j, \phi).\]
We think of $h$ as being defined up to an additive constant in $\R$ because
\[(h+c, \phi)=(h,\phi)+(c,\phi)=(h,\phi),\quad \forall \phi\in H_{s,0}(\C),\quad \forall c\in\R.\]
We can fix the additive constant by setting $(h,\phi_0)=0$ for some fixed $\phi_0\in H_s(\C)$ with $\int \phi_0(z)dz=1$.
\medbreak
Fix $u>0$ and $\phi_0\in H_s(\C)$ with $\int\phi_0(z)dz=1$. The \textbf{whole-plane $\GFF$ modulo a global additive constant in $u\Z$} is an equivalence class of distributions (where two distributions are equivalent when their difference is a constant in $u\Z$). An instance can be generated by 
\begin{enumerate}
\item [(1)] sampling a whole-plane $\GFF$ $h$ modulo a global additive constant in $\R$, and then
\item [(2)] choosing independently a uniform random variable $U\in [0,u)$ and fixing a global additive constant for $h$ by requiring that $(h,\phi_0)\in (U+u\Z)$.
\end{enumerate}

Suppose that $h$ is a zero-boundary $\GFF$ on a proper domain $D\subset\C$ with harmonically non-trivial boundary. We can consider $h$ modulo a global additive constant in $\R$ by restricting $h$ to functions in $H_{s,0}(D)$. Fix $\phi_0\in H_s(D)$ with $\int \phi_0(z)dz=1$. We can also consider $h$ modulo a global additive constant in $u\Z$ by replacing $h$ by $h-c$ where $c\in u\Z$ is chosen so that $(h,\phi_0)-c\in [0,u)$.

The following proposition describes the domain Markov property of whole-plane $\GFF$.

\begin{proposition}
Suppose that $h$ is a whole-plane $\GFF$ modulo a global additive constant in $\R$ (resp. in $u\Z$) and that $W\subset\C$ is open and bounded. The conditional law of $h|_W$ given $h|_{\C\setminus W}$ is that of a zero-boundary $\GFF$ on $W$ plus the harmonic extension of its boundary values from $\partial W$ to $W$ which is defined modulo a global additive constant in $\R$ (resp. in $u\Z$).  
\end{proposition}
\begin{proof}
\cite[Proposition 2.8]{MillerSheffieldIG4}.
\end{proof}

The following proposition explains that infinite volume limits of ordinary $\GFF$ converge to the whole-plane $\GFF$.

\begin{proposition}\label{prop::wholeplanegff_convergence_2.10}
Suppose that $(D_n)$ is any sequence of domains with harmonically non-trivial boundary containing the origin such that $\dist(0,\partial D_n)\to\infty$ as $n\to\infty$. For each $n$, let $h_n$ be an instance of $\GFF$ on $D_n$ with boundary conditions which are uniformly bounded in $n$. Fix $R>0$, we have the following.
\begin{enumerate}
\item [(1)] As $n\to\infty$, the laws of $h_n$ restricted to $B(0,R)$, viewed as a distribution modulo a global additive constant in $\R$, converge in total variation to the law of whole-plane $\GFF$ restricted to $B(0,R)$ modulo a global additive constant in $\R$.
\item [(2)] Fix $\phi_0\in H_s(\C)$ with $\int \phi_0(z)dz=1$ which is constant and positive on $B(0,R)$. As $n\to\infty$, the laws of $h_n$ restricted to $B(0,R)$, viewed as a distribution modulo a global additive constant in $u\Z$, converge in total variation to the law of whole-plane $\GFF$ restricted to $B(0,R)$ modulo a global additive constant in $u\Z$.
\end{enumerate}
\end{proposition}

\begin{proof}
\cite[Proposition 2.10]{MillerSheffieldIG4}.
\end{proof}

\begin{proposition}\label{prop::wholeplanegff_ac_2.11}
Suppose that $D\subset\C$ with harmonically non-trivial boundary and $h$ is a $\GFF$ on $D$ with some boundary data.
Fix $W\subset D$ open and bounded with $\dist(W,\partial D)>0$. The law of $h|_W$ modulo a global additive constant in $\R$ (resp. in $u\Z$) is mutually absolutely continuous with respect to the law of whole-plane $\GFF$ restricted to $W$ modulo a global additive constant in $\R$ (resp. in $u\Z$).
\end{proposition}

\begin{proof}
\cite[Proposition 2.11]{MillerSheffieldIG4}.
\end{proof}

The notion of \textbf{local set}, first introduced in \cite{SchrammSheffieldContinuumGFF} for ordinary $\GFF$, serves to generalize the domain Markov property to random subsets. There is an analogous theory of local sets for the whole-plane $\GFF$.

Suppose that $h$ is a whole-plane $\GFF$ modulo a global additive constant in $\R$ (resp. in $u\Z$) and suppose that $A\subset\C$ is a random closed subset which is coupled with $h$ such that $\partial(\C\setminus A)$ has harmonically non-trivial boundary. We say that $A$ is a local set of $h$ if there exists a law on pairs $(A, h_1)$, where $h_1$ is a distribution with the property that $h_1|_{\C\setminus A}$ is almost surely a harmonic function, such that a sample with the law $(A, h)$ can be produced by 
\begin{enumerate}
\item [(1)] choosing the pair $(A, h_1)$,
\item [(2)] sampling an instance $h_2$ of zero-boundary $\GFF$ on $\C\setminus A$ and setting $h=h_1+h_2$,
\item [(3)] considering the equivalence class of distribution modulo a global additive constant in $\R$ (resp. in $u\Z$) represented by $h$. 
\end{enumerate}
The definition is equivalent if we consider $h_1$ as being defined modulo a global additive constant in $\R$ (resp. in $u\Z$). We write $\LC_A$ for the function $h_1$ described above. We say that $\LC_A$ is the conditional mean of $h$ given $A$ and $h|_A$.

By this definition, Theorem \ref{thm::interior_levelloops_coupling} implies that, for any stopping time $N$, the sequence of loops $(L_n, n\le N)$ is a local set for the whole-plane $\GFF$ modulo a global additive constant in $r\lambda\Z$. 

The following propositions are properties of local sets of whole-plane $\GFF$. 

\begin{proposition}
Suppose that $h$ is a whole-plane $\GFF$ modulo a global additive constant in $\R$ or in $u\Z$. Suppose that $A_1,A_2$ are random closed subsets and that $(h,A_1)$ and $(h,A_2)$ are couplings for which $A_1,A_2$ are local. Let $A=A_1\tilde{\cup}A_2$ denote the random closed subset which is given by first sampling $h$, then sampling $A_1,A_2$ conditionally independent given $h$, and then taking the union of $A_1$ and $A_2$. Then $A$ is also local for $h$. 
\end{proposition}

\begin{proof}
\cite[Proposition 2.14]{MillerSheffieldIG4}.
\end{proof}

\begin{proposition}
Let $A_1,A_2$ be connected local sets of whole-plane $\GFF$ which are conditionally independent and $A=A_1\tilde{\cup}A_2$. Then $\LC_A-\LC_{A_2}$ is almost surely a harmonic function in $\C\setminus A$ that tends to zero along all sequences of points in $\C\setminus A$ that tend to a limit in a connected component of $A_2\setminus A_1$ (which consists of more than a single point) or that tend to a limit on a connected component of $A_1\cap A_2$ (which consists of more than a single point) at a point that is a positive distance from either $A_2\setminus A_1$ or $A_1\setminus A_2$.
\end{proposition}

\begin{proof}
\cite[Proposition 2.15]{MillerSheffieldIG4}.
\end{proof}

\begin{proposition}\label{prop::wholeplanegff_localset_2.16}
Suppose that $A$ is a local set for a whole-plane $\GFF$ modulo a global additive constant in $\R$ (resp. in $u\Z$). Fix $W\subset\C$ open and bounded and assume that $A\subset W$ almost surely.
Let $D\subset\C$ be a domain with harmonically non-trivial boundary such that $W\subset D$ with $\dist(W,\partial D)>0$ and let $h_D$ be a $\GFF$ on $D$. 
\begin{enumerate}
\item [(1)] There exists a law on random closed set $A_D$ which is mutually absolutely continuous with respect to the law of $A$ such that $A_D$ is a local set for $h_D$.
\item [(2)] Let $\LC^{\C}_{A_D}$ be the function which is harmonic in $\C\setminus A$ and which has the same boundary behavior as $\LC_{A_D}$ on $A_D$. Then the law of $\LC^{\C}_{A_D}$, modulo a global additive constant in $\R$ (resp. in $u\Z$), and the law of $\LC_A$ are mutually absolutely continuous. 
\item [(3)] If $A$ is almost surely determined by $h$, then $A_D$ is almost surely determined by $h_D$.
\end{enumerate}
\end{proposition}
\begin{proof}
\cite[Proposition 2.16]{MillerSheffieldIG4}.
\end{proof}

\subsection{Alternating height-varying sequence of level loops}
\label{subsec::alternate_levelloops_construction}
In \cite[Sections 2, 3.3, 3.4]{WangWuLevellinesGFFI}, we studied the level line of $\GFF$ starting from a boundary point targeted at a boundary point or an interior point. We briefly recall some basic properties. 

Suppose that $h$ is a $\GFF$ on $\U$ whose boundary value is piecewise constant and changes only finitely many times. Fix a starting point $x\in\partial\U$ and a target point $z\in \overline{\U}$. The level line of $h$ starting from $x$ targeted at $z$ is a random curve $\gamma$ coupled with $h$ such that the following is true. 
Suppose that $\tilde{\tau}$ is any $\gamma-$stopping time. Then, given $\gamma[0,\tilde{\tau}]$, the conditional law of $h$ restricted to $\U\setminus \gamma[0,\tilde{\tau}]$ is the same as $\GFF$s in each connected component whose boundary value is consistent with $h$ on $\partial\U$ and is $\lambda$ to the right of $\gamma$ and is $-\lambda$ to the left of $\gamma$. In this coupling, the curve $\gamma$ is almost surely determined by $h$ and is almost surely continuous up to and including the continuation threshold. 
Generally, for $u\in\R$, the level line of $h$ with height $u$ is the level line of $h+u$. 

In particular, suppose that $h$ is a zero-boundary $\GFF$ on $\U$, let $\gamma_u$ be the level line of $h$ starting from $x\in\partial\U$ targeted at $z\in\U$ with height $u\in (-\lambda,\lambda)$. We parameterize the curve by minus the log of the conformal radius of $\U\setminus \gamma_u[0,t]$ seen from $z$:
\[\CR(\U\setminus\gamma_u[0,t];z)=e^{-t},\quad t\ge 0.\]
Then the law of $\gamma_u$ is the same as radial $\SLE_4(\rho^L;\rho^R)$ process starting with $W_0=x$ with two force points next to the starting point and the corresponding weights are given by \cite[Proposition 3.3.1]{WangWuLevellinesGFFI}
\[\rho^L=-u/\lambda-1,\quad \rho^R=u/\lambda-1.\]
Suppose that $(V^R_t,W_t,V^L_t)_{t\ge 0}$ is the corresponding radial Loewner evolution. 
The continuation threshold of $\gamma_u$ is hit at the following time
\[\tau=\inf\{t>0: V^R_t=W_t=V^L_t\},\]
which is almost surely finite. Let $U_{\tau}$ be the connected component of $\U\setminus \gamma_u[0,\tau]$ that contains $z$, and let $L_u$ be the oriented boundary of $U_{\tau}$. We know that $L_u$ is independent of the choice of $x$. We call $L_u$ the \textbf{level loop of $h$ with height $u$} starting from $\partial\U$ targeted at $z$. The loop $L_u$ is almost surely determined by $h$ and 
\[\PP[L_u\text{ is clockwise}]=\frac{\lambda+u}{2\lambda},\quad \PP[L_u\text{ is counterclockwise}]=\frac{\lambda-u}{2\lambda}.\]

We recall several results from \cite{WangWuLevellinesGFFI} that we will use later in the current paper.

\begin{lemma}\label{lem::levellinesI_remark2515}
Suppose that $h$ is a $\GFF$ on $\U$ with piecewise constant boundary data which changes a finite number of times. 
Let $\gamma_u$ be the level line of $h$ with height $u\in\R$ starting from $x\in\partial\U$ targeted at $y\in\partial\U$. 
\begin{enumerate}
\item [(1)] Suppose that $h$ has boundary value $c_R$ to the right of $x$ and $c_L$ to the left of $x$. To have non-trivial $\gamma_u$ (i.e. $\gamma_u$ has strictly positive length), we must have
\[c_L+u<\lambda,\quad \text{and }\quad c_R+u>-\lambda.\]
\item [(2)] Suppose that $h$ has boundary value $c_R$ to the right of $y$ and $c_L$ to the left of $y$. Then the probability of $\gamma_u$ to reach $y$ is zero if one of the following two conditions holds.
\begin{itemize}
\item $c_L+u\ge\lambda$ and $c_R+u\ge\lambda$;
\item $c_L+u\le -\lambda$ and $c_R+u\le -\lambda$.
\end{itemize}
\item [(3)] Suppose that $h$ has boundary value $c$ on some open interval $I$ which does neither contain $x$ nor $y$. If $\gamma_u$ hits $I$ with strictly positive probability, then $c+u\in (-\lambda, \lambda)$.
\end{enumerate}
\end{lemma}
\begin{proof}
\cite[Remark 2.5.15]{WangWuLevellinesGFFI}.
\end{proof}

\begin{lemma}\label{lem::levellinesI_monotonicity}
Suppose that $h$ is a $\GFF$ on $\HH$ whose boundary value is piecewise constant, changes only finitely many times. For each $u\in\R$ and $x\in\partial\HH$, let $\gamma_u^x$ be the level line of $h$ with height $u$ starting from $x$ targeted at infinity. Fix $x_2\le x_1$.
\begin{enumerate}
\item [(1)] If $u_2>u_1$, then $\gamma_{u_2}^{x_2}$ almost surely stays to the left of $\gamma_{u_1}^{x_1}$.
\item [(2)] If $u_2=u_1$, then $\gamma_{u_2}^{x_2}$ may intersect $\gamma_{u_1}^{x_1}$ and, upon intersecting, the two curves merge and never separate.
\end{enumerate}

\end{lemma}

\begin{proof}
\cite[Theorem 1.1.4]{WangWuLevellinesGFFI}.
\end{proof}

\begin{lemma}\label{lem::levellinesI_levelloop}
Suppose that $h$ is a zero-boundary $\GFF$ on $\U$. Fix a height $u\in (-\lambda,\lambda)$ and a target point $z\in\U$. Let $L_u$ be the level loop of $h$ starting from the boundary $\partial\U$ targeted at $z$. Then we have the following.
\begin{enumerate}
\item [(1)] $L_u$ is oriented either clockwise or counterclockwise and is homeomorphic to the unit disc.
\item [(2)] $L_u\cap\partial\U\neq\emptyset$.
\item [(3)] Given $L_u$, the conditional law of $h$ restricted to each connected component of $\U\setminus L_u$ is the same as $\GFF$'s whose boundary value is zero on $\partial\U$, is $\lambda-u$ to the right of $L_u$, and is $-\lambda-u$ to the left of $L_u$.
\end{enumerate}
Moreover, the loop $L_u$ is almost surely determined by $h$; and any oriented loop in $\U$ that satisfies the above three properties almost surely coincides with $L_u$.
\end{lemma}

\begin{proof}
\cite[Lemmas 3.3.5 and 3.3.6]{WangWuLevellinesGFFI}.
\end{proof}

Suppose that $h$ is a zero-boundary $\GFF$ on $\U$. Fix $r\in (0,1)$ and a target point $z\in \U$. We will introduce the upward height-varying sequence of level loops. Set
\[u_k=-\lambda+kr\lambda,\quad \text{for } k\ge 1.\]
Let $L_1$ be the level loop of $h$ with height $u_1$. Then, given $L_1$, the conditional mean $m_1$ of $h$ restricted to $\inte(L_1)$ is
\[\begin{cases}
2\lambda-r\lambda,&\text{if } L_1 \text{ is clockwise};\\
-r\lambda,&\text{if } L_1 \text{ is counterclockwise}.
\end{cases}\] 
Moreover,
\[\PP[m_1=2\lambda-r\lambda]=r/2,\quad \PP[m_1=-r\lambda]=1-r/2.\]
If $L_1$ is clockwise, we stop and set $N=1$; if not, we continue. Generally, given $(L_1,...,L_k)$ and $L_k$ is counterclockwise, let $L_{k+1}$ be the level loop of $h$ restricted to $\inte(L_k)$ with height $u_{k+1}$. If $L_{k+1}$ is clockwise, we stop and set $N=k+1$; if not, we continue. At each step, we have chance $r/2$ to stop; thus we will stop at some finite time $N$ which we call \textbf{transition step}. We call the sequence $(L_1,...,L_N)$ \textbf{upward height-varying sequence of level loops} with height difference $r\lambda$ starting from $L_0=\partial\U$ targeted at $z$. We summarize some basic properties of this sequence in the following.
\begin{itemize}
\item [(a)] The transition step $N$ satisfies the geometric distribution
\[\PP[N>n]=(1-r/2)^n,\quad \forall n\ge 0.\]
\item [(b)] For $1\le k< N$, the loop $L_k$ is counterclockwise; the loop $L_N$ is clockwise.
\item [(c)] For $1\le k\le N$, let $m_k$ be the conditional mean of $h$ restricted to $\inte(L_k)$ given $(L_1,...,L_k)$, we have
\[m_k=-kr\lambda, \text{ for }1\le k<N; \quad m_N=2\lambda-Nr\lambda.\]
\end{itemize} 

Symmetrically, we could define \textbf{downward height-varying sequence of level loops} of $h$. For the downward height-varying sequence $(L_1,...,L_N)$, the transition step $N$ satisfies the geometric distribution. The loops $L_1,..., L_{N-1}$ are clockwise and the loop $L_N$ is counterclockwise. Let $m_k$ be the conditional mean of $h$ restricted to $\inte(L_k)$ given $(L_1,...,L_k)$, we have
\[m_k=kr\lambda,\text{ for }1\le k<N;\quad m_N=Nr\lambda-2\lambda.\]

Next, we will introduce alternating height-varying sequence of level loops. Suppose that $h$ is a zero-boundary $\GFF$ on $\U$. Fix $r\in (0,1)$ and a target point $z\in\U$. Assume that $L_0=\partial \U$ is oriented counterclockwise. We start by the upward height-varying sequence of level loops of $h$. Let $N_1$ be the transition step and denote sequence by $(L_1,..., L_{N_1})$. Let $m_{N_1}$ be the conditional mean of $h$ restricted to $\inte(L_{N_1})$ given $(L_1,...,L_{N_1})$, then
\[m_{N_1}=2\lambda-N_1r\lambda.\]
We continue the sequence by the downward height-varying sequence of level loops of $h$ restricted to $\inte(L_{N_1})$. Let $N_2$ be the its transition step and denote the sequence by $(L_{N_1+1},...,L_{N_1+N_2})$. Let $m_{N_1+N_2}$ be the conditional mean of $h$ restricted to $\inte(L_{N_1+N_2})$ given $(L_1,L_2,...,L_{N_1+N_2})$, then 
\[m_{N_1+N_2}=(N_2-N_1)r\lambda.\]
Set $M_1=N_1+N_2$. 

Generally, given $(L_1,L_2,..., L_{M_k})$ for some $k\ge 1$, we have 
\[m_{M_k}=\sum_{j=1}^{2k}(-1)^j N_j r\lambda.\]
We continue the sequence by the upward height-varying sequence $(L_{M_k+1},...,L_{M_k+N_{2k+1}})$ where $N_{2k+1}$ is its transition step. We then continue the sequence by the downward height-varying sequence $(L_{M_k+N_{2k+1}+1},...,L_{M_k+N_{2k+1}+N_{2k+2}})$ where $N_{2k+2}$ is its transition step. Set
\[M_{k+1}=M_k+N_{2k+1}+N_{2k+2},\] and let $m_{M_{k+1}}$ be the conditional mean of $h$ restricted to $\inte(L_{M_{k+1}})$ given $(L_1,..,L_{M_{k+1}})$, we have 
\[m_{M_{k+1}}=\sum_{j=1}^{2k+2}(-1)^j N_j r\lambda.\] 

In this way, we obtain an infinite sequence $(L_n,n\ge 0)$ which we call \textbf{alternating height-varying sequence of level loops} of $h$ with height difference $r\lambda$ starting from $L_0=\partial\U$ (counterclockwise) targeted at $z$.

If we orient $L_0=\partial\U$ clockwise, we can define alternating height-varying sequence similarly, the only difference is that we start by downward height-varying sequence. 

\begin{proposition}\label{prop::levelloops_inside_deterministic}
The alternating height-varying sequence of level loops of $\GFF$ is almost surely determined by the field. 
\end{proposition}

Note that, in the alternating height-varying sequence of level loops with height difference $r\lambda$, each loop has the height of the form 
\[-\lambda+nr\lambda,\text{ for some }n\in\Z.\]
The relation between the heights of two neighbor loops is explained in Figure \ref{fig::levelloops_boundary_inside}. 

\begin{figure}[ht!]
\begin{subfigure}[b]{0.48\textwidth}
\begin{center}
\includegraphics[width=0.625\textwidth]{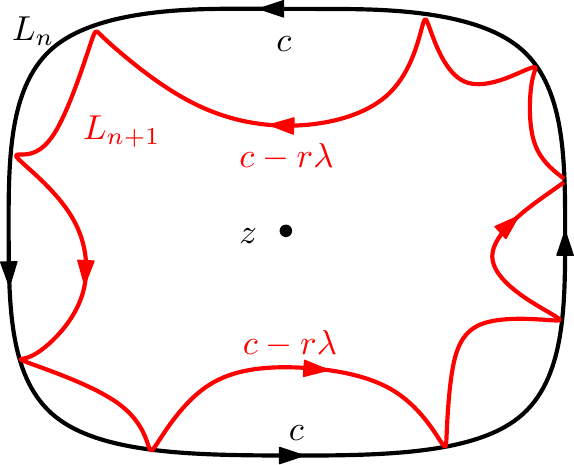}
\end{center}
\caption{If $L_{n+1}$ is counterclockwise, it has boundary value $c-r\lambda$ to the left-side.}
\end{subfigure}
$\quad$
\begin{subfigure}[b]{0.48\textwidth}
\begin{center}\includegraphics[width=0.62\textwidth]{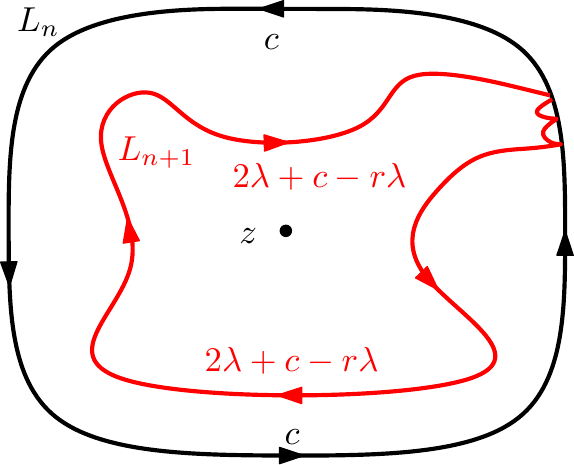}
\end{center}
\caption{If $L_{n+1}$ is clockwise, it has boundary value $2\lambda+c-r\lambda$ to the right-side.}
\end{subfigure}
\begin{subfigure}[b]{0.48\textwidth}
\begin{center}
\includegraphics[width=0.625\textwidth]{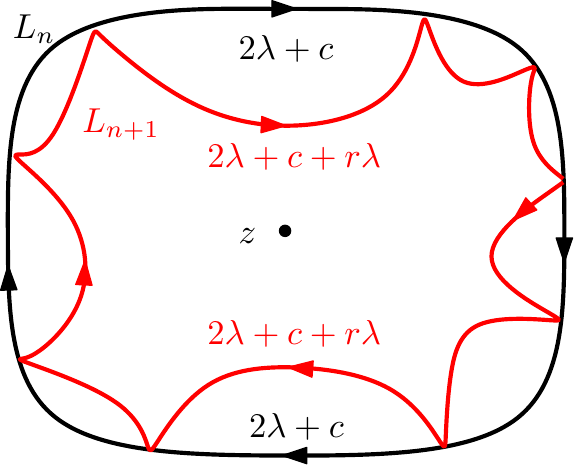}
\end{center}
\caption{If $L_{n+1}$ is clockwise, it has boundary value $2\lambda+c+r\lambda$ to the right-side.}
\end{subfigure}
$\quad$
\begin{subfigure}[b]{0.48\textwidth}
\begin{center}\includegraphics[width=0.62\textwidth]{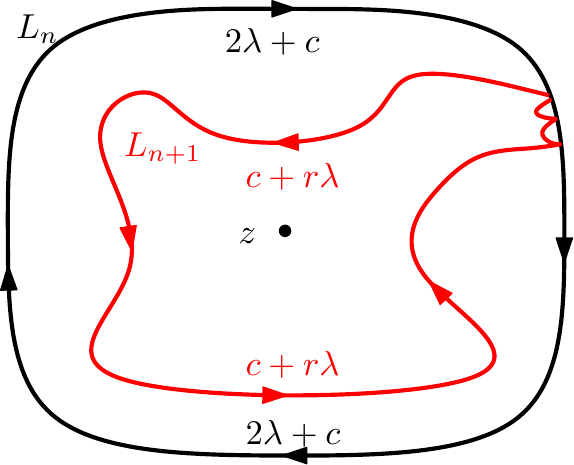}
\end{center}
\caption{If $L_{n+1}$ is counterclockwise, it has boundary value $c+r\lambda$ to the left-side.}
\end{subfigure}
\caption{\label{fig::levelloops_boundary_inside} Explanation of the relation between the heights of two neighbor loops.  Assume that the boundary value of the loop $L_n$ is $c$ to the left-side and is $2\lambda+c$ to the right-side. Note that, if the height of the loop $L_n$ is of the form $-\lambda+nr\lambda$ for some $n\in\Z$, the height of the loop $L_{n+1}$ is $-\lambda+(n\pm 1)r\lambda$.}
\end{figure}

From the construction, we know that the alternating height-varying sequence $(L_n, n\ge 0)$ satisfies the following domain Markov property.
\begin{proposition}\label{prop::levelloops_inside_markov}
Suppose that $(L_n, n\ge 0)$ is an alternating height-varying sequence of level loops. For any stopping time $N$, given $(L_n, n\le N)$, let $g_N$ be the conformal map from $\inte(L_N)$ onto $\U$ such that $g_N(z)=z, g_N'(z)>0$, then the conditional law of the sequence $(g_N(L_{N+n}), n\ge 0)$ is the same as alternating height-varying sequence of level loops. 
\end{proposition}

\begin{proposition}\label{prop::levelloops_inside_transience}
The alternating height-varying sequence of level loops $(L_n, n\ge 0)$ starting from $L_0=\partial\U$ targeted at $z\in\U$ is transient: the loop $L_n$ converges to $\{z\}$ in Hausdorff metric as $n\to\infty$.
\end{proposition}

\begin{proof}
We may assume that the target point is the origin and that $L_0$ is counterclockwise. Suppose that the height difference is $r\lambda$ for $r\in (0,1)$. 
\medbreak 
\textit{First}, we show that there exist constants $n_0, \eta\in(0,1)$, and $p\in (0,1)$ such that
\[\PP\left[L_{n_0}\subset \eta\U\right]\ge p.\]
Set $n_0=\lceil 2/r\rceil$. Let $N_1$ be the first transition step of the sequence $(L_n, n\ge 0)$. On the event $\{N_1\ge n_0\}$, the level loop $L_{n_0}$ is of height $-\lambda+n_0r\lambda\ge \lambda$, thus $L_{n_0}\cap \partial\U=\emptyset$ (by Lemma \ref{lem::levellinesI_remark2515}). Therefore,
\[\PP[L_{n_0}\cap \partial\U=\emptyset]\ge\PP[N_1>n_0]=\left(1-r/2\right)^{n_0}.\] 
Thus, there exists $\eta\in (0,1)$ such that
\[\PP[L_{n_0}\subset \eta\U]\ge \left(1-r/2\right)^{n_0}/2.\]
\medbreak
\textit{Next}, we show the transience. For $k\ge 0$, let $\phi_k$ be the conformal map from $\inte(L_{kn_0})$ onto $\U$ such that $\phi_k(0)=0, \phi_k'(0)>0$. From the domain Markov property of the sequence $(L_n, n\ge 0)$, we know that $(\phi_k(L_{(k+1)n_0}), k\ge 0)$ is an i.i.d. sequence. By the first step, we have
\[\sum_k \PP\left[\phi_k(L_{(k+1)n_0})\subset \eta\U\right]=\infty.\]
By Borel-Cantelli Lemma, almost surely, there exists a sequence $k_j\to\infty$ such that 
\[\left\{\phi_{k_j}(L_{(k_j+1)n_0})\subset \eta\U\right\} \text{  holds for all } j\ge 0.\]
In particular,
\[\left\{\phi_{k_j}(L_{(k_{j+1})n_0})\subset \eta\U\right\} \text{  holds for all } j\ge 0.\]
Denote $L_{k_jn_0}$ by $l_j$. We only need to show that the sequence $(l_j,j\ge 1)$ is transient. We will show $l_j\subset \eta^j \U$ by induction on $j$. 

For $j=1$, on the one hand, we have $\phi_{k_0}(l_1)\subset \eta\U$; on the other hand, since $|\phi_{k_0}(z)|\ge |z|$ for all $z$, we have $\phi_{k_0}^{-1}(\eta\U)\subset\eta\U$.
Combining these two facts, we have that $l_1\subset\phi_{k_0}^{-1}(\eta\U)\subset\eta\U$.

Assume that the claim holds for $j$, let us consider $l_{j+1}$. We know that  $\phi_{k_j}(l_{j+1})\subset \eta\U$, it is then sufficient to argue that $\phi_{k_j}^{-1}(\eta\U)\subset\eta^{j+1}\U$. Note that $\phi_{k_j}$ is the conformal map from $\inte(l_j)$ onto $\U$ with $\phi_{k_j}(0)=0, \phi_{k_j}'(0)>0$.  Let $\varphi_1(z)=z/\eta^j$. Since $l_j\subset \eta^j\U$, the set $\varphi_1(\inte(l_j))$ is still contained in $\U$. Let $\varphi_2$ be the conformal map from $\varphi_1(\inte(l_j))$ onto $\U$ with $\varphi_2(0)=0,\varphi_2'(0)>0$. Then $\phi_{k_j}=\varphi_2\circ\varphi_1$. Since $|\varphi_2(z)|\ge |z|$ for all $z$, we have $\varphi_2^{-1}(\eta\U)\subset\eta\U$. Thus
\[\phi_{k_j}^{-1}(\eta\U)=\varphi_1^{-1}\varphi_2^{-1}(\eta\U)\subset\varphi_1^{-1}(\eta\U)=\eta^{j+1}\U.\]
This completes the proof. 
\end{proof}

The following proposition describes the target-independence of the sequence of level loops.

\begin{proposition}\label{prop::levellinesI_targetindependence}
Fix $r\in (0,1)$. Suppose that $h$ is a zero-boundary $\GFF$ on $\U$. Fix two points $w_1,w_2\in\U$. For $i=1,2$, let $(L_n^{w_i}, n\ge 0)$ be the alternating height-varying sequence of level loops of $h$ with height difference $r\lambda$ starting from $\partial\U$ targeted at $w_i$; and denote by $U^{w_i}_n$ the connected component of $\U\setminus L^{w_i}_n$ that contains $w_i$. Then there exists a number $M$ such that
\[L^{w_1}_n=L^{w_2}_n\quad\text{for }n\le M-1;\quad U^{w_1}_n\cap U^{w_2}_n=\emptyset,\quad\text{for }n=M.\]
Given $(L_n^{w_1}, L_n^{w_2}, n\le M)$, the two sequences continue towards their target points respectively in a conditionally independent way. 
\end{proposition}
\begin{proof}
\cite[Section 3.5]{WangWuLevellinesGFFI}.
\end{proof}

Suppose that $(\tilde{L}_n, n\ge 0)$ is an alternating height-varying sequence of level loops starting from $\tilde{L}_0=\partial\U$ targeted at $z\in\U$. We want to change the index of the loops so that they are indexed by minus the log of the conformal radius seen from $z$. Namely, we construct the sequence $(L_t, t\ge 0)$ so that 
\[L_t=\tilde{L}_n,\quad\text{if }-\log\CR(\inte(\tilde{L}_n);z)\le t<-\log\CR(\inte(\tilde{L}_{n+1});z).\]
We call the obtained sequence $(L_t, t\ge 0)$ \textbf{alternating height-varying sequence of level loops indexed by minus the log conformal radius} seen from $z$.

Next, we will introduce a whole-plane version of alternating height-varying sequence of level loops. Before this, we briefly recall the conformal radius seen from infinity. When $D\subset\C$ is a simply connected compact set, its complement $\C\setminus D$ is a simply connected domain in the Riemann sphere that contains $\infty$. Let $\phi:\C\setminus\overline{\U}\to\C\setminus D$ be the unique bijective conformal map with $\phi(\infty)=\infty$ and the expansion at $\infty$ of the following form
\[\phi(z)=c_1z+c_0+c_{-1}z^{-1}+\cdots,\quad c_1>0.\]
We call $c_1$ the derivative of $\phi$ at $\infty$, denoted by $\phi'(\infty)$; call $1/c_1$ the conformal radius of $\C\setminus D$ seen from $\infty$, denoted by $\CR(\C\setminus D;\infty)$. Note that, for $\eps>0$,
\[\CR\left(\C\setminus (\eps\U);\infty\right)=\CR\left(\frac{1}{\eps}\U;0\right)=\frac{1}{\eps}.\]

For each $\eps>0$, let $h_{\eps}$ be a $\GFF$ on $\C_{\eps}=\C\setminus(\eps\U)$ with zero-boundary value. The boundary $\partial (\eps\U)$ is oriented either clockwise or counterclockwise. We can uniquely generate the alternating height-varying sequence of level loops of $h_{\eps}$ starting from $\partial(\eps\U)$ targeted at $\infty$: $(L_t^{\eps},t\ge \log\eps)$, where the loops are indexed by minus log conformal radius seen from $\infty$. Namely,
\[-\log\CR(\ext(L_t^{\eps});\infty)=t,\quad \text{for }t\ge\log\eps.\]
For each $t\ge \log\eps$, let $g_t^{\eps}$ be the conformal map from $\ext(L_t^{\eps})$ onto $\C\setminus\U$ with $g_t^{\eps}(\infty)=\infty, (g_t^{\eps})'(\infty)>0$. The following lemma explains that the family $(g_t^{\eps},t\ge \log\eps)$ is convergent. 

\begin{lemma}\label{lem::levelloops_outside_convergent}
The family of conformal maps $(g_t^{\eps}, t\ge\log\eps)$ converges weakly to some limit $(g_t, t\in\R)$ with respect to the topology of local uniform convergence. Namely, for every $\delta>0$ and every compact interval $[a,b]$, the map $g_t^{\eps}$ converges weakly to $g_t$ uniformly on $[a,b]\times\{w\in\C: \dist(w,\inte(L_b))\ge\delta\}$ where $L_t$ is the loop such that $g_t$ is the conformal map from $\ext(L_t)$ onto $\C\setminus\U$ with $g_t(\infty)=\infty, g_t'(\infty)>0$.
\end{lemma}

\begin{proof}
We will show that there exist universal finite constants $C, c>0$ such that the following is true. For any $\delta_1>\delta_2>0$ small, there exists a coupling between $(L_t^{\delta_1},t\ge\log\delta_1)$ and $(L_t^{\delta_2},t\ge\log\delta_2)$ such that
\[\PP\left[(L_t^{\delta_1}, t\ge 0)\neq(L_t^{\delta_2}, t\ge 0)\right]\le C\delta_1^c.\]
This will imply the conclusion.

For any $\eps>0$ small, consider the sequence $(L_t^{\eps}, t\ge\log\eps)$. By the first step in the proof of Proposition \ref{prop::levelloops_inside_transience}, we know that there exist universal constants $T_0,\eta>0$ so that the probability of the event 
$\{\dist(L^{\eps}_0,L^{\eps}_{T_0})\}\ge\eta]$ is bounded uniformly from below. Therefore, given $(L^{\delta_1}_t, L^{\delta_2}_t)$, we can couple $L^{\delta_1}_{t+T_0}$ and $L^{\delta_2}_{t+T_0}$ so that the probability of the event $\{L^{\delta_1}_{t+T_0}=L^{\delta_2}_{t+T_0}\}$ is bounded uniformly from below. 

By the domain Markov property in Proposition \ref{prop::levelloops_inside_markov} and the scaling invariance, there exists a coupling between $L_0^{\delta_1}$ and $L_0^{\delta_2}$ so that 
\[\PP\left[L_0^{\delta_1}\neq L_0^{\delta_2}\right]\le C\delta_1^c,\]
where $C$ is some finite universal constant. On the event $\{L_0^{\delta_1}=L_0^{\delta_2}\}$, we couple the processes to be identical for $t\ge 0$. This completes the proof.
\end{proof}

Suppose that $(g_t, t\in\R)$ is the limit family of conformal maps in Lemma \ref{lem::levelloops_outside_convergent} and that $(L_t, t\in\R)$ is the corresponding sequence of loops, i.e. $g_t$ is the conformal map from $\ext(L_t)$ onto $\C\setminus\U$ with $g_t(\infty)=\infty, g_t'(\infty)>0$. Note that the loops in the sequence $(L_t,t\in\R)$ are indexed by minus the log of conformal radius seen from $\infty$:
\[-\log\CR(\ext(L_t);\infty)=t,\quad \text{for all } t\in\R.\]
We also say that the alternating height-varying sequence of level loops $(L_t^{\eps}, t\ge \log\eps)$ converges weakly to $(L_t,t\in\R)$ with respect to the topology of local uniform convergence. Clearly, the sequence $(L_t, t\in\R)$ is a transient sequence of adjacent simple loops disconnecting the origin from infinity by Proposition \ref{prop::levelloops_inside_transience}. We now are ready to complete the proof of Theorem \ref{thm::interior_levelloops_coupling}.

\begin{proof}[Proof of Theorem \ref{thm::interior_levelloops_coupling}]
\textit{First}, we construct a coupling between whole-plane $\GFF$ and a transient sequence of adjacent simple loops. 
For each $\eps>0$, let $h_{\eps}$ be a $\GFF$ on $\C_{\eps}=\C\setminus \eps\U$ with zero-boundary value. Let $(L^{\eps}_t, t\ge\log\eps)$ be the alternating height-varying sequence of level loops of $h_{\eps}$ starting from $\partial(\eps\U)$ targeted at $\infty$ indexed by minus the log conformal radius seen from $\infty$. Let $(g^{\eps}_t, t\ge \log\eps)$ be the corresponding sequence of conformal maps: $g_t^{\eps}$ is the conformal map from $\ext(L_t^{\eps})$ onto $\C\setminus\U$ with $g_t^{\eps}(\infty)=\infty, (g_t^{\eps})'(\infty)>0$. On the one hand, from Lemma \ref{lem::levelloops_outside_convergent}, we know that the sequence $(L_t^{\eps}, t\ge \log\eps)$ converges weakly to $(L_t,t\in\R)$ and the sequence $(g_t^{\eps}, t\ge\log\eps)$ converges weakly to $(g_t, t\in\R)$ with respect to the topology of local uniform convergence. On the other hand, the $\GFF$s $h_{\eps}$, viewed as distributions defined modulo a global additive constant in $r\lambda\Z$, converge to a whole-plane $\GFF$, modulo a global additive constant in $r\lambda\Z$, as $\eps\to 0$ by Proposition \ref{prop::wholeplanegff_convergence_2.10}. 
This gives a coupling $(h, (L_t,  t\in\R))$ of a whole-plane $\GFF$ $h$ and a transient sequence of adjacent simple loops disconnecting the origin from $\infty$.
\medbreak
\textit{Second}, we show that the coupling $(h, (L_t,  t\in\R))$ satisfies the desired domain Markov property. For each $\eps>0$ and $T\ge\log\eps$, given $(L^{\eps}_t, t\le T)$, the conditional law of $h_{\eps}$ restricted to $\ext(L_T^{\eps})$ is the same as a $\GFF$ with boundary value 
\[\begin{cases}
0,&\text{if }L_T^{\eps}\text{ is clockwise},\\
2\lambda, &\text{if }L_T^{\eps}\text{ is counterclockwise},
\end{cases}\]
modulo a global additive constant in $r\lambda\Z$; and the conditional law of $h_{\eps}$ restricted to other connected components of $\C_{\eps}\setminus\cup_{t\le T}L_t^{\eps}$ is the same as $\GFF$s with boundary value given by level loop boundary conditions with height difference $r\lambda$. The convergence of $(h_{\eps}, (L^{\eps}_t,t\ge \log\eps))$ to $(h, (L_t, t\in\R))$ implies that, given $(L_t, t\le T)$, the conditional law of $h$ restricted to $\C\setminus\cup_{t\le T}L_t$ is the same as $\GFF$s with boundary values given by level loop boundary conditions with height difference $r\lambda$. This implies that the coupling $(h, (L_t,  t\in\R))$ satisfies the desired domain Markov property at deterministic times $T\in\R$. 

For general stopping time $\tau$, suppose that $T\in\R$ is deterministic and $\tau\ge T$. The conditional law of $h$ given $(L_t, t\le T)$ is the same as $\GFF$ in $\ext(L_T)$. Let $(\tilde{L}_t, t\ge T)$ be the alternating height-varying sequence of level loops of $h$ restricted to $\ext(L_T)$ starting from $L_T$ targeted at $\infty$. 
\begin{enumerate}
\item [(a)] By Proposition \ref{prop::levelloops_inside_deterministic}, we know that the sequence $(\tilde{L}_t, t\ge T)$ coincides with the sequence $(L_t, t\ge T)$ almost surely. 
\item [(b)] By Proposition \ref{prop::levelloops_inside_markov}, we know that the sequence $(\tilde{L}_t, t\ge T)$ satisfies the desired domain Markov property for any stopping time.
\end{enumerate}
Combining these two facts, the coupling $(h,(L_t, t\in\R))$ satisfies the desired domain Markov property for the stopping time $\tau\ge T$. Since $T$ is arbitrary, the domain Markov property holds for any stopping time $\tau$.
\medbreak
\textit{Finally}, we prove the conclusion for general domain $D$. Since the law of $h$ (modulo a global additive constant in $r\lambda\Z$) in a small neighborhood of the origin changes in an absolutely continuous way when we replace $\C$ with $D$. We can couple the sequence of loops with the field as if the domain were $\C$ up until the first time that the sequence of the loops exits the small neighborhood. Then the usual boundary emanating level lines allow us to extend the sequence of level loops uniquely by Proposition \ref{prop::wholeplanegff_localset_2.16}. 
\end{proof}

In the coupling $(h, (L_n, n\in\Z))$ as in Theorem \ref{thm::interior_levelloops_coupling}, we call the sequence $(L_n,n\in\Z)$ the \textbf{alternating height-varying sequence of level loops starting from the origin targeted at infinity} (resp. starting from $z$ targeted at $\partial D$) of $h$ with height difference $r\lambda$ when $D=\C$ (resp. when $D\subsetneq\C$).

\subsection{Interaction}
\label{subsec::interaction}
In this section, we will discuss the interaction behavoirs.
\begin{itemize}
\item The interaction between sequence of level loops starting from interior point and boundary emanating level lines: Proposition \ref{prop::interaction_interior_boundary}.
\item The interaction between two sequences of level loops starting from distinct points but targeted at a common point: Proposition \ref{prop::interaction_interior_interior} and Theorem \ref{thm::interior_levelloops_interacting_commontarget}.
\item The interaction between two sequences of level loops starting from a common point but targeted at distinct points: Theorem \ref{thm::interior_levelloops_interacting_commonstart}.
\end{itemize}
\begin{figure}[ht!]
\begin{subfigure}[b]{0.48\textwidth}
\begin{center}
\includegraphics[width=0.8\textwidth]{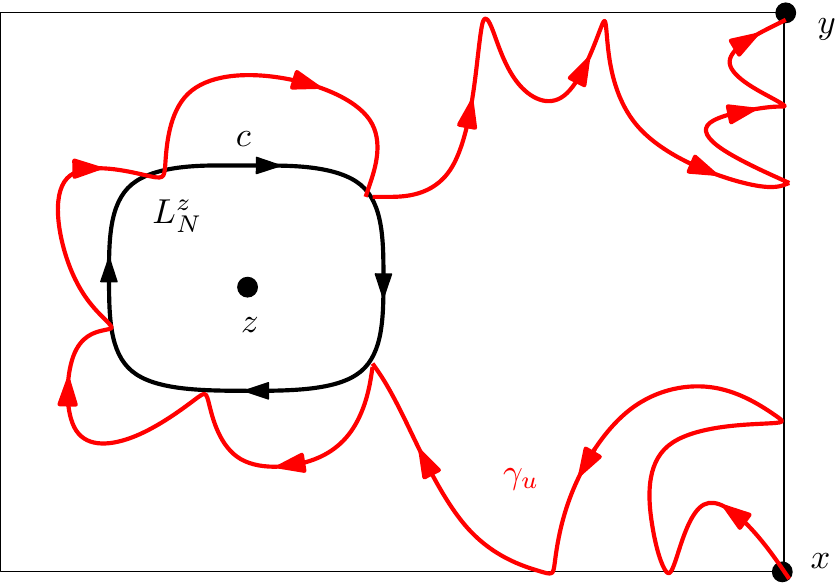}
\end{center}
\caption{The loop $L^z_N$ is clockwise and the level line $\gamma_u$ is targeted at a boundary point.}
\end{subfigure}
$\quad$
\begin{subfigure}[b]{0.48\textwidth}
\begin{center}
\includegraphics[width=0.8\textwidth]{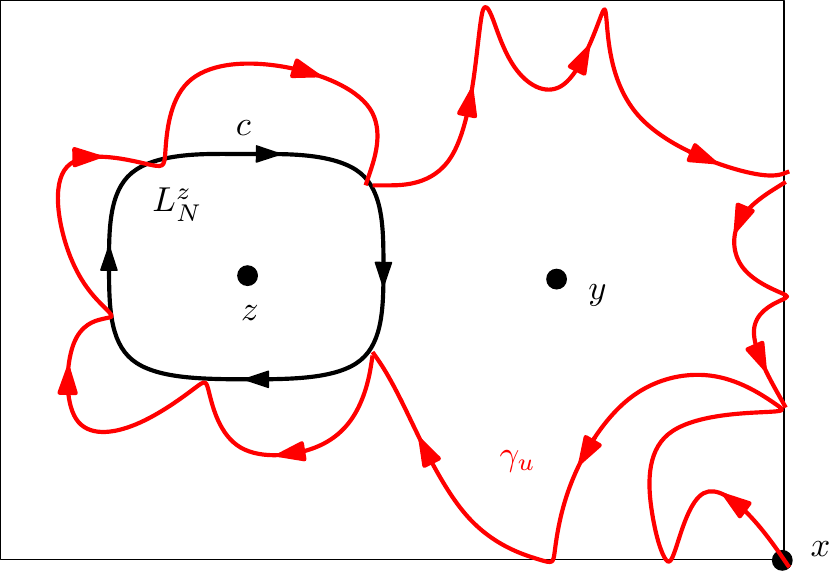}
\end{center}
\caption{The loop $L^z_N$ is clockwise and the level line $\gamma_u$ is targeted at an interior point.}
\end{subfigure}
\begin{subfigure}[b]{0.48\textwidth}
\begin{center}
\includegraphics[width=0.8\textwidth]{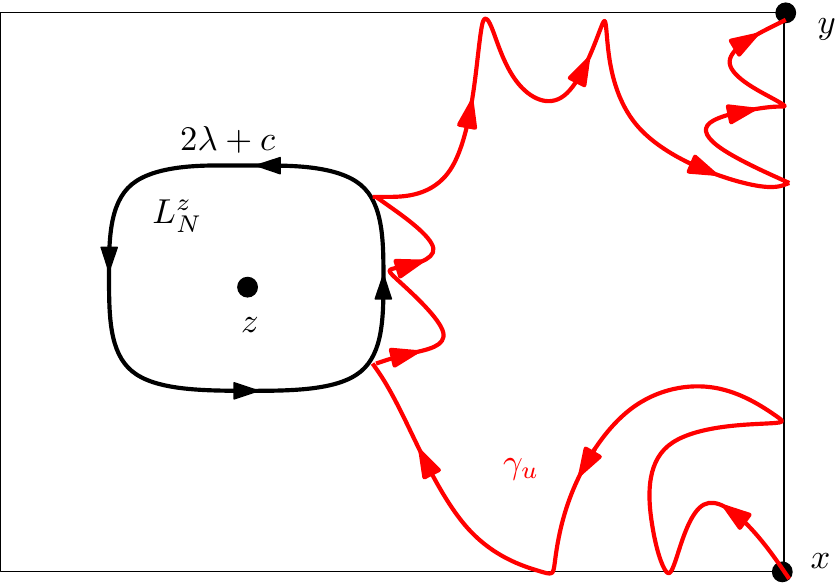}
\end{center}
\caption{The loop $L^z_N$ is counterclockwise and the level line $\gamma_u$ is targeted at a boundary point.}
\end{subfigure}
$\quad$
\begin{subfigure}[b]{0.48\textwidth} \begin{center}
\includegraphics[width=0.8\textwidth]{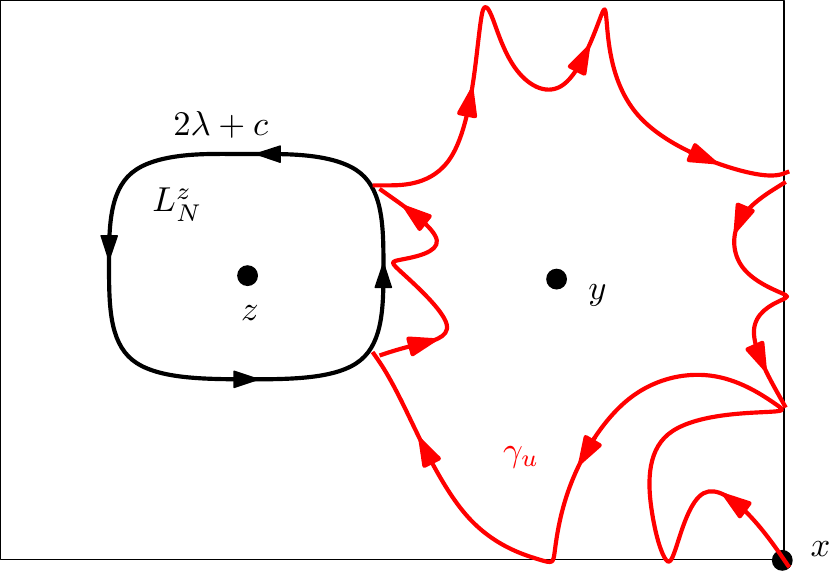}
\end{center}
\caption{The loop $L^z_N$ is counterclockwise and the level line $\gamma_u$ is targeted at an interior point.}
\end{subfigure}
\caption{\label{fig::interaction_interior_boundary} The interaction between a sequence of level loops starting from interior and a boundary emanating level line.}
\end{figure}

\begin{proposition}\label{prop::interaction_interior_boundary}
Fix $r\in (0,1)$ and a domain $D\subset\C$ with harmonically non-trivial boundary. Fix three points $z\in D, x\in\partial D$ and $y\in\overline{D}$. Let $h$ be a $\GFF$ on $D$, let $(L_n^z,n\in\Z)$ be the alternating height-varying sequence of level loops of $h$ with height difference $r\lambda$ starting from $z$ targeted at $\partial D$, and let $\gamma_u^{x\to y}$ be the level line of $h$ with height $u\in\R$ starting from $x$ targeted at $y$. For any $(L^z_n, n\in\Z)$-stopping time $N$, assume that the loop $L^z_N$ does not hit $\partial D$ and $L^z_N$ has boundary value $c$ to the left-side and $2\lambda+c$ to the right-side. On the event $\{\gamma_u^{x\to y}\text{ hits }L^z_N\}$, we have that
\[\begin{cases}
2\lambda+c+u\in (-\lambda,\lambda),&\text{if }L^z_N\text{ is counterclockwise},\\
c+u\in (-\lambda,\lambda),&\text{if }L^z_N\text{ is clockwise}.
\end{cases}\]
Moreover, the level line $\gamma_u^{x\to y}$ stays outside of $L^z_N$. See Figure \ref{fig::interaction_interior_boundary}.
\end{proposition}

\begin{proof}
We may assume that $L^z_N$ is clockwise, and the case that $L^z_N$ is counterclockwise can be proved similarly. 
\medbreak
\textit{First}, we show that $c+u\in(-\lambda,\lambda)$.
Given $(L^z_n, n\le N)$, denote by $\tilde{h}$ the field of $h$ restricted to $D\cap\ext(L^z_N)$. We know that the conditional law of $\tilde{h}$ is the same as $\GFF$ with boundary value $c$ on $L^z_N$. Moreover, the curve $\gamma_u^{x\to y}$ is also the level line of $\tilde{h}$ with height $u$. By Lemma \ref{lem::levellinesI_remark2515} Item (3), we have that $c+u\in (-\lambda, \lambda)$.
\medbreak
\textit{Next}, we show that $\gamma_u^{x\to y}$ stays outside of $L^z_N$. We prove by contradiction. Assume that $\gamma_u^{x\to y}$ has non-trivial pieces in $\inte(L^z_N)$. Given $(L^z_n, n\le N)$, the conditional law of $h$, restricted to the connected components of $D\setminus \cup_{n\le N}L^z_n$ that are in $\inte(L^z_N)$, is the same as $\GFF$ whose boundary values are $2\lambda+c$ along $L^z_N$. For these $\GFF$s, to have non-trivial level lines with height $u$ starting from $L^z_N$, we must have $2\lambda+c+u\in (-\lambda, \lambda)$ which contradicts with the fact that $c+u\in (-\lambda, \lambda)$. 
\end{proof}

\begin{figure}[ht!]
\begin{subfigure}[b]{0.48\textwidth}
\begin{center}
\includegraphics[width=0.8\textwidth]{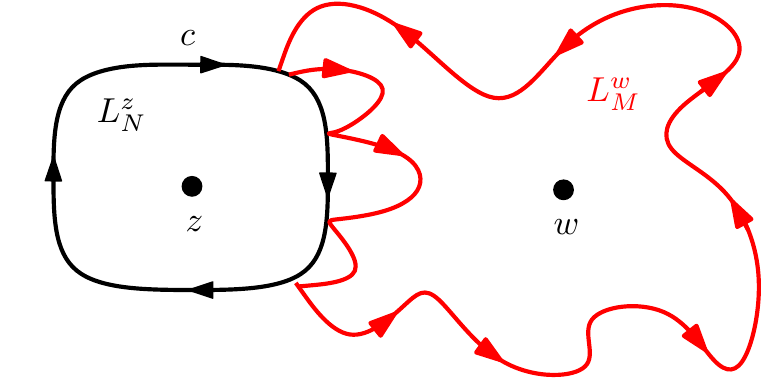}
\end{center}
\caption{The loop $L^z_N$ is clockwise and the loop $L^w_M$ is counterclockwise.}
\end{subfigure}
$\quad$
\begin{subfigure}[b]{0.48\textwidth}
\begin{center}
\includegraphics[width=0.8\textwidth]{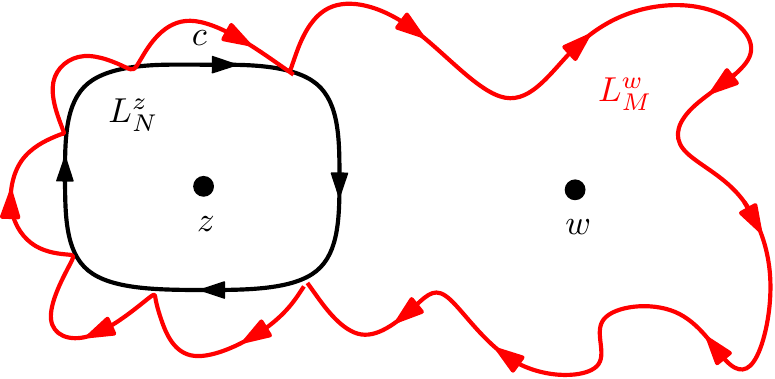}
\end{center}
\caption{The loop $L^z_N$ is clockwise and the loop $L^w_M$ is clockwise.}
\end{subfigure}
\caption{\label{fig::interaction_interior_interior} The interaction between two sequences of level loops starting from distinct interior points.}
\end{figure}

\begin{proposition}\label{prop::interaction_interior_interior}
Fix $r\in(0,1)$. Let $h$ be a whole-plane $\GFF$ modulo a global additive constant in $r\lambda\Z$. Fix two starting points $z,w$. Let $(L^z_n, n\in\Z)$ (resp. $(L^w_n, n\in\Z)$) be the alternating height-varying sequence of level loops of $h$ with height difference $r\lambda$ starting from $z$ (resp. starting from $w$) targeted at $\infty$. Let $N$ be any $(L^z_n, n\in\Z)$-stopping time such that $L^z_N$ does not disconnect $w$ from $\infty$. Assume that $L^z_N$ has boundary value $c$ to the left-side and $2\lambda+c$ to the right-side. Let $M$ be any stopping time with respect to the filtration
\[\LF_n=\sigma( L^z_k, k\le N; L^w_m, m\le n).\]
Suppose that $L^w_M$ hits $L^z_N$ and has height $u$. Then
\[\begin{cases}
2\lambda+c+u\in (-\lambda, \lambda), &\text{if }L^z_N\text{ is counterclockwise},\\
c+u\in (-\lambda, \lambda), &\text{if }L^z_N\text{ is clockwise}.
\end{cases}\]
Moreover, the loop $L^w_M$ stays outside of $L^z_N$. See Figure \ref{fig::interaction_interior_interior}.
\end{proposition}

\begin{proof}
We may assume that $L^z_N$ is clockwise, and the case that $L^z_N$ is counterclockwise can be proved similarly. 
\medbreak
\textit{First}, we show that $c+u\in(-\lambda,\lambda)$.
Given $(L^z_n, n\le N)$, denote by $\tilde{h}$ the field of $h$ restricted to $\ext(L^z_N)$. We have the following observations.
\begin{enumerate}
\item [(a)] Given $(L^z_n, n\le N)$, the conditional law of $\tilde{h}$ is the same as $\GFF$ on $\ext(L^z_N)$ with boundary value $c$.
\item [(b)] Given $(L^z_n, n\le N)$, the sequence $(L^w_n, n\in\Z)$ is coupled with $\tilde{h}$ in the same way as in Theorem \ref{thm::interior_levelloops_coupling} for $D=\ext(L^z_N)$, up to the first time that the loop exits $\ext(L^z_N)$.
\end{enumerate}
When $L^w_M$ hits $L^z_N$, the pieces of $L^w_M$ contained in $\ext(L^z_N)$ are level lines of $\tilde{h}$ with height $u$. 
 By Lemma \ref{lem::levellinesI_remark2515} Item (3), we have that $c+u\in (-\lambda, \lambda)$.
\medbreak
\textit{Next}, we show that $L^w_M$ stays outside of $L^z_N$. This can be proved similarly as in the proof of Proposition \ref{prop::interaction_interior_boundary}.
\end{proof}

\begin{proposition}\label{prop::interaction_distinctstart}
Fix $r\in(0,1)$. Let $h$ be a whole-plane $\GFF$ modulo a global additive constant in $r\lambda\Z$. Fix two starting points $z,w$. Let $(L^z_n, n\in\Z)$ (resp. $(L^w_n, n\in\Z)$) be the alternating height-varying sequence of level loops of $h$ with height difference $r\lambda$ starting from $z$ (resp. starting from $w$) targeted at $\infty$. 
Define the stopping times
\[N=\min\{n: L^z_n\text{ disconnects }w\text{ from }\infty\};\quad M=\min\{n: L^w_n\text{ disconnects }z\text{ from }\infty\}.\]
Then, almost surely, the loop $L^z_N$ coincides with $L^w_M$.
\end{proposition}

The proof of Proposition \ref{prop::interaction_distinctstart} needs the following lemma.

\begin{lemma}\label{lem::interaction_distinctstart_auxiliary}
Assume the same notations as in Proposition \ref{prop::interaction_distinctstart}. Let $U^w$ be the connected component of $\C\setminus\cup_{n\le N}L^z_n$ that contains $w$. Let $M$ be any stopping time with respect to the following filtration
\[\LF_n=\sigma(L^z_k, k\le N; L^w_m, m\le n).\] On the event 
\[\{L^w_M\subset\overline{U^w},\quad L^w_M\cap L^z_N=\emptyset\},\]
we have almost surely that the loop $L^w_{M+1}$ is contained in the closure of $\inte(L^z_N)$.
\end{lemma}

\begin{figure}[ht!]
\begin{subfigure}[b]{0.48\textwidth}
\begin{center}
\includegraphics[width=0.8\textwidth]{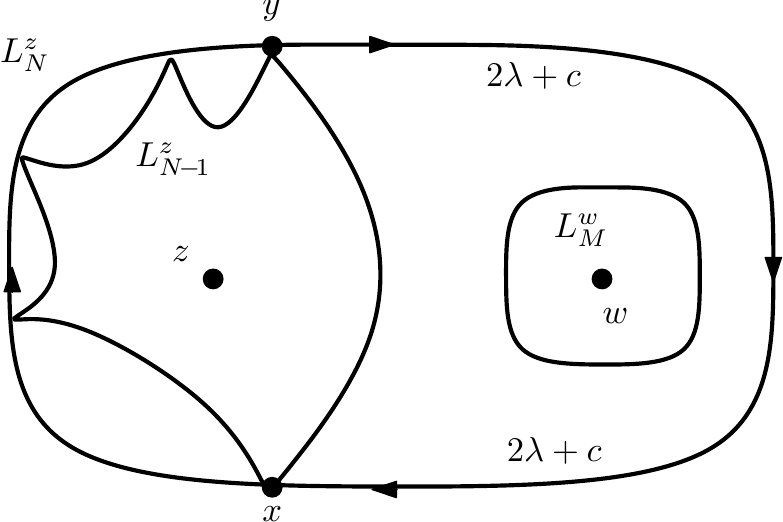}
\end{center}
\caption{The boundary $\partial U^w$ has two special points $x$ and $y$ such that the clockwise part of $\partial U^w$ from $x$ to $y$ is part of $L^z_{N-1}$ and the counterclockwise part of $\partial U^w$ from $x$ to $y$ is part of $L^z_N$.}
\end{subfigure}
$\quad$
\begin{subfigure}[b]{0.48\textwidth}
\begin{center}
\includegraphics[width=0.8\textwidth]{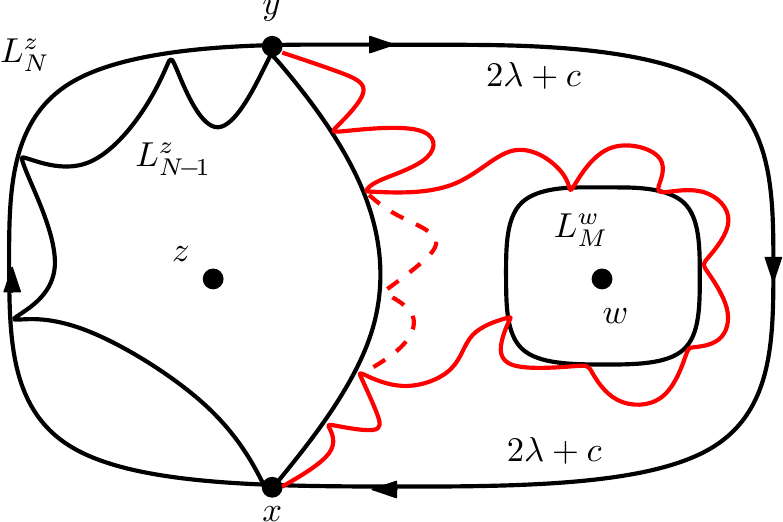}
\end{center}
\caption{Consider the field of $h$ restricted to $U^w\cap\ext(L^w_M)$. Since its boundary value along $\partial U^w\cap L^z_{N-1}$ is $c+r\lambda$ and $c+r\lambda+u\in (-\lambda,\lambda)$. 
The level lines with height $u$ can be continued towards both $x$ and $y$.}
\end{subfigure}
\caption{\label{fig::interaction_distinctstart_auxiliary} Explanation of the behavior of the level lines in Lemma \ref{lem::interaction_distinctstart_auxiliary}.}
\end{figure}

\begin{proof}
We may assume that $L^z_N$ is clockwise, and the case that $L^z_N$ is counterclockwise can be proved similarly. We have the following observations.
\begin{enumerate}
\item [(a)] Given $(L^z_n, n\le N)$, the conditional law of $h$ restricted to $\ext(L^z_N)$ is the same as $\GFF$ with boundary value $c$ which is some multiple of $r\lambda$.
\item [(b)] Note that $\partial U^w$ has two special points $x$ and $y$ so that the clockwise part of $\partial U^w$ from $x$ to $y$ is part of $L^z_{N-1}$ and the counterclockwise part of $\partial U^w$ from $x$ to $y$ is part of $L^z_N$. See Figure \ref{fig::interaction_distinctstart_auxiliary}(a).
Given $(L^z_n, n\le N)$, the conditional law of $h$ restricted to $U^w$ is the same as $\GFF$ with boundary value $2\lambda+c$ on $\partial U^w\cap L^z_N$ and 
\[\begin{cases}
c+r\lambda \text{ on }\partial U^w\cap L^z_{N-1}, &\text{if }L^z_{N-1}\text{ is clockwise},\\
2\lambda+c-r\lambda \text{ on }\partial U^w\cap L^z_{N-1}, &\text{if }L^z_{N-1}\text{ is counterclockwise}.
\end{cases}\]
\item [(c)] Given $(L^z_n, n\le N)$, the sequence $(L^w_n, n\in\Z)$ is coupled with $h$ restricted to $U^w$ in the same way as in Theorem \ref{thm::interior_levelloops_coupling} for $D=U^w$, up to the first time that the loop exits $U^w$.
\item [(d)] Given $(L^z_n, n\le N; L^w_k, k\le M)$ and on the event $\{L^w_M\subset\overline{U^w}, L^w_M\cap L^z_N=\emptyset\}$, the conditional law of $h$ restricted to $U^w\cap \ext(L^w_M)$ is the same as $\GFF$ with certain boundary data which is $2\lambda+c$ on $\partial U^w\cap L^z_N$.
\end{enumerate}
Assume that the loop $L^w_{M+1}$ has height $u$ and we will prove the conclusion by contradiction. Assume that $L^w_{M+1}$ enters $\ext(L^z_N)$. Then the pieces of $L^w_{M+1}$ that are contained in $\ext(L^z_N)$ are level lines of $h$ restricted to $\ext(L^z_N)$ with height $u$, thus $c+u\in (-\lambda, \lambda)$ by Lemma \ref{lem::levellinesI_remark2515} Item (1). Since $c$ is a multiple of $r\lambda$ and $\lambda+u$ is also a multiple of $r\lambda$, we have that $c+u\in [-\lambda+r\lambda, \lambda)$.
\medbreak
\textit{First}, we show that the loop $L^w_{M+1}$ can only exit $U^w$ through $x$ or $y$. Combining $2\lambda+c+u>\lambda$ with Lemma \ref{lem::levellinesI_remark2515} Item (3), we have that the loop $L^w_{M+1}$ can not hit the interior of $\partial U^w\cap L^z_N$ from inside. We only need to explain that $L^w_{M+1}$ can not exit $U^w$ through the interior points of $\partial U^w\cap L^z_{N-1}$. 

Assume that the loop $L^w_{M+1}$ does exit $U^w$ through an interior point of $\partial U^w\cap L^z_{N-1}$ and assume that the boundary value of the field along $\partial U^w\cap L^z_{N-1}$ is $v_1$ to the inside of $U^w$ and is $v_2$ to the outside of $U^w$. 
\begin{enumerate}
\item [(a)] Since the loop $L^w_{M+1}$ hits interior points of $\partial U^w\cap L^z_{N-1}$, we have $v_1+u\in (-\lambda, \lambda)$ by Lemma \ref{lem::levellinesI_remark2515} Item (3).
\item [(b)] Since the pieces of $L^w_{M+1}$ outside of $\partial U^w\cap L^z_{N-1}$ are level lines with height $u$, we have $v_2+u\in (-\lambda, \lambda)$ by Lemma \ref{lem::levellinesI_remark2515} Item (1).
\end{enumerate}
Combining these two facts with $|v_1-v_2|=2\lambda$, we get a contradiction.

Since $L^w_{M+1}$ is a loop, if it exits $U^w$, it has to come back to $U^w$. With a similar analysis as above, we know that $L^w_{M+1}$ can only come back to $U^w$ through $x$ or $y$. Therefore, if $L^w_{M+1}$ exits $U^w$, it has to exit $U^w$ through one of $\{x,y\}$ and come back to $U^w$ through the other one. 
\medbreak
\textit{Second}, we show that $L^z_{N-1}$ has to be clockwise. If the loop $L^z_{N-1}$ is counterclockwise, then the boundary value of $h$ restricted to $U^w$ is $2\lambda+c-r\lambda$ on $\partial U^w\cap L^z_{N-1}$. Since $2\lambda+c+u>2\lambda+c+u-r\lambda\ge \lambda$, the level lines of $h$, restricted to $U^w$, with height $u$ can not hit $y$ by Lemma \ref{lem::levellinesI_remark2515} Item (2), and the field of $h$, restricted to $U^w$, can not have non-trivial level line with height $u$ starting from $y$. This contracts that $L^w_{M+1}$ exits $U^w$ or comes back to $U^w$ through $y$.
\medbreak
\textit{Finally}, we show that there exists a level loop of $h$ restricted to $U^w\cap\ext(L^w_M)$ with height $u$ that is contained in $\overline{U^w}$. Denote by $\tilde{h}$ the field of $h$ restricted to $U^w\cap \ext(L^w_M)$. 
Since $L^z_{N-1}$ is clockwise, the field $\tilde{h}$ has boundary value $c+r\lambda$ along $\partial U^w\cap L^z_{N-1}$. 
Since the loop $L^w_{M+1}$ has height $u$ and exits $U^w$ at $x$ or $y$, we have $c+r\lambda+u<\lambda$ by Lemma \ref{lem::levellinesI_remark2515} Item (2). Therefore, we have $c+r\lambda+u\in (-\lambda, \lambda)$. 

Consider the level lines of $\tilde{h}$ with height $u$ starting from some points on $L^w_M$. Since $c+r\lambda+u\in (-\lambda,\lambda)$, these level lines can be continued towards both $x$ and $y$ after they hit $\partial U^w\cap L^z_{N-1}$ by \cite[Lemma 2.3.1]{WangWuLevellinesGFFI}. This implies that there exists a complete level loop of $\tilde{h}$ with height $u$ contained in the closure of $U^w\cap\ext(L^w_M)$ and touching $L^w_M$. We denote it by $\tilde{L}_u$. See Figure \ref{fig::interaction_distinctstart_auxiliary}(b).

Now, we have the following observations.
\begin{enumerate}
\item [(a)] The loop $\tilde{L}_u$ is a level loop of $h$ restricted to $\ext(L^w_M)$ with height $u$ starting from $L^w_M$ targeted at $\infty$.
\item [(b)] The loop $L^w_{M+1}$ is a level loop of $h$ restricted to $\ext(L^w_M)$ with height $u$ starting from $L^w_M$ targeted at $\infty$.
\end{enumerate} 
Combining these two facts with Lemma \ref{lem::levellinesI_levelloop}, we get a contradiction.
\end{proof}

\begin{proof}
[Proof of Proposition \ref{prop::interaction_distinctstart}]
We may assume that $L^z_N$ is clockwise and it has the boundary value $c$ to the left-side where $c$ is a multiple of $r\lambda$.
\medbreak
\textit{First}, we show that $L^z_N$ coincides with $L_n^w$ for some $n\in\Z$. Let $L^w_k$ be the first loop in $(L^w_n, n\in\Z)$ that hits the inside of $L^z_N$. By Lemma \ref{lem::interaction_distinctstart_auxiliary}, we know that $L^w_k$ is contained in the closure of $\inte(L^z_N)$. Suppose that the loop $L^w_k$ has height $u$. Since $L^w_k$ hits the inside of $L^z_N$, we have $2\lambda+c+u\in (-\lambda, \lambda)$ by Lemma \ref{lem::levellinesI_remark2515} Item (3). Note that $c$ is a multiple of $r\lambda$ and $\lambda+u$ is also a multiple of $r\lambda$, thus $\lambda+c+u=-mr\lambda$ for some integer $m\ge 1$. In particular, the loop $L^z_N$ has height $-\lambda-c=u+mr\lambda$. 

If $L^w_k$ is clockwise, then $L^z_N$ has to coincides with $L^w_{k+m}$, since both of $L^z_N$ and $L^w_{k+m}$ are level loops of $h$, restricted to $\ext(L^w_k)$, with height $u+mr\lambda$ starting from $L^w_k$ targeted at $\infty$ (Lemma \ref{lem::levellinesI_levelloop}).

If $L^w_k$ is counterclockwise, we have the following observations.
\begin{enumerate}
\item [(a)] The loop $L^z_N$ is a level loop of $h$ restricted to $\ext(L^w_k)$ with height $u+mr\lambda$ starting from $L^w_k$ targeted at $\infty$.
\item [(b)] The sequence $(L^w_{k+n}, n\ge 0)$ is the alternating height-varying sequence of $h$ restricted to $\ext(L^w_k)$ with height difference $r\lambda$ starting from $L^w_k$ targeted at $\infty$. Thus the sequence starts by downward height-varying sequence, changes the orientation after some finite step, and continues with upward height-varying sequence etc. 
\end{enumerate}
Combining these two facts, we know that $L^z_N$ coincides with $L^w_{k+2n+m}$ for some integer $n\ge 0$.
\medbreak
\textit{Next}, we show that $L^z_N$ coincides with $L^w_M$. By the first step, we know that $L^z_N$ coincides with $L^w_m$ for some $m\in\Z$. Since $L^w_M$ disconnects $z$ from $\infty$, we have $M\le m$. Symmetrically, the loop $L^w_M$ coincides with $L^z_n$ for some $n\ge N$. Combining these two facts, we have that $L^z_N$ coincides with $L^w_M$.
\end{proof}

\begin{proof}
[Proof of Theorem \ref{thm::interior_levelloops_interacting_commontarget}]
Suppose that, for $i=1,2$, the sequence $(L^{z_i}_t, t\in \R)$ is the alternating height-varying sequence of level loops of $h$ with height difference $r\lambda$ starting from $z_i$ targeted at $\infty$ indexed by minus the log of the conformal radius seen from $\infty$, and define
\[T_1=\inf\{t: L_t^{z_1}\text{ disconnects }z_2\text{ from }\infty\};\quad T_2=\inf\{t: L_t^{z_2}\text{ disconnects }z_1\text{ from }\infty\}.\]
By Proposition \ref{prop::interaction_distinctstart}, there exists some finite $t\in\R$ such that $L_t^{z_1}=L_t^{z_2}$; we define $T$ to be the inf of such times. We have the following observations.
\begin{enumerate}
\item [(a)] For any $t<T_1$, we have $z_2\not\in \inte(L^{z_1}_t)$; whereas, $z_2\in\inte(L^{z_1}_{T})$, thus $t<T$.Therefore $T_1\le T$.
\item [(b)] By Proposition \ref{prop::interaction_distinctstart}, we have $L^{z_1}_{T_1}=L^{z_2}_{T_2}$, thus $T\le T_1$.
\end{enumerate}
Combining these two facts, we have that $T=T_1$ almost surely. By symmetry, we have $T=T_1=T_2$ almost surely.

We have the following observations.
\begin{enumerate}
\item [(a)] The sequence $(L^{z_1}_t, t\ge T_1)$ is the alternating height-varying sequence of level loops of $h$ restricted to $\ext(L^{z_1}_{T_1})$.
\item [(b)] The sequence $(L^{z_2}_t, t\ge T_2)$ is the alternating height-varying sequence of level loops of $h$ restricted to $\ext(L^{z_2}_{T_2})$.
\item [(c)] The loops $L^{z_1}_{T_1}$ and  $L^{z_2}_{T_2}$ coincide.
\end{enumerate}
Combining these three facts and Proposition \ref{prop::levelloops_inside_deterministic}, we have that 
\[L^{z_1}_t=L^{z_2}_t,\quad \forall t\ge T_1=T_2.\]
This completes the proof.
\end{proof}

\begin{proof}
[Proof of Theorem \ref{thm::interior_levelloops_deterministic}]
We only need to show the conclusion for $D=\C$, and the proof for  $D\subsetneq\C$ follows from absolute continuity. 

Suppose that $h$ is a whole-plane $\GFF$ modulo a global additive constant in $r\lambda\Z$. Suppose that $(L^1_t, t\in\R)$ and $(L^2_t, t\in\R)$ are two alternating height-varying sequences of level loops of $h$ starting from the origin targeted at $\infty$ indexed by minus the log of the conformal radius seen from $\infty$, and they are coupled with $h$ so that, given $h$, they are conditionally independent. 

Pick any $z\in\C$ and let $(L^z_t, t\in\R)$ be the alternating height-varying sequence of level loops of $h$ starting from $z$ targeted at $\infty$. By Theorem \ref{thm::interior_levelloops_interacting_commontarget}, for $i=1,2$, there exists a finite number $T_i$ such that
\[L^i_t=L^z_t, \quad \forall t\ge T_i.\]
Therefore, there exists a finite number $T$ such that
\[L^1_t=L^2_t, \quad \forall t\ge T.\]
Let $R$ be the inf of such times. The scaling invariance of the coupling $(h, (L^1_t, t\in\R), (L^2_t, t\in\R))$ implies that the law of $e^{-R}$ is scaling invariant. Thus $e^{-R}=0$ almost surely. Therefore, almost surely,
\[L^1_t=L^2_t, \quad \forall t\in\R.\]
\end{proof}

\begin{proof}
[Proof of Theorem \ref{thm::interior_levelloops_interacting_commonstart}]
Consider $(L_n^{z\to w_1},  n\in\Z)$,  let $N_0$ be any stopping time such that both $w_1$ and $w_2$ are contained in $\ext(L_{N_0}^{z\to w_1})$. Given $(L^{z\to w_1}_n, n\le N_0)$, denote by $\tilde{h}$ the field of $h$ restricted to $\ext(L^{z\to w_1}_{N_0})$. For $i=1,2$, let $(\tilde{L}_{n}^{z\to w_i}, n\ge 0)$ be the alternating height-varying sequence of level loops of $\tilde{h}$ with height difference $r\lambda$ starting from $\tilde{L}^{z\to w_i}_0=L^{z\to w_1}_{N_0}$ targeted at $w_i$; and denote by $\tilde{U}_n^{w_i}$ the connected component of $\C\setminus \tilde{L}_n^{z\to w_i}$ that contains $w_i$. We have the following observations.
\begin{enumerate}
\item [(a)] From the target-independence in Proposition \ref{prop::levellinesI_targetindependence}, we know that there exists a number $\tilde{M}$ such that
\[\tilde{L}^{z\to w_1}_n=\tilde{L}^{z\to w_2}_n,\quad \text{for }n\le \tilde{M}-1; \quad \tilde{U}^{w_1}_n\cap \tilde{U}^{w_2}_n=\emptyset, \quad \text{for }n=\tilde{M}.\]
Given $(\tilde{L}_n^{z\to w_1}, \tilde{L}_n^{z\to w_2}, n\le\tilde{M})$, the two sequences continue towards their target points respectively in a conditionally independent way. 
\item [(b)] By Theorem \ref{thm::interior_levelloops_deterministic}, for $i=1,2$, we know that the collection of the loops in the union of $(L_n^{z\to w_1}, n\le N_0)$ and $(\tilde{L}_n^{z\to w_i}, n\ge 0)$ is the same as the collection of the loops in the sequence $(L^{z\to w_i}_n, n\in\Z)$. 
\end{enumerate}
Combining these two facts, we obtain the conclusion.
\end{proof}

\begin{proof}
[Proof of Theorem \ref{thm::interior_levelloops_determins_field}]
For $p\ge 1$, let $\LF_p$ be the $\sigma$-algebra generated by $((L^m_n, n\in\Z), m\le p)$. Fix $f\in H_{s,0}$ and define $X_p=\E[(h,f)\cond \LF_p]$ for $p\ge 1$.
\medbreak
\textit{First}, we argue that, to obtain the conclusion, it is sufficient to show that
\begin{equation}\label{eqn::mart_l2_convergence}
\E\left[\left((h,f)-X_p\right)^2\right]\to 0,\quad \text{as }p\to\infty.
\end{equation}
We have the following observations.
\begin{enumerate}
\item [(a)] Since $(h,f)$ is $L^2$ integrable, the martingale $(X_p, p\ge 1)$ converges to $\E[(h,f)\cond \LF_{\infty}]$ almost surely and in $L^2$.
\item [(b)] Assuming Equation (\ref{eqn::mart_l2_convergence}) holds, we have that $(X_p, p\ge 1)$ converges to $(h,f)$ in $L^2$.
\end{enumerate}
Combining these two facts, we have that $\E[(h,f)\cond \LF_{\infty}]=(h,f)$ almost surely, which implies the conclusion.
\medbreak
\textit{Next}, we show Equation (\ref{eqn::mart_l2_convergence}). For each $p\ge 1$, let $G_p$ be the Green's function of $D_p:=\C\setminus \cup_{m\le p}\cup_{n\in\Z}L^m_n$.
Since $\C\setminus D_p$ has Lebesgue measure zero, we can view $G_p$ as a function defined on $\C\times\C$ and the value of $G_p$ on $\C\setminus D_p$ will not affect later integrations.
Note that $D_p$ has countably many connected components. If $x$ and $y$ are contained in the same connected component of $D_p$, then $G_p(x,y)$ is exactly the Green's function of that component taking value at $(x,y)$; if $x$ and $y$ are contained in distinct components of $D_p$, then $G_p(x,y)=0$. Since the sequence $(D_p, p\ge 1)$ is decreasing, the sequence $(G_p(x,y), p\ge 1)$ is non-increasing for fixed $x$ and $y$. 

Let $K$ be a compact set which contains the support of $f$. Then
\[\E\left[\left((h,f)-X_p\right)^2\right]
=\int_K\int_K f(x)G_p(x,y)f(y)dxdy
\le ||f||_{\infty}^2\int_K\int_K G_p(x,y)dxdy.\]
Note that $G_p$ is bounded by $G_1$ which is integrable; and that $\lim_pG_p(x,y)=0$, for fixed $x$ and $y$, since $x$ and $y$ will be disconnected by $((L^m_n, n\in\Z), m\le p)$ for $p$ large enough.
Thus $\int_K\int_K G_p(x,y)dxdy\to 0$, which implies Equation (\ref{eqn::mart_l2_convergence}).
\end{proof}

\section{Continuum exploration process starting from interior}
\label{sec::continuum_exploration}

\subsection{Growing process of $\CLE_4$}
\label{subsec::growing_cle4}
In this section, we will introduce a continuum growing process of $\CLE_4$. We first recall some features of $\CLE$ and refer to \cite{SheffieldWernerCLE} for details and proofs; and then recall the construction and properties of the continuum growing process of $\CLE_4$ and refer to \cite{WernerWuCLEExploration} for details and proofs. 

$\CLE$ is a random collection of non-nested disjoint simple loops that possesses Conformal Invariance and Domain Markov Property. The loops in $\CLE$ are $\SLE_{\kappa}$-type loops for some $\kappa\in (8/3,4]$; moreover, for each such value of $\kappa$, there exists exactly one $\CLE$ distribution that has $\SLE_{\kappa}$-type loops. We denote the corresponding $\CLE$ by $\CLE_{\kappa}$ for $\kappa\in (8/3,4]$. Throughout the current paper, we fix $\kappa=4$. 

We call a bubble $l$ in $\U$ if $l\subset\overline{\U}$ is a simple loop and $l\cap\partial\U$ contains exactly one point; we call the point in $l\cap\partial\U$ the root of $l$, denoted by $R(l)$. 
Consider a $\CLE$ in $\U$, draw a small disc $B(x,\eps)$ with center $x\in\partial\U$ and radius $\eps>0$, and let $l^{\eps}$ be the loop in $\CLE$ that intersects $B(x,\eps)$ with largest radius. Define the quantity
$u(\eps)=\PP[l^{\eps}\text{ surrounds the origin}]$.
The law of $l^{\eps}$ normalized by $1/u(\eps)$ converges to a limit measure (see \cite[Section 4]{SheffieldWernerCLE}), denoted by $M(x)$. Define \textbf{$\SLE_4$-bubble measure} $M$ by 
\[M=\int_{x\in\partial\U}dx M(x),\]
where $dx$ is Lebesgue measure on $\partial\U$. When $\kappa=4$, the bubble measure $M$ is conformal invariant. Namely, for any M\"{o}bius transformation $\phi$ of $\U$, we have $\phi\circ M=M$ (see \cite[Lemma 6]{WernerWuCLEExploration}). 

Let $(l_u, u\ge 0)$ be a Poisson point process with intensity $M$. (In other words, let $((l_j, t_j), j\in J)$ be a Poisson point process with intensity $M\times [0,\infty)$, and then arrange the bubbles according to time $t_j$, i.e. denote the bubble $l_j$ by $l_u$ if $u=t_j$, and $l_u$ is empty set if there is no $t_j$ that equals $u$.) Clearly, there are only countably many bubbles in $(l_u, u\ge 0)$ that are not empty set.

We will construct a growing process from the sequence of bubbles $(l_u, u\ge 0)$. Define $\tau_1$ to be the first time $u$ that $l_u$ surrounds the origin. Consider the sequence $(l_u, 0\le u\le \tau_1)$. For each $u<\tau_1$, the bubble $l_u$ does  not surround the origin. Define $f_u$ to be the conformal map from the connected component of $\U\setminus l_u$ that contains the origin onto the unit disc such that $f_u(0)=0$ and $f_u'(0)>0$. We have the following properties. See \cite[Section 7]{SheffieldWernerCLE}
\begin{enumerate}
\item [(a)] The stopping time $\tau_1$ has exponential law: $\PP[\tau_1>u]=e^{-u}$.
\item [(b)] For $\delta>0$ small, let $u_1, u_2, ..., u_j$ be the times $u$ before $\tau_1$ at which the bubble $l_u$ has radius greater than $\delta$. Define $\Psi^{\delta}=f_{u_j}\circ\cdots\circ f_{u_1}$. Then $\Psi^{\delta}$ almost surely converges to some conformal map $\Psi$ in Carath\'eodory topology seen from the origin as $\delta$ goes to zero. Therefore, $\Psi$ can be interpreted as $\Psi=\circ_{u<\tau_1}f_u$, and $\Psi$ maps some simply connected open subset of $\U$ onto $\U$. 
\item [(c)] Define $L_{\tau_1}=\Psi^{-1}(l_{\tau_1})$. Then $L_{\tau_1}$ has the same law as the loop in $\CLE_4$ that surrounds the origin. In particular, $L_{\tau_1}$ is a simple loop. We give the clockwise orientation to the loop $L_{\tau_1}$.
\item [(d)] Generally, for each $u\le \tau_1$, we can define $\Psi_u=\circ_{s<u}f_s$, and $\Psi_u$ maps some simply connected open subset of $\U$ onto $\U$. We denote by $L_{u-}$ the boundary of $\Psi_u^{-1}(\U)$ which is a quasisimple loop. We give the counterclockwise orientation to $L_{u-}$. Note that $\Psi_u$ is the conformal map from the connected component of $\U\setminus L_{u-}$ that contains the origin onto $\U$ with $\Psi_u(0)=0$ and $\Psi'_u(0)>0$. 
\item [(e)] For each $u<\tau_1$, let $L_u$ be the boundary of $(f_u\circ\Psi_u)^{-1}(\U)$. See Figure \ref{fig::growing_cle4_quasisimpleloop}. Then $L_u$ is a quasisimple loop and we give counterclockwise orientation to it. 
\end{enumerate}

\begin{figure}[ht!]
\begin{subfigure}[b]{\textwidth}
\begin{center}
\includegraphics[width=0.6\textwidth]{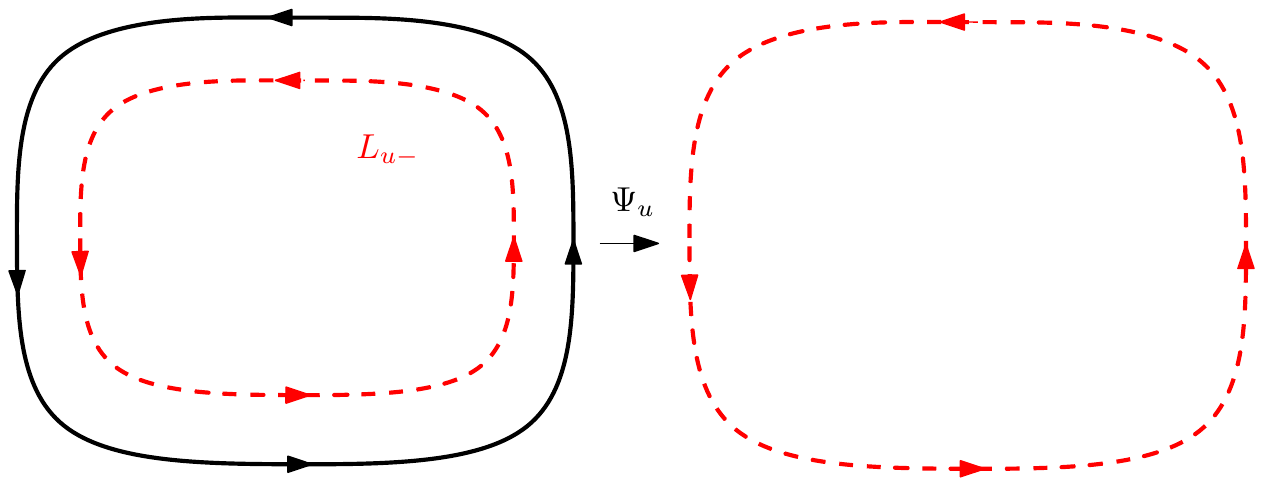}
\end{center}
\caption{The boundary of $\Psi_u^{-1}(\U)$ is a quasisimple loop, denoted by $L_{u-}$.}
\end{subfigure}
$\quad$
\begin{subfigure}[b]{\textwidth}
\begin{center}\includegraphics[width=0.6\textwidth]{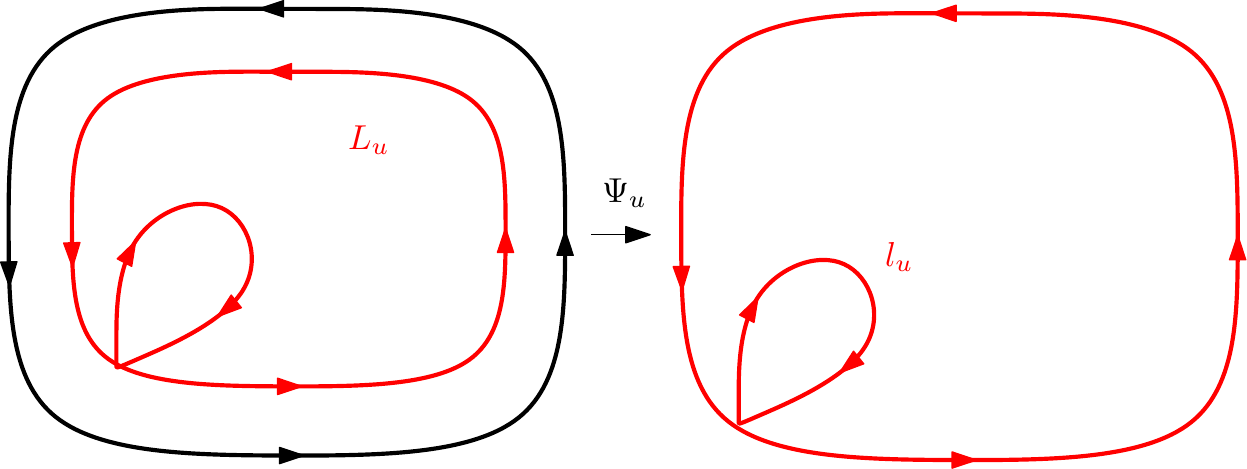}
\end{center}
\caption{The boundary of $(f_u\circ\Psi_u)^{-1}(\U)$ is a quasisimple loop, denoted by $L_{u}$.}
\end{subfigure}
\caption{\label{fig::growing_cle4_quasisimpleloop} The growing process of $\CLE_4$. For each $u<\tau_1$, the map $\Psi_u=\circ_{s<u}f_s$ is a conformal map from some simply connected subset of $\U$ onto $\U$.}
\end{figure}

From the above construction, we get a sequence of quasisimple loops $(L_u, 0\le u\le \tau_1)$ from the sequence of bubbles $(l_u, 0\le u\le \tau_1)$. We will show the sequence $(L_u, 0\le u\le \tau_1)$ is c\`adl\`ag. Before this, we need to be precise about the topology. Fix $z\in \C$ and consider a sequence of quasisimple loops $(L_n, n\ge 1)$. For each $n$, denote by $U_n^z$ the connected component of $\C\setminus L_n$ that contains $z$. We say that the sequence $(L_n, n\ge 1)$ converges to a quasisimple loop $L$ in Carath\'eodory topology seen from $z$ if the sequence $(U_n^z, n\ge 1)$ converges to $U^z$ in Carath\'eodory topology seen from $z$ where $U^z$ is the connected component of $\C\setminus L$ that contains $z$. 

\begin{lemma} \label{lem::growing_cle4_cadlag}
Consider the sequence of quasisimple loops $(L_u, 0\le u\le \tau_1)$. 
\begin{itemize}
\item For each $u\in [0,\tau_1)$, the right limit $\lim_{v\downarrow u}L_v$ exists and equals $L_u$ almost surely ;
\item For each $u\in (0,\tau_1]$, the left limit $\lim_{v\uparrow u}L_v$ exists and equals $L_{u-}$ almost surely.
\end{itemize}
\end{lemma} 
\begin{proof}
\textit{First}, we show the conclusion for the right limit. For $u\in [0,\tau_1)$ and $v\in (u,\tau_1)$, we have the following observations.
\begin{enumerate}
\item [(a)] The map $f_u\circ\Psi_u$ (resp. $f_v\circ\Psi_v$) is the conformal map from the connected component of $\U\setminus L_u$  (resp. $\U\setminus L_v$) that contains the origin onto $\U$ with $f_u\circ\Psi_u(0)=0$ and $(f_u\circ\Psi_u)'(0)>0$ (resp. $f_v\circ\Psi_v(0)=0$ and $(f_v\circ\Psi_v)'(0)>0$).
\item [(b)] As $v\downarrow u$, the map $f_v\circ\Psi_v$ converges almost surely to $f_u\circ\Psi_u$ in  Carath\'eodory topology seen from the origin. 
\end{enumerate}
Combining these two facts, we have that the right limit $\lim_{v\downarrow u}L_v$ exists and equals $L_u$.
\medbreak
\textit{Next}, we show the conclusion for the left limit. For $u\in (0,\tau_1]$ and $v\in [0,u)$, we have the following observations.
\begin{enumerate}
\item [(a)] The map $f_u\circ\Psi_u$ (resp. $f_v\circ\Psi_v$) is the conformal map from the connected component of $\U\setminus L_u$  (resp. $\U\setminus L_v$) that contains the origin onto $\U$ with $f_u\circ\Psi_u(0)=0$ and $(f_u\circ\Psi_u)'(0)>0$ (resp. $f_v\circ\Psi_v(0)=0$ and $(f_v\circ\Psi_v)'(0)>0$).
\item [(b)] As $v\uparrow u$, the map $f_v\circ\Psi_v$ converges almost surely to $\Psi_u$ in  Carath\'eodory topology seen from the origin. 
\end{enumerate}
Combining these two facts, we have that the left limit $\lim_{v\uparrow u}L_v$ exists and equals $L_{u-}$.
\end{proof}

\begin{figure}[ht!]
\begin{center}
\includegraphics[width=0.6\textwidth]{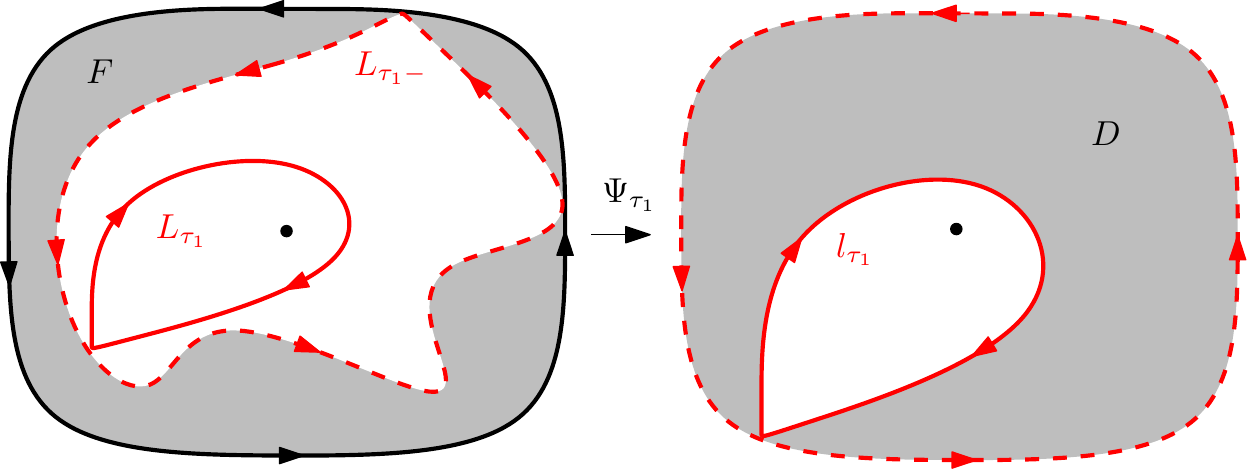}
\end{center}
\caption{\label{fig::cle4_D_F} Explanation of the definition of the compact sets $D$ and $F$: in the left figure, the gray region is $F$ which is the compact set lying between $\partial\U$ and $L_{\tau_1-}$; in the right figure, the gray region is $D$ which is the compact set lying between $\partial\U$ and the bubble $l_{\tau_1}$.}
\end{figure}

Consider the Poisson point process $(l_u, u\ge 0)$ and keep the same notations as above. Recall that $\Psi_{\tau_1}=\circ_{u<\tau_1}f_u$ is the composition of  maps $f_u$ which is the conformal map from the connected component of $\U\setminus l_u$ that contains the origin onto $\U$ with $f_u(0)=0$ and $f_u'(0)>0$. The loop $L_{\tau_1-}$ is the boundary of $\Psi_{\tau_1}^{-1}(\U)$. In other words, $\Psi_{\tau_1}$ is the conformal map from the connected component of $\U\setminus L_{\tau_1-}$ that contains the origin, denoted by $U_{\tau_1-}$, onto $\U$ with $\Psi_{\tau_1}(0)=0$ and $\Psi_{\tau_1}'(0)>0$. 

Define the compact set $F$ to be the domain lying between $\partial\U$ and $L_{\tau_1-}$ (see Figure \ref{fig::cle4_D_F}):
\[F=\overline{\U}\setminus U_{\tau_1-},\]
and denote its law by $\mu_F^{\sharp}$. Note that the complement of $F$ in the Riemann sphere has two connected components: $U_{\tau_1-}$ and $\C\cup\{\infty\}\setminus \overline{\U}$. 

The bubble $l_{\tau_1}$ is the first bubble in $(l_u, u\ge 0)$ that surrounds the origin. Define the compact set $D$ to be the domain lying between $\partial\U$ and $l_{\tau_1}$ (see Figure \ref{fig::cle4_D_F}):
\[D=\overline{\U}\setminus \inte(l_{\tau_1}),\]
and denote its law by $\mu^{\sharp}_D$. Note that the complement of $D$ in the Riemann sphere has two connected components: $\inte(l_{\tau_1})$ and $\C\cup\{\infty\}\setminus\overline{\U}$; and that the interior of $D$ is simply connected. 

The loop $L_{\tau_1}$ has the same law as the loop in $\CLE_4$ in the unit disc that surrounds the origin. Define $A$ to be the annulus lying between $\partial\U$ and $L_{\tau_1}$:
\[A=\U\setminus\overline{\inte(L_{\tau_1})},\]
and denote its law by $\mu_A^{\sharp}$. 

From the above analysis, we have the following relation between the pair $(F, D)$ and the loop in $\CLE_4$.

\begin{lemma}\label{lem::growing_cle4_operation1}
Suppose that $(F, D)$ is a pair of compact sets sampled according to the law $\mu_F^{\sharp}\otimes\mu_D^{\sharp}$. Let $\Psi_F$ be the conformal map from the connected component of $\C\setminus F$ that contains the origin onto $\U$ with $\Psi_F(0)=0$ and $\Psi_F'(0)>0$. Then the law of the compact set $F\cup\Psi_F^{-1}(D)$ is the same as the law of the closure of an annulus with law $\mu_A^{\sharp}$.
\end{lemma}

\subsection{From discrete to continuum exploration of $\GFF$}
\label{subsec::gff_discrete_continuum}
In this section, we will explain how to couple $\GFF$ with the sequence of quasisimple loops $(L_u, 0\le u\le \tau_1)$ introduced in Section \ref{subsec::growing_cle4}. Roughly speaking, the sequence $(L_u, 0\le u\le \tau_1)$ can be viewed as the limit of the upward height-varying sequence of level loops as the height difference goes to zero. 

Suppose that $h$ is a zero-boundary $\GFF$ on $\U$. Fix $r\in (0,1)$, let $(L_n, 0\le n\le N)$ be the upward height-varying sequence of level loops with height difference $r\lambda$ starting from $L_0=\partial\U$ (counterclockwise) targeted at the origin, where $N$ is the transition step. We have the following properties. 

\begin{enumerate}
\item [(a)] The transition step satisfies geometric distribution. The normalized quantity $Nr/2 $ converges in distribution to exponential law as $r$ goes to zero.
\item [(b)] The loops $L_n$ are counterclockwise for $0\le n\le N-1$; and the loop $L_N$ is clockwise.
\item [(c)] For any stopping time $M\le N$, the conditional law of $h$ given $(L_n, 0\le n\le M)$ is the same as $\GFF$ on $\U\setminus \cup_{n\le M}L_n$ whose boundary value is, for each $n$,
\[\begin{cases}
-nr\lambda,&\text{to the left-side of }L_n,\\
2\lambda-nr\lambda, &\text{to the right-side of }L_n.
\end{cases}\]
\end{enumerate}

Let $(\tilde{L}_n, 0\le n\le \tilde{N})$ be the upward height-varying sequence of level loops of $h$ with height difference $r\lambda/2$, where $\tilde{N}$ is the transition step. By \cite[Lemma 3.4.1]{WangWuLevellinesGFFI}, the relation between $(L_n, 0\le n\le N)$ and $(\tilde{L}_n, 0\le n\le \tilde{N})$ can be summarized as follows.
\begin{enumerate}
\item [(a)] Almost surely, we have $\tilde{L}_{2n}=L_n$ for $1\le n\le N-1$.
\item [(b)] There are two possibilities for $\tilde{L}_{\tilde{N}}$:
\[\begin{cases}
\tilde{N}=2N-1,\quad \tilde{L}_{\tilde{N}}\subset \overline{\inte(L_N)}, &\text{if }\tilde{L}_{2N-1} \text{ is clockwise};\\
\tilde{N}=2N,\quad \tilde{L}_{\tilde{N}}=L_N,&\text{if }\tilde{L}_{2N-1} \text{is counterclockwise}.
\end{cases}\]
\end{enumerate}

Suppose that $h$ is a zero-boundary $\GFF$ in $\U$. For each $k\ge 1$, let $(L_n^k, 0\le n\le N^k)$ be the upward height-varying sequence of level loops with height difference $2^{-k}\lambda$ starting from $L_0=\partial\U$ targeted at the origin, where $N^k$ is the transition step. From the above analysis, we have the following observations.
\begin{enumerate}
\item [(a)] For each $k\ge 1$, define $\tau^{(k)}=2^{-k-1}N^k$. Then 
\[0\le \tau^{(k)}-\tau^{(k+1)}\le 2^{-k-2},\quad\text{for all } k\ge 1.\]
Thus, the sequence $(\tau^{(k)}, k\ge 1)$ almost surely converges to some limit, denoted by $\tau^{(\infty)}$, which has exponential law.
\item [(b)] We have
\[\inte(L^{k+1}_{N^{k+1}})\subset \inte(L^k_{N^k}),\quad \text{for all }k\ge 1.\]
The sequence $(L^k_{N^k}, k\ge 1)$ almost surely converges in Carath\'eodory topology to some limit, denoted by $L^{(\infty)}$, which has the same law as the loop in $\CLE_4$ that surrounds the origin \cite[Proposition 3.4.2]{WangWuLevellinesGFFI}.
\end{enumerate}

We reindex each sequence $(L^k_n, 0\le n\le N^k)$ in the following way:
\[L^{(k)}_u=L^k_n, \quad \text{for}\quad 2^{-k-1}n<u\le 2^{-k-1}(n+1),\quad  0\le n\le N^k-1.\]
For each  $0\le u\le \tau^{(k)}$, let $\Psi^{(k)}_u$ be the conformal map from $\inte(L^{(k)}_u)$ onto $\U$ with $\Psi^{(k)}_u(0)=0$ and $(\Psi^{(k)}_u)'(0)>0$.

Suppose that $(l_u, u\ge 0)$ is a Poisson point process with intensity $\SLE_4$-bubble measure $M$ and let $\tau_1$ be the first time $u$ that $l_u$ surrounds the origin. For each $u\in [0,\tau_1)$, let $f_u$ be the conformal map from the connected component of $\U\setminus l_u$ that contains the origin onto $\U$ with $f_u(0)=0, f_u'(0)>0$. For $u\in (0,\tau_1]$, let $\Psi_u=\circ_{s<u}f_s$ which is well-defined, see Section \ref{subsec::growing_cle4}. Let $L_{\tau_1}$ be $\Psi_{\tau_1}^{-1}(l_{\tau_1})$, and for $u\in [0,\tau_1)$, let $L_u$ be the boundary of $(f_u\circ\Psi_u)^{-1}(\U)$. Then we have the following convergence.

\begin{lemma}\label{lem::discrete_continuum_cvg}
The family of conformal maps $(\Psi^{(k)}_u, 0\le u\le \tau^{(k)})$ converges almost surely to some limit $(\Psi^{(\infty)}_u, 0\le u\le \tau^{(\infty)})$ with respect to the topology of local uniform convergence. Namely, for any $\delta>0$ and every compact interval $[a,b]$, the map $\Psi^{(k)}_u$ converges almost surely to $\Psi^{(\infty)}_u$ uniformly on \[[a,b]\times \{w\in \U: \dist(w,L_{b-})\ge \delta \text{ and }w \in U^0_{b-}\},\] 
where $U^0_{b-}$ is the connected component of $\U\setminus L_{b-}$ that contains the origin. 

Moreover, the family of conformal maps $(\Psi^{(\infty)}_u, 0\le u\le \tau^{(\infty)})$ has the same law as the family $(\Psi_u, 0\le u\le \tau_1)$.
\end{lemma}
\begin{proof}
Note that, for each $u$, the sequence of simply connected domains $(\inte(L_u^{(k)}), k\ge 1)$ is decreasing. To show the convergence of the domains in Carath\'eodory topology (seen from the origin), it is sufficient to show the convergence in conformal radii (seen from the origin) $(\CR\left(\inte(L_u^{(k)})\right), k\ge 1)$ (see \cite[Lemma 3.2.1]{WangWuLevellinesGFFI}).

For $u\ge 0$ and $m>n$, define
\[M_u=\log\frac{\CR\left(\inte(L_u^{(n)})\right)}{\CR\left(\inte(L_u^{(m)})\right)}\ge 0.\]
By \cite[Lemma 3.3.8]{WangWuLevellinesGFFI}, we have the following observations.
\begin{enumerate}
\item [(a)] There is a universal constant $C$ such that $\E[M_u]\le C2^{-n}$.
\item [(b)] We can show that the process $(M_u, u\ge 0)$ is a non-negative submartingale. 
\end{enumerate}
Combining these two facts with Doob's maximal inequality, we have that, for any $\eps>0$,
\[\PP\left[\sup_{u\in [a,b]}\left|\log\frac{\CR\left(\inte(L_u^{(n)})\right)}{\CR\left(\inte(L_u^{(m)})\right)}\right|\ge \eps\right]\le\frac{C}{\eps}2^{-n}.\]
This implies the uniform convergence in conformal radii, which completes the proof.
\end{proof}

\begin{remark}\label{rem::discrete_continuum_cvg}
From the convergence in Lemma \ref{lem::discrete_continuum_cvg}, we have the following observations.
\begin{enumerate}
\item [(1)] Almost surely, the sequence $(L^k_{N^k}, k\ge 1)$ monotonically converges to some limit, denoted by $L^{(\infty)}$, in Carath\'eodory topology seen from the origin.  Moreover, the loop $L^{(\infty)}$ has the same law as the loop in $\CLE_4$ that surrounds the origin.
\item [(2)] Almost surely, for each $u\in (0,\tau^{(\infty)}]$, the sequence $(L^{(k)}_u, k\ge 1)$ monotonically converges to some limit, denoted by $L^{(\infty)}_{u-}$, in Carath\'eodory topology seen from the origin. Moreover, the loop $L^{(\infty)}_{u-}$ has the same law as $L_{u-}$ for the sequence $(L_u, 0\le u\le \tau_1)$ introduced in Section \ref{subsec::growing_cle4}.
\item [(3)] Define 
\[L^{(\infty)}_{\tau^{(\infty)}}=L^{(\infty)}; \quad L^{(\infty)}_u=\lim_{v\downarrow u}L^{(\infty)}_{v-}, \quad \text{for }u\in [0,\tau^{(\infty)}).\]
Then the sequence $(L^{(\infty)}_u, 0\le u\le \tau^{(\infty)})$ has the same law as the sequence $(L_u, 0\le u\le \tau_1)$ introduced in Section \ref{subsec::growing_cle4}.
\item [(4)] The sequence $(L^{(\infty)}_u, 0\le u\le \tau^{(\infty)})$ is almost surely determined by the field $h$.
\end{enumerate}
\end{remark}

From the convergence in Lemma \ref{lem::discrete_continuum_cvg}, we get the coupling between $\GFF$ and the sequence of quasisimple loops $(L_u, 0\le u\le \tau_1)$.

\begin{proposition}\label{prop::gff_upward_continuum}
There exists a coupling between zero-boundary $\GFF$ and c\`adl\`ag sequence of adjacent quasisimple loops $(L_u, 0\le u\le \tau_1)$ such that, for any stopping time $T\le \tau_1$, the conditional law of $h$ given $(L_u, u\le T)$ is the same as $\GFF$ whose boundary value is, for each $u\in [0,T]$,
\[\begin{cases}
-2\lambda u,&\text{to the left-side of } L_u,\\
2\lambda(1-u),&\text{ to the right-side of } L_u.
\end{cases}\]
In particular, for any stopping time $T<\tau_1$, given $(L_u, 0\le u\le T)$, let $U_T^0$ be the connected component of $\U\setminus L_T$ that contains the origin, then the conditional law of $h$ restricted to $U^0_T$ is the same as $\GFF$ with boundary value $-2\lambda T$. 
Given $(L_u, 0\le u\le \tau_1)$, the conditional law of $h$ restricted to $\inte(L_{\tau_1})$ is the same as $\GFF$ with boundary value $2\lambda(1-\tau_1)$. See Figure \ref{fig::gff_upward_continuum} (a).
\end{proposition}

\begin{figure}[ht!]
\begin{subfigure}[b]{0.48\textwidth}
\begin{center}
\includegraphics[width=0.75\textwidth]{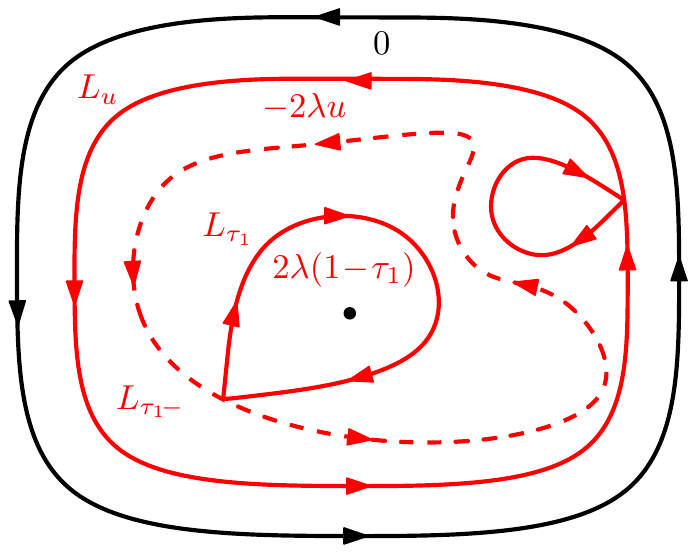}
\end{center}
\caption{Upward continuum exploration process. For $u<\tau_1$, the loop $L_u$ is counterclockwise and the boundary value is $-2\lambda u$ to the left-side of $L_u$. The loop $L_{\tau_1}$ is clockwise and the boundary value is $2\lambda (1-\tau_1)$ to the right-side of $L_{\tau_1}$.}
\end{subfigure}
$\quad$
\begin{subfigure}[b]{0.48\textwidth}
\begin{center}
\includegraphics[width=0.75\textwidth]{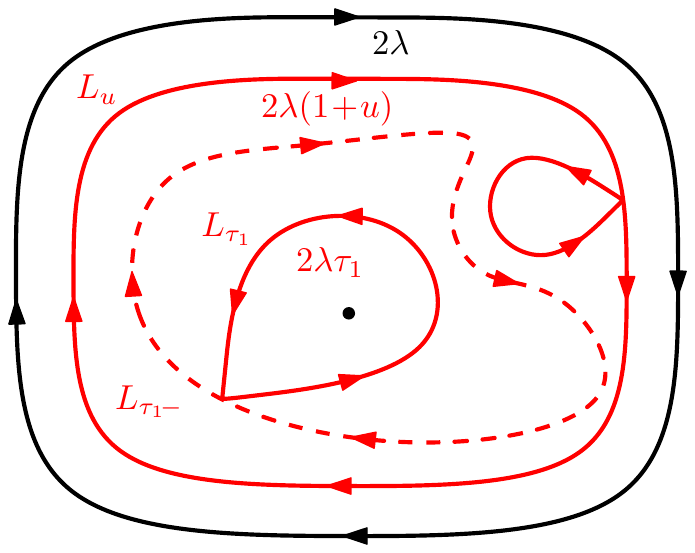}
\end{center}
\caption{Downward continuum exploration process. For $u<\tau_1$, the loop $L_u$ is clockwise and the boundary value is $2\lambda(1+ u)$ to the right-side of $L_u$. The loop $L_{\tau_1}$ is counterclockwise and the boundary value is $2\lambda\tau_1$ to the left-side of $L_{\tau_1}$.}
\end{subfigure}
\caption{\label{fig::gff_upward_continuum} Explanation of the upward and the downward continuum exploration processes of $\GFF$.}
\end{figure}

\begin{proposition}\label{prop::gff_upward_continuum_determinisitc}
In the coupling between $\GFF$ $h$ and c\`adl\`ag sequence of quasisimple loops $(L_u, 0\le u\le \tau_1)$ as in Proposition \ref{prop::gff_upward_continuum}, the sequence of loops $(L_u, 0\le u\le \tau_1)$ is almost surely determined by the field $h$. 
\end{proposition}

\begin{proof}
\textit{First}, we show that, for any $v\in (0,1]$ and $v<\tau_1$, let $\tilde{L}_a$ be the level loop of $h$ with height $-\lambda+2\lambda  a $ for some $a\in (0,v)$, then $\tilde{L}_a$ stays outside of $U_v^0$ which is the connected component of $\U\setminus L_v$ that contains the origin. By the coupling in Proposition \ref{prop::gff_upward_continuum}, we know that, given $(L_u, u\le v)$, the conditional law of $h$ restricted to $U_v^0$ is the same as $\GFF$ with boundary value $-2\lambda v$. If $\tilde{L}_a$ enters $U_v^0$, then the pieces of $\tilde{L}_a$ in $U_v^0$ are level lines of $h|_{U_v^0}$ with height $-\lambda +2\lambda a$. Whereas $-\lambda+2\lambda a -2\lambda v<-\lambda$, this contradicts with Lemma \ref{lem::levellinesI_remark2515} Item (1).
\medbreak
For each $k\ge 1$, let $(L^k_n, 0\le n\le N^k)$ be the upward height-varying sequence of level loops of $h$ with height difference $2^{-k}\lambda$, where $N^k$ is the transition step. We reindex the sequence in the following way:
\[L^{(k)}_u=L^k_n, \quad \text{for}\quad 2^{-k-1}n<u\le 2^{-k-1}(n+1),\quad  0\le n\le N^k-1.\]
We use the same notations as in Remark \ref{rem::discrete_continuum_cvg}. 
\medbreak
\textit{Second}, we show that, for any $v\in (0,1]$ and $v<\tau_1$, the loop $L_{v-}$ almost surely coincides with $L_{v-}^{(\infty)}$ and hence the loop $L_{v-}$ is almost surely determined by $h$. We have the following observations.
\begin{enumerate}
\item [(a)] By the first step, we know that, for each $k\ge 1$, the loop $L_v^{(k)}$ stays outside of $L_{v-}$. By Lemma \ref{lem::discrete_continuum_cvg} and Remark \ref{rem::discrete_continuum_cvg}, we know that, almost surely, the sequence $(L^{(k)}_v, k\ge 1)$ monotonically converges  to $L^{(\infty)}_{v-}$. Since each loop in the sequence stays outside of $L_{v-}$, the limit $L^{(\infty)}_{v-}$ also stays outside of $L_{v-}$.
\item [(b)] The loop $L^{(\infty)}_{v-}$ has the same law as $L_{v-}$.
\end{enumerate}
Combining these two facts, we have that $L^{(\infty)}_v$ and $L_{v-}$ coincide. 
\medbreak
\textit{Third}, we show that, for any $u\in (0,1)$ and $u<\tau_1$, the loop $L_u$ almost surely coincides with $L^{(\infty)}_u$ and hence the loop $L_u$ is almost surely determined by $h$. This is true because $L_u=\lim_{v\downarrow u}L_{v-}$.
\medbreak
\textit{Fourth}, we show that, for any $u< \tau_1$, the loop $L_{u-}$ coincides with $L^{(\infty)}_{u-}$ and is almost surely determined by $h$; and that, the loop $L_{u}$ coincides with $L^{(\infty)}_{u}$ and is almost surely determined by $h$. We only show the conclusion for $L_u$ and the conclusion for $L_{u-}$ can be proved similarly.

Assume that $u<n_0$ for some integer $n_0\ge 1$, define $u_j=ju/n_0$. We can prove by induction on $j$ that $L_{u_j}$ is almost surely determined by $h$. When $j=1$, since $u_1\in (0,1)$, the conclusion holds by the third step. Suppose that the conclusion holds for $j$, consider the loop $L_{u_{j+1}}$, we have the following observations.
\begin{enumerate}
\item [(a)] The loop $L_{u_j}$ is almost surely determined by $h$. Let $U_j^0$ be the connected component of $\U\setminus L_{u_j}$ that contains the origin. We also know that the conditional law of $h$ restricted to $U_j^0$ is the same as $\GFF$ with boundary value $-2\lambda u_j$.
\item [(b)] The sequence $(L_{u_j+u}, 0\le u\le \tau_1-u_j)$ is coupled with $h|_{U_j^0}$ in the same way as in Proposition \ref{prop::gff_upward_continuum}, thus $L_{u_{j+1}}$ is almost surely determined by $h|_{U_j^0}$ by the third step.  
\end{enumerate}
Combining these two facts, we have that $L_{u_{j+1}}$ is almost surely determined by $h$.
\medbreak
\textit{Finally}, we show that, the loop $L_{\tau_1}$ is almost surely determined by $h$. By the fourth step, we know that, for any $u<\tau_1$, the loop $L_{u-}$ coincides with $L^{(\infty)}_{u-}$. Therefore, $\tau^{(\infty)}\ge \tau_1$. Whereas, the quantity $\tau^{(\infty)}$ has the same law as $\tau_1$, thus $\tau^{(\infty)}=\tau_1$ almost surely. We will argue that $L^k_{N_k}$ stays outside of $L_{\tau_1}$. Assume this is true, then we know that the limit $L^{(\infty)}$ stays outside of $L_{\tau_1}$. Combining with the fact that $L^{(\infty)}$ has the same law as $L_{\tau_1}$, we know that $L_{\tau_1}$ coincides with $L^{(\infty)}$ and hence it is almost surely determined by $h$. Thus, we only need to show that $L^k_{N^k}$ stays outside of $L_{\tau_1}$. We have the following observations.
\begin{enumerate}
\item [(a)] By the coupling in Proposition \ref{prop::gff_upward_continuum}, we know that, the conditional law of $h$ restricted to $\inte(L_{\tau_1})$ is the same as $\GFF$ with boundary value $2\lambda (1-\tau_1)$.
\item [(b)] The loop $L^k_{N^k}$ is a level loop of $h$ with height $-\lambda+2\lambda \tau^{(k)}$.
\item [(c)] We know that $\tau^{(k)}\ge\tau^{(\infty)}=\tau_1$.
\end{enumerate}
Combining these three facts, if $L^k_{N^k}$ enters $\inte(L_{\tau_1})$, then the pieces of $L^k_{N^k}$ in $\inte(L_{\tau_1})$ are level lines of $h|_{\inte(L_{\tau_1})}$ with height $-\lambda+2\lambda\tau^{(k)}$. Whereas, $-\lambda+2\lambda\tau^{(k)}+2\lambda(1-\tau_1)\ge \lambda$, this contradicts with Lemma \ref{lem::levellinesI_remark2515} Item (1).
\end{proof}

In the coupling $(h, (L_u, 0\le u\le \tau_1))$ given by Proposition \ref{prop::gff_upward_continuum}, we call the sequence of quasisimple loops $(L_u, 0\le u\le \tau_1)$ the \textbf{upward continuum exploration process} of $h$ targeted at the origin. The following proposition describes the target-independence of the continuum exploration process.

\begin{figure}[ht!]
\begin{subfigure}[b]{\textwidth}
\begin{center}
\includegraphics[width=0.6\textwidth]{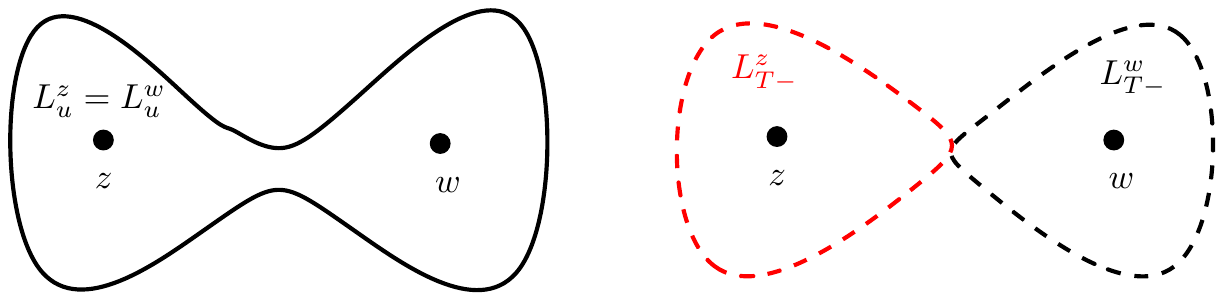}
\end{center}
\caption{We know that the loops $L^z_u$ and $L^w_u$ coincide for $u<T$ (the left panel). As $u\uparrow T$, the two left limits may become disjoint (the right panel). In this case, we have $S=T$.}
\end{subfigure}
\begin{subfigure}[b]{\textwidth}
\begin{center}\includegraphics[width=0.93\textwidth]{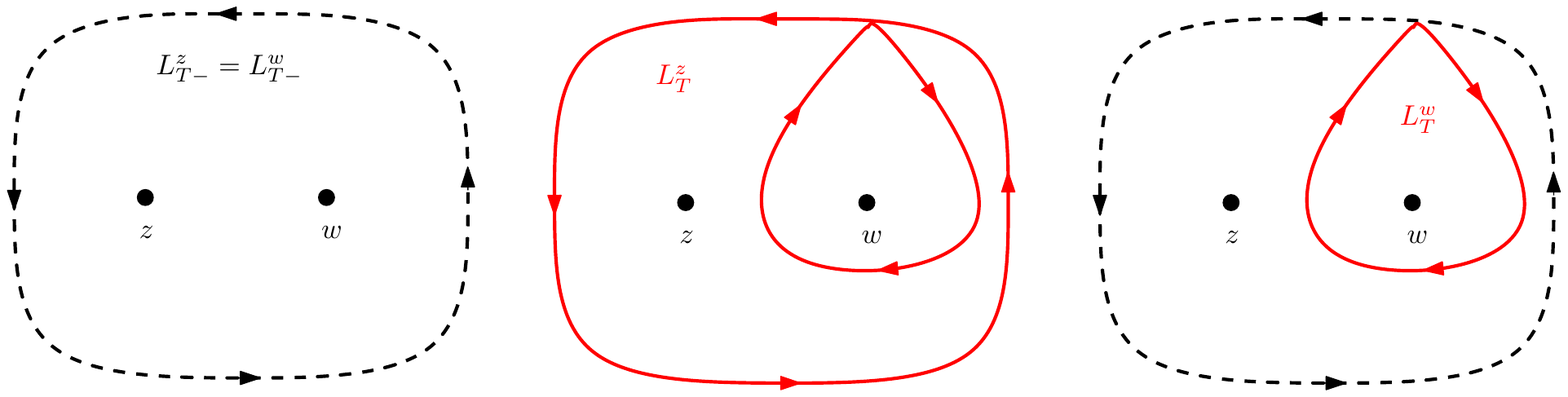}
\end{center}
\caption{As $u\uparrow T$, the two left limits may coincide. In this figure, at time $T$, the loop $L^z_T$ disconnects $w$ from $z$ by attaching a bubble that surrounds $w$ to $L^z_{T-}$. In this case, we have $S=T$.}
\end{subfigure}
\begin{subfigure}[b]{\textwidth}
\begin{center}\includegraphics[width=0.93\textwidth]{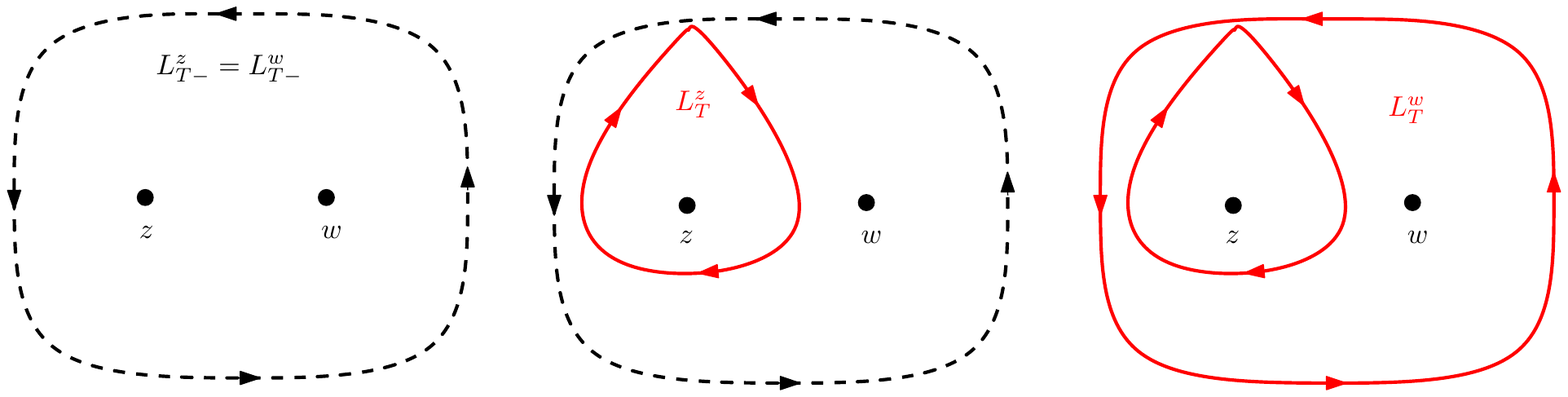}
\end{center}
\caption{As $u\uparrow T$, the two left limits may coincide. In this figure, at time $T$, the loop $L^z_T$ is a bubble that disconnects $w$ from $z$. In this case, we have $S=T$.}
\end{subfigure}
\caption{\label{fig::gff_upward_targetindependence} Consider two upward continuum exploration processes with distinct target points: $(L^z_u, 0\le u\le \tau_1^z)$ and $(L^w_u, 0\le u\le \tau_1^w)$. Suppose that $T$ is the first time $u$ that $L^z_u$ disconnects $w$ from $z$. We know that, for all $u<T$, the loops $L^z_u$ and $L^w_u$ coincide. As $u\uparrow T$, there are two possibilities for the loops $L^z_{T-}$ and $L^w_{T-}$: they are disjoint (a) or they coincide (b)(c). }
\end{figure}

\begin{proposition}\label{prop::continuum_upward_targetindependence}
Suppose that $h$ is a zero-boundary $\GFF$ on $\U$. Fix two target points $z,w\in\U$. Let $(L^z_u, 0\le u\le \tau_1^z)$ (resp. $(L_u^w, 0\le u\le \tau_1^w)$) be the upward continuum exploration process of $h$ targeted at $z$ (resp. targeted at $w$). Define stopping times: (with the convention that $\inf\emptyset=\infty$)
\[T=\inf\{u\in [0,\tau_1^z]: L_u^z\text{ disconnects }w\text{ from }z\};\quad S=\inf\{u\in [0,\tau_1^w]: L_u^w\text{ disconnects }z\text{ from }w\}.\]
Then, almost surely, we have $T=S$; moreover, the loops $L^z_u$ and $L_u^w$ coincide for all $u<T=S$.
\end{proposition}
\begin{proof}
\textit{First}, we show the conclusion for the case $T=\infty$. This is the case that $w$ is contained in $\inte(L^z_{\tau_1^z})$. In this case, we have that the sequence $(L^z_u, 0\le u\le \tau_1^z)$ is coupled with $h$ in the same way as the upward continuum exploration process targeted at $w$. By Proposition \ref{prop::gff_upward_continuum_determinisitc}, we have that the two sequences $(L^w_u, 0\le u\le \tau_1^w)$ and $(L^z_u, 0\le u\le \tau_1^z)$ coincide. In particular $S=\infty$.
\medbreak
\textit{Next}, we show the conclusion for the case $T\le \tau_1^z$. Let $U$ be any $(L^z_u, 0\le u\le \tau_1^z)$-stopping time such that $U<T$. Note that the sequence $(L^z_u, 0\le u\le U)$ is coupled with $h$ in the same way as the upward continuum exploration process targeted at $w$ stopped at $U$. Combining with Proposition \ref{prop::gff_upward_continuum_determinisitc}, we know that the two sequences $(L^z_u, 0\le u\le U)$ and $(L^w_u, 0\le u\le U)$ coincide. This holds for any $U<T$. In particular, for any $u<T$, we have $L_u^z=L_u^w$. We only need to show that $S=T$.

Consider the relation between $L^z_{T-}$ and $L^w_{T-}$, there are two possibilities: the two loops are disjoint (see Figure \ref{fig::gff_upward_targetindependence}(a))or the two loops coincide (see Figure \ref{fig::gff_upward_targetindependence}(b)(c)).
In the former case, clearly, we have $S=T$. In the latter case, let us consider the relation between $L^z_T$ and $L^z_{T-}$.
If the loop $L^z_T$ has the same orientation as $L^z_{T-}$ (Figure \ref{fig::gff_upward_targetindependence}(b)), then at time $T$, the loop $L^z_T$ disconnects $w$ from $z$ by attaching a bubble that surrounds $w$ to $L^z_{T-}$. This implies that $L^w_T$ is exactly this bubble, thus $S=T$. If the loop $L^z_T$ has the different orientation from $L^z_{T-}$ 
(Figure \ref{fig::gff_upward_targetindependence}(c)), in other words $T=\tau_1^z$, then at time $T$, the loop $L^z_T$ is a bubble attaching to $L^z_{T-}$ that disconnects $w$ from $z$. This implies that $L^w_T$ is the union of $L^z_T$ and $L^w_{T-}$. Thus $S=T$.
\end{proof}

From Proposition \ref{prop::continuum_upward_targetindependence}, we see that two upward continuum exploration processes with distinct target points coincide up to the first disconnecting time. In fact, in \cite[Lemma 8]{WernerWuCLEExploration}, the authors prove that the two processes (constructed from Poisson point process of $\SLE_4$-bubbles): $(L_u^z, 0\le u\le \tau_1^z)$ and $(L^w_u, 0\le u\le \tau_1^w)$ have the same law up to the first disconnecting time, and hence can be coupled so that they coincide up to the first disconnecting time. By Proposition \ref{prop::continuum_upward_targetindependence}, we see that $\GFF$ provides exactly this coupling. Namely, we couple the two processes with the same $\GFF$ $h$ so that $(L_u^z, 0\le u\le \tau_1^z)$ is the upward continuum exploration process of $h$ targeted at $z$ and $(L^w_u, 0\le u\le \tau_1^w)$  is the upward continuum exploration process of $h$ targeted at $w$. Then the two processes almost surely coincide up to the first disconnecting time. 

Now we can consider the upward continuum exploration processes of $\GFF$ $h$ targeted at all points in $\U$ simultaneously. Note that, for each pair of points $(z,w)$, the process targeted at $z$ and the process targeted at $w$ coincide up to the first disconnecting time. We summarize some properties of the process in the following.

\begin{proposition}\label{prop::upward_continuum_cle4}
Suppose that $h$ is a zero-boundary $\GFF$ on $\U$.
\begin{enumerate}
\item [(1)] The upward continuum exploration processes targeted at all points in $\U$ are almost surely determined by $h$.
\item [(2)] The law of the upward continuum exploration processes targeted at all points in $\U$ is invariant under any M\"obius transformation of $\U$. 
\item [(3)] The collection of all clockwise loops in the upward continuum exploration processes targeted at all points in $\U$ has the same law as $\CLE_4$ in $\U$.
\end{enumerate}
\end{proposition}  

\subsection{Alternating continuum exploration process of $\GFF$}
\label{subsec::gff_alternate_continuum}
In this section, we will describe an alternating continuum exploration process which can be viewed as the limit of the alternating height-varying sequence of level loops as the height difference goes to zero. Before this, let us summarize what we have obtained by now.

Suppose that $(l_u, u\ge 0)$ is a Poisson point process with $\SLE_4$-bubble measure $M$. Let $\tau_1$ be the first time $u$ that $l_u$ surrounds the origin. Define
\[\tau_{k+1}=\inf\{u>\tau_k: l_u\text{ surrounds the origin}\}, \quad k\ge 1.\]

\begin{enumerate}
\item [(a)] From Section \ref{subsec::growing_cle4}, we can map the bubbles $(l_u, 0\le u\le \tau_1)$ one-by-one into the unit disc $\U$ and obtain a c\`adl\`ag sequence of quasisimple loops $(L_u, 0\le u\le\tau_1)$. 
We give the counterclockwise orientation to the loops $L_u$ for $0\le u<\tau_1$; and the clockwise orientation to the loop $L_{\tau_1}$.
The loop $L_{\tau_1}$ has the same law as the loop in $\CLE_4$ that surrounds the origin. 
\item [(b)] From Section \ref{subsec::gff_discrete_continuum}, we know that there exists a coupling between zero-boundary $\GFF$ and the sequence $(L_u, 0\le u\le \tau_1)$ such that, for any stopping time $T\le \tau_1$, the conditional law of $h$ given $(L_u, 0\le u\le T)$ is the same as $\GFF$ with certain boundary value. Moreover, in this coupling, the sequence of loops is almost surely determined by the field. 
\end{enumerate}

In this construction, we call the sequence of loops $(L_u, 0\le u\le \tau_1)$ the upward continuum exploration process of $h$ starting from $L_0=\partial\U$ (counterclockwise) targeted at the origin. See Figure \ref{fig::gff_upward_continuum} (a).
Symmetrically, we can also construct the \textbf{downward continuum exploration process} of a $\GFF$ $h$ with $2\lambda$ boundary value in $\U$ starting from $L_0=\partial\U$ (clockwise) targeted at the origin. See Figure \ref{fig::gff_upward_continuum}(b). Note that, for the downward continuum exploration process $(L_u, 0\le u\le \tau_1)$, we have that, for any stopping time $T\le \tau_1$, the conditional law of $h$ given $(L_u, 0\le u\le T)$ is the same as $\GFF$ whose boundary value is 
\[\begin{cases}
2\lambda u , &\text{ to the left-side of }L_u,\\
2\lambda (1+u),&\text{ to the right-side of }L_u.
\end{cases}\]
In particular, for any stopping time $T<\tau_1$, 
the conditional law of $h$, given $(L_u, 0\le u\le T)$, restricted to the connected component of $\U\setminus L_T$ that contains the origin is the same as $\GFF$ with boundary value $2\lambda (1+T)$. The conditional law of $h$, given $(L_u, 0\le u\le \tau_1)$, restricted to $\inte(L_{\tau_1})$ is the same as $\GFF$ with boundary value $2\lambda\tau_1$.

\medbreak

Now we can construct alternating continuum exploration process going all the way to the origin. Suppose that $h$ is a zero-boundary $\GFF$ on $\U$ and $L_0=\partial\U$ is oriented counterclockwise. We start by the upward continuum exploration process $(L_u, 0\le u\le \tau_1)$ which can be viewed as being constructed from $(l_u, 0< u\le \tau_1)$. Given $(L_u, 0\le u\le \tau_1)$, the conditional law of $h$ restricted to $\inte(L_{\tau_1})$ is the same as $\GFF$ with boundary value $2\lambda(1-\tau_1)$. We continue the process by the downward continuum exploration process of $h$ restricted to $\inte(L_{\tau_1})$: $(L_u, \tau_1\le u\le \tau_2)$. This process can be viewed as being constructed from $(l_u, \tau_1<u\le \tau_2)$. Given $(L_u, 0\le u\le \tau_2)$, the conditional law of $h$ restricted to $\inte(L_{\tau_2})$ is the same as $\GFF$ with boundary value $2\lambda(\tau_2-2\tau_1)$. 

Generally, given $(L_u, 0\le u\le \tau_{2k})$ for some $k\ge 1$, we continue the sequence by the upward continuum exploration process $(L_u, \tau_{2k}\le u\le \tau_{2k+1})$ and then continue the sequence by the downward continuum exploration process $(L_u, \tau_{2k+1}\le u\le \tau_{2k+2})$. 

In this way, we obtain a sequence of quasisimple loops $(L_u, u\ge 0)$ which we call the \textbf{alternating continuum exploration process} of $h$ starting from $L_0=\partial\U$ targeted at the origin. We call $(\tau_k, k\ge 0)$ the \textbf{sequence of transition times}. For the sequence of transition times, we have the following.
\begin{enumerate}
\item [(a)] For $k\ge 0$, the loops $L_u$ have the same orientation for $u\in [\tau_k,\tau_{k+1})$, which is different from  $L_{\tau_{k+1}}$. 
\item [(b)] The random variables $(\tau_{k+1}-\tau_k, k\ge 0)$ are i.i.d. with exponential law.
\end{enumerate}

We summarize some basic properties of the alternating continuum exploration process of $\GFF$ in the following proposition.

\begin{proposition} Suppose that $(L_u, u\ge 0)$ is the alternating continuum exploration process of a zero-boundary $\GFF$ $h$. Then we have the following.
\begin{enumerate}
\item [(1)] The sequence $(L_u, u\ge 0)$ is almost surely c\`adl\`ag and transient.
\item [(2)] The sequence $(L_u, u\ge 0)$ is almost surely determined by $h$.
\item [(3)] The sequence $(L_u, u\ge 0)$ satisfies domain Markov property: for any stopping time $T$, let $U_T^0$ be the connected component of $\U\setminus L_T$ that contains the origin, and let $g_T$ be the conformal map from $U_T^0$ onto $\U$ with $g_T(0)=0$ and $g_T'(0)>0$. Then the sequence $(g_T(L_{T+u}), u\ge 0)$ has the same law as the alternating continuum exploration process of $\GFF$.
\end{enumerate}
\end{proposition}

\begin{proof}
Lemma \ref{lem::growing_cle4_cadlag} implies that the sequence is c\`adl\`ag; Proposition \ref{prop::gff_upward_continuum_determinisitc} implies that the sequence is almost surely determined by the field; and the domain Markov property is immediate from the construction. We only need to explain that the sequence is transient. 

Note that $L_{\tau_1}$ has the same law as the loop in $\CLE_4$ that surrounds the origin. In particular, $L_{\tau_1}\cap\partial\U=\emptyset$ and there exist $\eta\in (0,1)$ and $p\in (0,1)$ such that
\[\PP\left[L_{\tau_1}\subset \eta\U\right]\ge p.\]
Then the transience can be proved similarly as the proof of Proposition \ref{prop::levelloops_inside_transience}.
\end{proof}

By Proposition \ref{prop::continuum_upward_targetindependence}, we have the following target-independence property of the alternating continuum exploration process.
\begin{proposition}\label{prop::continuum_alternate_targetindependence}
Suppose that $h$ is a zero-boundary $\GFF$ in $\U$. Fix two target points $w_1,w_2\in \U$. For $i=1,2$, let $(L^{w_i}_u,u\ge 0)$ be the alternating continuum exploration process of $h$ starting from $L^{w_i}_0=\partial\U$ targeted at $w_i$; and denote by $U_u^{w_i}$ the connected component of $\U\setminus L_u^{w_i}$ that contains $w_i$. Then there exists a number $T$ such that 
\[L_u^{w_1}=L_u^{w_2},\quad \text{for }u<T; \quad U_T^{w_1}\cap U_T^{w_2}=\emptyset.\]
Given $(L_u^{w_1}, L_u^{w_2}, u\le T)$, the two processes continue towards their target points respectively in a conditionally independent way.
\end{proposition}

Suppose that $h$ is a zero-boundary $\GFF$ on $\U$ and $(L^z_u, u\ge 0)$ is the alternating continuum exploration process of $h$ targeted at $z$ where $(\tau^z_k, k\ge 0)$ is the sequence of transition times. We know that the loop $L^z_{\tau_1^z}$ has the same law as the loop in $\CLE_4$ in $\U$ that surrounds $z$. By domain Markov property, we know that the sequence $(L^z_{\tau^z_k}, k\ge 1)$ has the same law as the sequence of loops in nested $\CLE_4$ in $\U$ that surrounds $z$. 

From Proposition \ref{prop::continuum_alternate_targetindependence}, we know that the two alternating continuum exploration processes of $h$ with distinct target points coincide up to the first disconnecting time. Let us consider the alternating continuum exploration processes of $h$ targeted at all points in $\U$ simultaneously. We summarize some properties of this process in the following proposition.

\begin{proposition}\label{prop::alternate_continuum_cle4_nested}
Suppose that $h$ is a zero-boundary $\GFF$ in $\U$. For each $z\in\U$, let $(L^z_u, u\ge 0)$ be the alternating continuum exploration process targeted at $z$ and let $(\tau^z_k, k\ge 1)$ be the sequence of transition times. 
\begin{enumerate}
\item [(1)] The alternating continuum exploration processes targeted at all points in $\U$ are almost surely determined by $h$.
\item [(2)] The law of the alternating continuum exploration processes targeted at all points in $\U$ is invariant under any M\"obius transformation of $\U$. 
\item [(3)] The collection of loops $\{L^z_{\tau^z_k}: k\ge 1, z\in \U\}$ in the alternating continuum exploration processes targeted at all points in $\U$ has the same law as nested $\CLE_4$ in $\U$.
\end{enumerate}
\end{proposition}

\begin{proof}
[Proof of Theorem \ref{thm::interior_continuum_coupling}] This can be completed similarly as the proof of Theorem \ref{thm::interior_levelloops_coupling}.
\end{proof}

In the coupling $(h, (L_u, u\in\R))$ as in Theorem \ref{thm::interior_continuum_coupling}, we call the sequence $(L_u, u\in\R)$ the \textbf{alternating continuum exploration process} of $h$ starting from the origin targeted at infinity when $D=\C$ (resp. starting from $z$ targeted at $\partial D$ when $D\subsetneq\C$).

\subsection{Interaction}
\label{subsec::continuum_interaction}
The following proposition describes the interaction between the continuum exploration process and the level line of $\GFF$, it can be proved similarly as the proof of Proposition \ref{prop::interaction_interior_boundary}.

\begin{proposition}\label{prop::interaction_continuum_boundary}
Fix a domain $D\subset\C$ with harmonically non-trivial boundary. Fix three points $z\in D, x\in\partial D$ and $y\in\overline{D}$. Let $h$ be a $\GFF$ on $D$, let $(L_u^z, u\in\R)$ be the alternating continuum exploration process of $h$ starting from $z$ targeted at $\partial D$, and let $\gamma_u^{x\to y}$ be the level line of $h$ with height $u\in\R$ starting from $x$ targeted at $y$. For any $(L^z_u, u\in\R)$-stopping time $T$, assume that the loop $L^z_T$ does not hit $\partial D$ and $L^z_T$ has boundary value $c$ to the left-side and $2\lambda+c$ to the right-side. On the event $\{\gamma_u^{x\to y}\text{ hits }L^z_T\}$, we have that
\[\begin{cases}
2\lambda+c+u\in (-\lambda,\lambda),&\text{if }L^z_T\text{ is counterclockwise},\\
c+u\in (-\lambda,\lambda),&\text{if }L^z_T\text{ is clockwise}.
\end{cases}\]
Moreover, the level line $\gamma_u^{x\to y}$ stays outside of $L^z_T$.
\end{proposition}

The following proposition describes the interaction between two alternating continuum exploration processes with distinct starting points.

\begin{figure}[ht!]
\begin{subfigure}[b]{\textwidth}
\begin{center}
\includegraphics[width=0.6\textwidth]{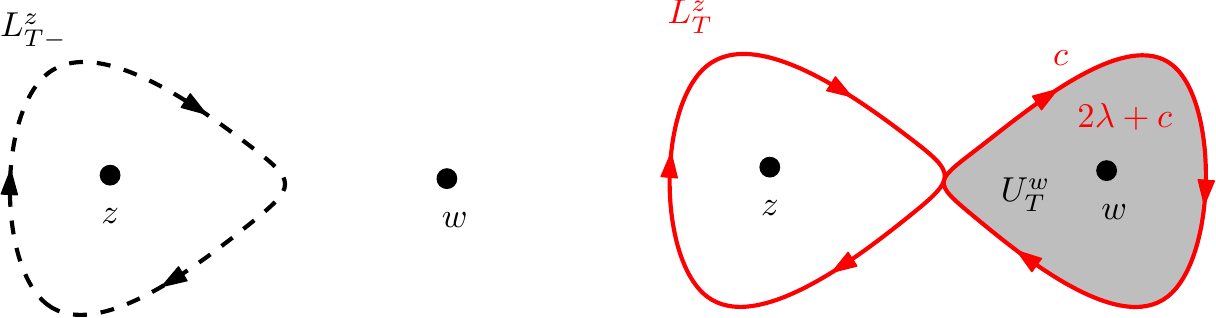}
\end{center}
\caption{Case 1. The loop $L^z_{T-}$ does not disconnect $w$ from $\infty$ and $L^z_{T-}$ is clockwise. The conditional law of $h$ restricted to $U^w_T$ is the same as $\GFF$ with boundary value $2\lambda+c$. }
\end{subfigure}
\begin{subfigure}[b]{\textwidth}
\begin{center}\includegraphics[width=0.6\textwidth]{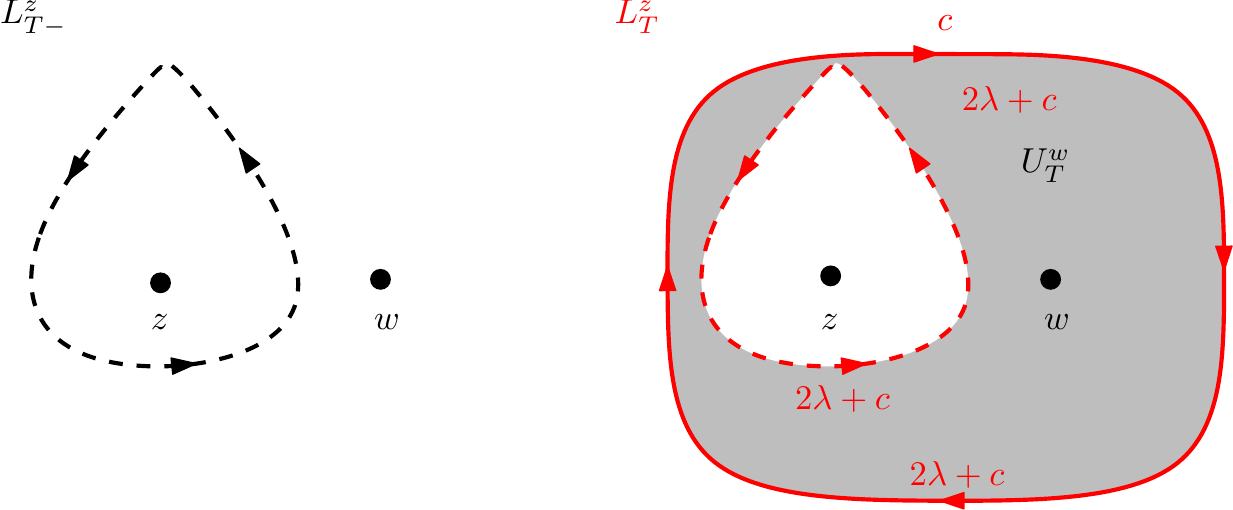}
\end{center}
\caption{Case 2. The loop $L^z_{T-}$ does not disconnect $w$ from $\infty$ and $L^z_{T-}$ is counterclockwise. The conditional law of $h$ restricted to $U^w_T$ is the same as $\GFF$ with boundary value $2\lambda+c$.}
\end{subfigure}
\begin{subfigure}[b]{\textwidth}
\begin{center}\includegraphics[width=0.6\textwidth]{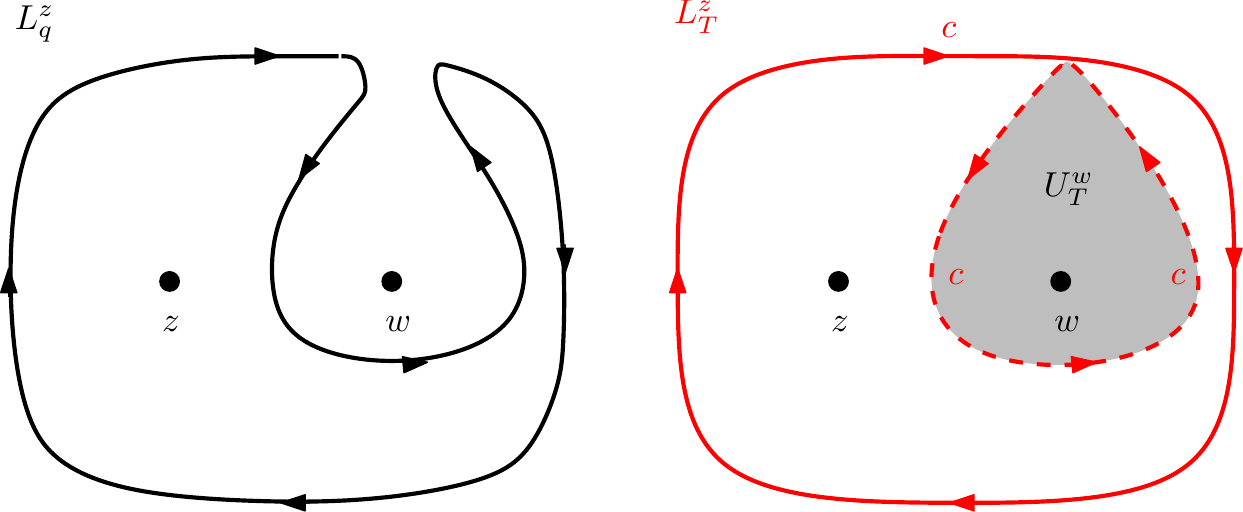}
\end{center}
\caption{Case 3. The loop $L^z_{T-}$ disconnects $w$ from $\infty$. The conditional law of $h$ restricted to $U^w_T$ is the same as $\GFF$ with boundary value $c$.}
\end{subfigure}
\caption{\label{fig::gff_continuum_merging}  Consider two alternating continuum exploration processes with distinct starting points: $(L^z_u, u\in\R)$ and $(L^w_u, u\in\R)$. Suppose that $T$ is the first time $u$ that $L^z_u$ disconnects $w$ from $\infty$. Assume that $L^z_T$ has boundary value $c$ to the left-side and $2\lambda+c$ to the right-side. Assume that $L^z_T$ is clockwise. There are three possibilities for the relation between $L^z_T$ and $L^z_{T-}$. }
\end{figure}

\begin{proposition}\label{prop::continuum_alternate_merging}
Let $h$ be a whole-plane $\GFF$ modulo a global additive constant in $\R$. Fix two starting points $z$ and $w$. Let $(L_u^z, u\in\R)$ (resp. $(L^w_u, u\in \R)$) be the alternating continuum exploration process of $h$ starting from $z$ (resp. starting from $w$) targeted at $\infty$.  Define stopping times
\[T=\inf\{u: L_u^z\text{ disconnects }w\text{ from }\infty\};\quad S=\inf\{u: L_u^w\text{ disconnects }z\text{ from }\infty\}.\]
Then, almost surely, the quasisimple loops $L^z_T$ and $L_S^w$ coincide.
\end{proposition}
\begin{proof}
We will argue that $L^w_S$ is contained in the closure of $\ext(L^z_T)$. If this is true, then by symmetry, the loop $L^z_T$ is contained in the closure of $\ext(L^w_S)$. Combining these two facts, the two loops $L^z_T$ and $L^w_S$ have to be equal.

We may assume that $L^z_T$ is clockwise and the case that $L^z_T$ is counterclockwise can be proved similarly. We assume that $L^z_T$ has boundary value $c$ to the left-side and $2\lambda+c$ to the right-side. There are three possibilities for the relation between $L^z_T$ and $L^z_{T-}$: 
\begin{itemize}
\item Case 1. The loop $L^z_{T-}$ does not disconnect $w$ from $\infty$ and $L^z_{T-}$ is clockwise (Figure \ref{fig::gff_continuum_merging}(a)).
\item Case 2. The loop $L^z_{T-}$ does not disconnect $w$ from $\infty$ and $L^z_{T-}$ is counterclockwise (Figure \ref{fig::gff_continuum_merging}(b)).
\item Case 3. The loop $L^z_{T-}$ disconnects $w$ from $\infty$ (Figure \ref{fig::gff_continuum_merging}(c)).
\end{itemize}

Let $U^w_T$ be the connected component of $\C\setminus\cup_{u\le T}L_u^z$ that contains $w$. From the coupling between $h$ and the sequence $(L^z_u, u\in\R)$, we have the following.
\begin{enumerate}
\item [(a)] Given $(L^z_u, u\le T)$, the conditional law of $h$ restricted to $\ext(L^z_T)$ is the same as $\GFF$ with boundary value $c$.
\item [(b)] Given $(L^z_u, u\le T)$, the conditional law of $h$ restricted to $U^w_T$ is the same as $\GFF$ with boundary value
\[\begin{cases}
c,&\text{if }\partial U^w_T\text{ is counterclockwise};\\
2\lambda+c,&\text{if }\partial U^w_T\text{ is clockwise}.
\end{cases}\]
\end{enumerate}  

Fix some $p\in\R$ so that $L^w_p$ is contained in $U^w_T$. Consider the field restricted to $\ext(L^w_p)$. Given $(L^w_u, u\le p)$, the conditional law of $h$ restricted to $\ext(L^w_p)$ is the same as $\GFF$ and the process $(L^w_u, u\ge p)$ is the alternating continuum exploration process of $h|_{\ext(L^w_p)}$ targeted at $\infty$. For each $k\ge 1$, let $(L^{w,k}_n,n\ge 0)$ be the alternating height-varying sequence of level loops of $h|_{\ext(L^w_p)}$ with height difference $2^{-k}\lambda$ starting from $L^w_p$ targeted at $\infty$. Let $M^k$ be the first $n$ such that $L^{w,k}_n$ disconnects $z$ from $\infty$. We know that the loop $L^{w,k}_{M^k}$ almost surely converges to $L^w_S$ in Carath\'eodory topology seen from $\infty$ as $k$ goes to $\infty$. Therefore, in order to show that $L^w_S$ is contained in the closure of $\ext(L^z_T)$, it is sufficient to show that, for each $k$ large enough, the loop $L^{w,k}_{M^k}$ is contained in the closure of $\ext(L^z_T)$. We prove this case by case. 
\medbreak
\textit{Case 1.} Consider the sequence $(L^{w,k}_n, n\ge 0)$. Let $N$ be the first $n$ that $L^{w,k}_n$ exits $\overline{U^w_T}$. Assume that the level loop $L^{w,k}_N$ has height $u$. We have the following observations.
\begin{enumerate}
\item [(a)] Since the loops $L^{w,k}_N$ and $L^{w,k}_{N-1}$ have nonempty intersection, the loop $L^{w,k}_N$ has nonempty intersection with $\partial U^w_T$.
\item [(b)] The pieces of $L^{w,k}_N$ that are outside of $U^w_T$ are level lines of $h|_{\ext(L^z_T)}$ with height $u$, thus $c+u\in (-\lambda, \lambda)$ by Lemma \ref{lem::levellinesI_remark2515} Item (1).
\item [(c)] For each $v$ such that $c+v\in (-\lambda, \lambda)$, let $\tilde{L}_v$ be the level loop of $h|_{\ext(L^z_T)}$ with height $v$ targeted at $\infty$. Clearly, $\tilde{L}_v$ disconnects both $z$ and $w$ from $\infty$.
\end{enumerate}
Combining these three facts with Lemma \ref{lem::levellinesI_levelloop}, we know that $L^{w,k}_N$ coincides with $\tilde{L}_u$. Consequently, the loop $L^{w,k}_N$ is the first loop in $(L^{w,k}_n, n\ge 0)$ that disconnects $z$ from $\infty$ and it is contained in  the closure of $\ext(L^z_T)$. 
\medbreak
\textit{Case 2.} Consider the sequence $(L^{w,k}_n, n\ge 0)$. Let $N$ be the first $n$ that $L^{w,k}_n$ exits $\overline{U^w_T}$. Assume that the level loop $L^{w,k}_N$ has height $u$. If $L^{w,k}_N$ has non-trivial pieces in $\ext(L^z_T)$, then the same argument for Case 1 would show that $L^{w,k}_N$ is the first loop in $(L^{w,k}_n, n\ge 0)$ that disconnects $z$ from $\infty$ and it is contained in  the closure of $\ext(L^z_T)$. In fact, it is impossible that $L^{w,k}_N\cap \ext(L^z_T)=\emptyset$. We will prove it by contradiction. Assume that $L^{w,k}_N\cap \ext(L^z_T)=\emptyset$. We have the following observations.
\begin{enumerate}
\item [(a)] The loop $L^{w,k}_N$ can hit $\partial U^w_T$ from the inside, this implies that $2\lambda+c+u\in (-\lambda,\lambda)$ by Lemma \ref{lem::levellinesI_remark2515} Item (3).
\item [(b)] Consider the field $h$ restricted to $U^w_T$, it has boundary value $2\lambda+c$. For any $v$ such that $2\lambda+c+v\in (-\lambda,\lambda)$, let $\tilde{L}_v$ be the level loop with height $v$. Clearly, it is contained in $\overline{U^w_T}$.
\end{enumerate}
Combining these two facts, the loop $L^{w,k}_N$ has to coincide with $\tilde{L}_u$. This contradicts with the fact that $L^{w,k}_N$ exits $\overline{U^w_T}$.
\medbreak
\textit{Case 3.} Let $q$ be any $(L^z_u, u\in\R)$-stopping time such that $q<T$. We have the following observations.
\begin{enumerate}
\item [(a)] Recall that $(L^{w,k}_n,n\ge 0)$ is the alternating height-varying sequence of level loops of $h|_{\ext(L^w_p)}$ and that $M^k$ is the first $n$ such that $L^{w,k}_n$ disconnects $z$ from $\infty$.
\item [(b)] the sequence $(L^z_u, u\le q)$ is coupled with $h|_{\ext(L^w_p)}$ in the same way as in Theorem \ref{thm::interior_continuum_coupling} stopped at $q$.
\end{enumerate}
Combining these two facts with Proposition \ref{prop::interaction_continuum_boundary}, we have that each loop in $(L^{w,k}_n, n\ge 0)$ stays outside of $L^z_q$. In particular, the loop $L^{w,k}_{M^k}$ is contained in the closure of $\ext(L^z_q)$. 

Consider the field $h$ restricted to $\ext(L^z_q)$, for $m>>k$, let $(L^{z,m}_n,n\ge 0)$ be the alternating height-varying sequence of level loops of $h|_{\ext(L^z_q)}$ with height difference $2^{-m}\lambda$ starting from $L^z_q$ targeted at $\infty$. Let $N^m$ be the first $n$ that $L^{z,m}_n$ disconnects $w$ from $\infty$. We can see that, for fixed $k$ and for $m$ large enough, the loop $L^{w,k}_{M^k}$ is contained in the closure of $\ext(L^{z,m}_{N^m})$. As $m$ goes to $\infty$, this implies that the loop $L^{w,k}_{M^k}$ is contained in the closure of $\ext(L^z_T)$. This completes the proof.
\end{proof}

\begin{proof}
[Proof of Theorem \ref{thm::interior_continuum_interacting_commontarget}] This can be proved similarly as the proof of Theorem \ref{thm::interior_levelloops_interacting_commontarget} where we replace Proposition \ref{prop::interaction_distinctstart} by Proposition \ref{prop::continuum_alternate_merging} and replace Proposition \ref{prop::levelloops_inside_deterministic} by Proposition \ref{prop::gff_upward_continuum_determinisitc}.
\end{proof}

\begin{proof}
[Proof of Theorem \ref{thm::interior_continuum_deterministic}]
This can be proved similarly as the proof of Theorem \ref{thm::interior_levelloops_deterministic} where we replace Theorem \ref{thm::interior_levelloops_interacting_commontarget} by Theorem \ref{thm::interior_continuum_interacting_commontarget}.
\end{proof}

\begin{proof}
[Proof of Theorem \ref{thm::interior_continuum_interacting_commonstart}]
This can be proved similarly as the proof of Theorem \ref{thm::interior_levelloops_interacting_commonstart} where we replace Proposition \ref{prop::levellinesI_targetindependence} by Proposition \ref{prop::continuum_alternate_targetindependence}.
\end{proof}

\begin{proof}
[Proof of Theorem \ref{thm::interior_continuum_determins_field}]
This is can be proved similarly as the proof of Theorem \ref{thm::interior_levelloops_determins_field}.
\end{proof}

\subsection{Reversibility}
\label{subsec::continuum_reversibility}
In this section, we will complete the proof of Theorem \ref{thm::wholeplane_continuum_reversibility} and derive some consequences.

\begin{figure}[ht!]
\begin{subfigure}[b]{\textwidth}
\begin{center}
\includegraphics[width=0.57\textwidth]{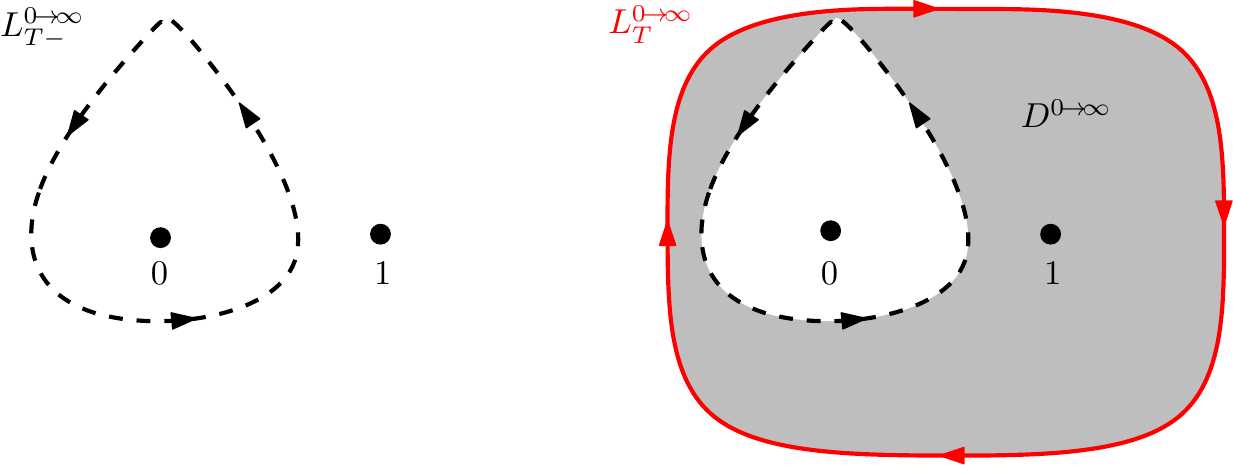}
\end{center}
\caption{Consider the alternating continuum exploration process starting from the origin targeted at $\infty$. The event $E^{0\to\infty}$ requires that $L_{T-}^{0\to\infty}$ does not disconnect $1$ from $\infty$ and that $T$ is a transition time. }
\end{subfigure}
\begin{subfigure}[b]{\textwidth}
\begin{center}\includegraphics[width=0.57\textwidth]{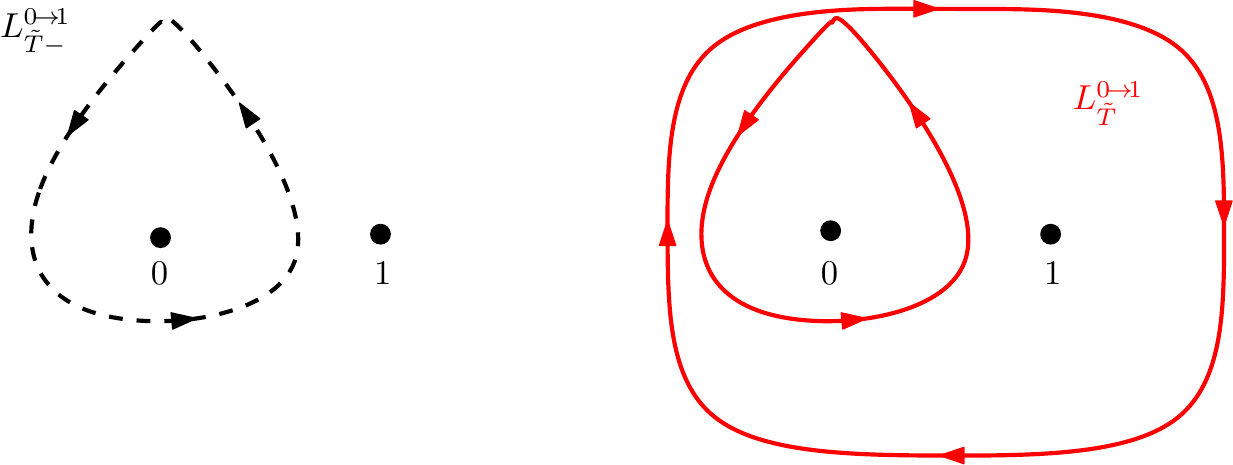}
\end{center}
\caption{Consider the alternating continuum exploration process starting from the origin targeted at $1$. On the event $E^{0\to\infty}$, we have that $L^{0\to 1}_{\tilde{T}-}$ does not disconnect $1$ from $\infty$ and that $\tilde{T}$ is not a transition time which is $E^{0\to 1}$.}
\end{subfigure}
\begin{subfigure}[b]{\textwidth}
\begin{center}\includegraphics[width=0.57\textwidth]{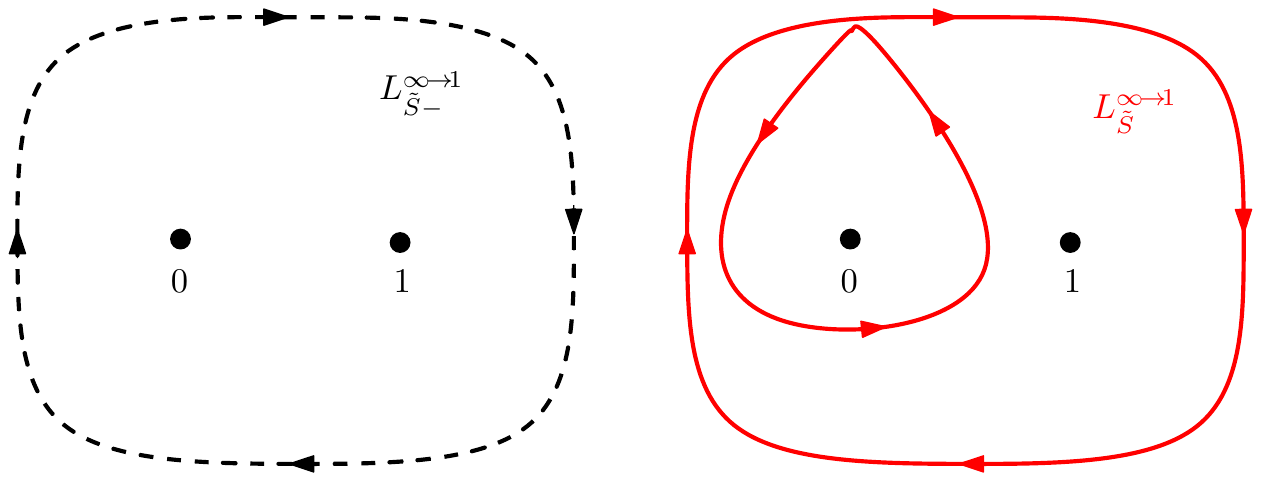}
\end{center}
\caption{Consider the alternating continuum exploration process starting from $\infty$ targeted at $1$. On the event $E^{0\to1}$, we have that $L^{\infty\to 1}_{\tilde{S}-}$ can not disconnect $0$ from $1$ and that $\tilde{S}$ is not a transition time which is $E^{\infty\to 1}$.}
\end{subfigure}
\begin{subfigure}[b]{\textwidth}
\begin{center}\includegraphics[width=0.57\textwidth]{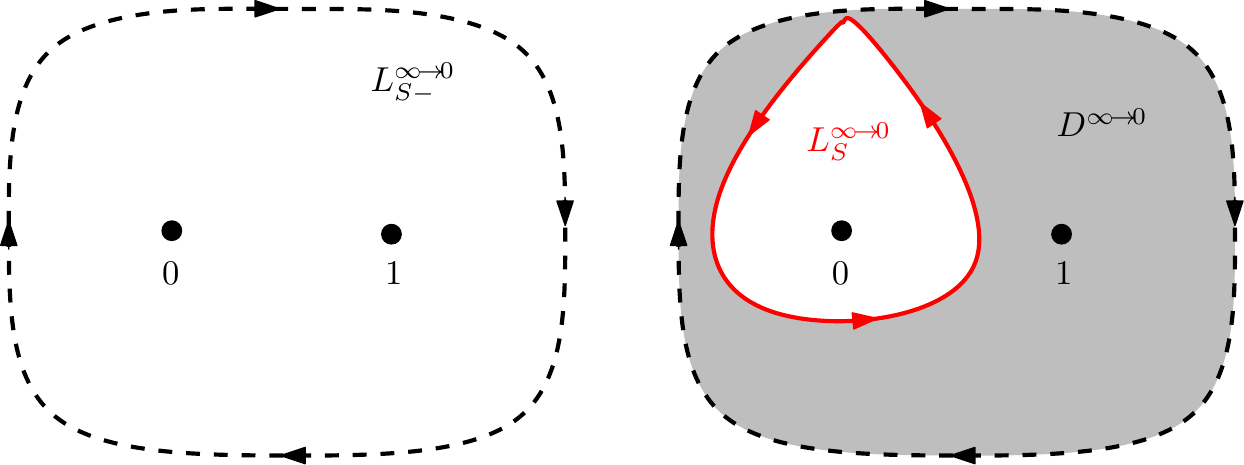}
\end{center}
\caption{Consider the alternating continuum exploration process starting from $\infty$ targeted at the origin. On the event $E^{\infty\to1}$, we have that $L^{\infty\to 0}_{S-}$ can not disconnect $0$ from $1$ and that $S$ is a transition time which is $E^{\infty\to 0}$.}
\end{subfigure}
\caption{\label{fig::gff_continuum_reversibility} Explanation of the behaviors of continuum exploration processes in Lemma \ref{lem::gff_continuum_reversibility}.}
\end{figure}

\begin{lemma}\label{lem::gff_continuum_reversibility}
Suppose that $h$ is a whole-plane $\GFF$ modulo a global additive constant in $\R$. Let $(L^{0\to\infty}_u, u\in\R)$ be the alternating continuum exploration process of $h$ starting from the origin targeted at $\infty$. Define $T$ to be the first time $u$ that $L_u^{0\to\infty}$ disconnects $1$ from $\infty$. Define $E^{0\to\infty}$ to be the event that the loop $L_{T-}^{0\to\infty}$ does not disconnect $1$ from $\infty$ and that $T$ is a transition time (see Figure \ref{fig::gff_continuum_reversibility}(a)). On the event $E^{0\to\infty}$, define $D^{0\to\infty}$ to be the domain between the loop $L_{T-}^{0\to\infty}$ and the loop $L_T^{0\to\infty}$:
\[D^{0\to\infty}=\overline{\ext(L^{0\to\infty}_{T-})}\setminus\ext(L^{0\to\infty}_T).\]
Similarly, we can define the event $E^{\infty\to 0}$ and the domain $D^{\infty\to 0}$ for the alternating continuum exploration process of $h$ starting from $\infty$ targeted at the origin. Then, almost surely on the event $E^{0\to\infty}$, the event $E^{\infty\to 0}$ also holds and the regions $D^{0\to\infty}$ and $D^{\infty\to 0}$ coincide.
\end{lemma}

\begin{proof}
For $z,w\in\C$, we denote by $(L^{z\to w}_u, u\in\R)$ the alternating continuum exploration process of $h$ starting from $z$ targeted at $w$. Define
\[T=\inf\{u: L^{0\to\infty}_u\text{ disconnects }1\text{ from }\infty\},\]
\[E^{0\to\infty}=\{L^{0\to\infty}_{T-}\text{ does not disconnect }1\text{ from }\infty\text{ and }T\text{ is a transition time}\};\]
\[\tilde{T}=\inf\{u: L^{0\to1}_u\text{ disconnects }1\text{ from }\infty\},\]
\[E^{0\to1}=\{L^{0\to1}_{\tilde{T}-}\text{ does not disconnect }1\text{ from }\infty\text{ and }\tilde{T}\text{ is not a transition time}\};\]
\[\tilde{S}=\inf\{u: L^{\infty\to1}_u\text{ disconnects }0\text{ from }1\},\]
\[E^{\infty\to1}=\{L^{\infty\to1}_{\tilde{S}-}\text{ does not disconnect }0\text{ from }1\text{ and }\tilde{S}\text{ is not a transition time}\};\]
\[S=\inf\{u: L^{\infty\to0}_u\text{ disconnects }0\text{ from }1\},\]
\[E^{\infty\to0}=\{L^{\infty\to0}_{S-}\text{ does not disconnect }0\text{ from }1\text{ and }S\text{ is a transition time}\}.\]
We have the following observations.
\begin{enumerate}
\item [(a)] Consider the sequences $(L_u^{0\to\infty}, u\in\R)$ and $(L^{0\to 1}_u, u\in\R)$. By Theorem \ref{thm::interior_continuum_interacting_commonstart}, we know that $T=\tilde{T}$ and $L_u^{0\to\infty}=L_u^{0\to 1}$ for all $u<T=\tilde{T}$. On the event $E^{0\to\infty}$, we have that $L^{0\to\infty}_{T-}=L^{0\to 1}_{\tilde{T}-}$ and that $E^{0\to 1}$ holds; moreover, the connected component of $\C\setminus L^{0\to 1}_{\tilde{T}}$ that contains $1$, denoted by $U^{0\to 1}_{\tilde{T}}$, coincides with the interior of $D^{0\to\infty}$. See Figure \ref{fig::gff_continuum_reversibility}(a)(b).
\item [(b)] Consider the sequences $(L_u^{0\to1}, u\in\R)$ and $(L^{\infty\to 1}_u, u\in\R)$. By Theorem \ref{thm::interior_continuum_interacting_commontarget}, we know that $L^{0\to 1}_{\tilde{T}}$ and $L^{\infty\to 1}_{\tilde{S}}$ coincide. In particular, on the event $E^{0\to 1}$, the event $E^{\infty\to 1}$ holds and the connected component of $\C\setminus L^{\infty\to 1}_{\tilde{S}}$ that contains $1$, denoted by $U^{\infty\to 1}_{\tilde{S}}$, coincides with $U^{0\to 1}_{\tilde{T}}$.
See Figure \ref{fig::gff_continuum_reversibility}(b)(c).
\item [(c)] Consider the sequences $(L_u^{\infty\to 1}, u\in\R)$ and $(L^{\infty\to 0}_u, u\in\R)$. By Theorem \ref{thm::interior_continuum_interacting_commonstart}, we know that $S=\tilde{S}$ and $L_u^{\infty\to 1}=L_u^{\infty\to 0}$ for all $u<S=\tilde{S}$. On the event $E^{\infty\to 1}$, we have that $L^{\infty\to 1}_{\tilde{S}-}=L^{\infty\to 0}_{S-}$ and that $E^{\infty\to 0}$ holds; moreover, the domain $U^{\infty\to 1}_{\tilde{S}}$ coincides with the interior of $D^{\infty\to 0}$. See Figure \ref{fig::gff_continuum_reversibility}(c)(d).
\end{enumerate}
Combining these three facts, we have that, almost surely on $E^{0\to\infty}$, the event $E^{\infty\to 0}$ holds and that the regions $D^{0\to\infty}$ and $D^{\infty\to 0}$ coincide.
\end{proof}

\begin{remark}\label{rem::gff_continuum_transitionregion_complement}
Assume the same notations as in Lemma \ref{lem::gff_continuum_reversibility}, we have the following consequences.
\begin{enumerate}
\item [(a)] Both the loops $L_T^{0\to\infty}$ and $L_{T-}^{0\to\infty}$ are simple loops.
\item [(b)] Consider the compact set $D^{0\to\infty}$, its interior is simply connected and its complement in the Riemann sphere has two connected components: one contains the origin and the other one contains $\infty$.
\end{enumerate}
\end{remark}

\begin{proof}
[Proof of Theorem \ref{thm::wholeplane_continuum_reversibility}]
We only need to show that the two unions of transition regions coincide, and this implies that the two unions of stationary regions coincide. 

For any $z\in\C$, suppose that $z$ is contained in the interior of $\cup_j D_j^{0\to\infty}$. Since $(D_j^{0\to\infty}, j\in\Z)$ is a sequence of compact sets with disjoint interiors, there is exactly one $D_{j_0}^{0\to\infty}$ that contains $z$. By Lemma \ref{lem::gff_continuum_reversibility}, there exists $D^{\infty\to 0}_{k_0}$ that contains $z$; moreover, the region $D_{k_0}^{\infty\to 0}$ coincides with $D_{j_0}^{0\to\infty}$. This implies that the union $\cup_j D_j^{0\to\infty}$ is contained in the union $\cup_j D_j^{\infty\to 0}$. Combining with the symmetry, the two unions coincide.
\end{proof}

Suppose that $h$ is a whole-plane $\GFF$ modulo a global additive constant in $\R$. Let $(L_u^{0\to\infty}, u\in\R)$ be the alternating continuum exploration process of $h$ starting from the origin targeted at $\infty$ where $(u_j, j\in\Z)$ is the sequence of transition times. 
Let $(D_j^{0\to\infty}, j\in\Z)$ (resp. $(F_j^{0\to\infty}, j\in\Z)$) be the sequence of transition regions (resp. the sequence of stationary regions) of $h$ starting from the origin targeted at $\infty$.
From Remark \ref{rem::gff_continuum_transitionregion_complement}, we have the following observations.
\begin{enumerate}
\item [(a)] For each transition time $u_j$, the loops $L_{u_j}^{0\to\infty}$ and $L_{u_j-}^{0\to\infty}$ are simple loops.
\item [(b)] For each $j$, the interior of $D_j^{0\to\infty}$ is simply connected, and the complement of $D_j^{0\to\infty}$ in the Riemann sphere has two connected components: one contains the origin and the other one contains $\infty$.
\item [(c)] For each $j$, the complement of $F_j^{0\to\infty}$ in the Riemann sphere has two connected components: one contains the origin and the other one contains $\infty$. (The interior of $F_j^{0\to\infty}$  is an annulus when the loops $L_{u_j}^{0\to\infty}$ and $L_{u_{j+1}-}^{0\to\infty}$ are disjoint, or is not connected when the loops $L_{u_j}^{0\to\infty}$ and $L_{u_{j+1}-}^{0\to\infty}$ have non-empty intersection.)
\item [(d)] The complement of the union $\cup_j D_j^{0\to\infty}$ is the union of interiors of $F_j^{0\to\infty}$. 
\end{enumerate}
\bibliographystyle{alpha}
\bibliography{bibliography}
\smallbreak
\noindent Menglu Wang\\
\noindent Department of Mathematics, Massachusetts Institute of Technology, Cambridge, MA, USA\\
\noindent mengluw@math.mit.edu

\noindent Hao Wu\\
\noindent NCCR/SwissMAP, Section de Math\'{e}matiques, Universit\'{e} de Gen\`{e}ve, Switzerland\\
\noindent\textit{and}
Yau Mathematical Sciences Center, Tsinghua University, China\\
\noindent hao.wu.proba@gmail.com

\end{document}